\documentclass[11pt,oneside]{amsart}

\usepackage{fullpage}

\usepackage{amsthm,amsfonts,amssymb,amsmath,amsxtra}
\usepackage[all]{xy}
\SelectTips{cm}{}
\usepackage{tikz}
\usepackage{verbatim}
\usepackage{mathrsfs}

\RequirePackage{xspace}
\RequirePackage{etoolbox}
\RequirePackage{varwidth}
\RequirePackage{enumitem}
\RequirePackage{tensor}
\RequirePackage{mathtools}
\RequirePackage{longtable}
\RequirePackage{multirow}
\RequirePackage{makecell}
\RequirePackage{bigints}

\usepackage[backref=page]{hyperref} %

\newcommand{\sD}{\ensuremath{\mathscr{D}}\xspace}

\newcommand{\sH}{\ensuremath{\mathscr{H}}\xspace}

\newcommand{\sS}{\ensuremath{\mathscr{S}}\xspace}

\newcommand{\sV}{\ensuremath{\mathscr{V}}\xspace}

\newcommand{\fkm}{\ensuremath{\mathfrak{m}}\xspace}

\newcommand{\BC}{\ensuremath{\mathbb{C}}\xspace}

\newcommand{\BL}{\ensuremath{\mathbb{L}}\xspace}

\newcommand{\BV}{\ensuremath{\mathbb{V}}\xspace}

\newcommand{\BZ}{\ensuremath{\mathbb{Z}}\xspace}

\newcommand{\CN}{\ensuremath{\mathcal{N}}\xspace}
\newcommand{\CO}{\ensuremath{\mathcal{O}}\xspace}

\newcommand{\CV}{\ensuremath{\mathcal{V}}\xspace}

\newcommand{\CZ}{\ensuremath{\mathcal{Z}}\xspace}

\newcommand{\RZ}{\ensuremath{\mathrm{Z}}\xspace}

\DeclareMathOperator{\Aut}{Aut}

\newcommand{\Ch}{{\mathrm{Ch}}}

\newcommand{\cl}{{\mathrm{cl}}}

\newcommand{\disc}{{\mathrm{disc}}}

\DeclareMathOperator{\End}{End}

\newcommand{\Fil}{\ensuremath{\mathrm{Fil}}\xspace}

\newcommand{\Fr}{\mathbf{F}}

\DeclareMathOperator{\Gal}{Gal}

\newcommand{\GL}{\mathrm{GL}}

\DeclareMathOperator{\Gr}{Gr}

\newcommand{\GSpin}{\mathrm{GSpin}}

\DeclareMathOperator{\Hom}{Hom}

\newcommand{\id}{\ensuremath{\mathrm{id}}\xspace}

\DeclareMathOperator{\im}{im}

\newcommand{\Ind}{{\mathrm{Ind}}}

\DeclareMathOperator{\Int}{\ensuremath{\mathrm{Int}}\xspace}

\DeclareMathOperator{\length}{length}
\DeclareMathOperator{\Lie}{Lie}

\newcommand{\Mp}{{\mathrm{Mp}}}
\newcommand{\mult}{{\mathrm{mult}}}

\newcommand{\OGr}{\mathrm{OGr}}
\DeclareMathOperator{\Orb}{Orb}

\newcommand{\PGL}{{\mathrm{PGL}}}

\renewcommand{\Re}{{\mathrm{Re}}}
\newcommand{\red}{\ensuremath{\mathrm{red}}\xspace}

\DeclareMathOperator{\Res}{Res}

\newcommand{\Sh}{\mathrm{Sh}}

\newcommand{\SL}{{\mathrm{SL}}}

\DeclareMathOperator{\Spec}{Spec}
\DeclareMathOperator{\Spf}{Spf}

\newcommand{\SO}{{\mathrm{SO}}}
\newcommand{\Sp}{{\mathrm{Sp}}}
\newcommand{\GSp}{{\mathrm{GSp}}}

\newcommand{\val}{{\mathrm{val}}}

\newcommand{\Ver}{{\mathrm{Vert}}}

\DeclareMathOperator{\supp}{supp}
\DeclareMathOperator{\Sym}{Sym}
\DeclareMathOperator{\sgn}{sgn}

\DeclareMathOperator{\tr}{tr}

\renewcommand{\O}{\mathrm{O}}

\DeclareMathOperator{\vol}{vol}

\newcommand{\jiao}{\stackrel{\BL}{\cap}}

\newcommand{\wit}{\widetilde}
\newcommand{\wh}{\widehat}

\newcommand{\pair}[1]{\langle {#1} \rangle}

\newcommand{\ov}{\overline}

\newcommand{\lra}{\longrightarrow}

\newenvironment{altenumerate}
   {\begin{list}
      {(\theenumi) }
      {\usecounter{enumi}
       \setlength{\labelwidth}{0pt}
       \setlength{\labelsep}{0pt}
       \setlength{\leftmargin}{0pt}
       \setlength{\itemsep}{\the\smallskipamount}
       \renewcommand{\theenumi}{\roman{enumi}}
      }}
   {\end{list}}

\newenvironment{altitemize}
   {\begin{list}
      {$\bullet$}
      {\setlength{\labelwidth}{0pt}
	   \setlength{\itemindent}{5pt}
       \setlength{\labelsep}{5pt}
       \setlength{\leftmargin}{0pt}
       \setlength{\itemsep}{\the\smallskipamount}
      }}
   {\end{list}}

\setitemize[0]{leftmargin=*,itemsep=\the\smallskipamount}
\setenumerate[0]{leftmargin=*,itemsep=\the\smallskipamount}

\renewcommand{\to}{%
   \ifbool{@display}{\longrightarrow}{\rightarrow}%
   }
\let\shortmapsto\mapsto
\renewcommand{\mapsto}{%
   \ifbool{@display}{\longmapsto}{\shortmapsto}%
   }
\newlength{\olen}
\newlength{\ulen}
\newlength{\xlen}
\newcommand{\xra}[2][]{%
   \ifbool{@display}%
      {\settowidth{\olen}{$\overset{#2}{\longrightarrow}$}%
       \settowidth{\ulen}{$\underset{#1}{\longrightarrow}$}%
       \settowidth{\xlen}{$\xrightarrow[#1]{#2}$}%
       \ifdimgreater{\olen}{\xlen}%
          {\underset{#1}{\overset{#2}{\longrightarrow}}}%
          {\ifdimgreater{\ulen}{\xlen}%
             {\underset{#1}{\overset{#2}{\longrightarrow}}}
             {\xrightarrow[#1]{#2}}}}%
      {\xrightarrow[#1]{#2}}
   }
\makeatother
\newcommand{\xyra}[2][]{%
   \settowidth{\xlen}{$\xrightarrow[#1]{#2}$}%
   \ifbool{@display}%
      {\settowidth{\olen}{$\overset{#2}{\longrightarrow}$}%
       \settowidth{\ulen}{$\underset{#1}{\longrightarrow}$}%
       \ifdimgreater{\olen}{\xlen}%
          {\mathrel{\xymatrix@M=.12ex@C=3.2ex{\ar[r]^-{#2}_-{#1} &}}}%
          {\ifdimgreater{\ulen}{\xlen}%
             {\mathrel{\xymatrix@M=.12ex@C=3.2ex{\ar[r]^-{#2}_-{#1} &}}}
             {\mathrel{\xymatrix@M=.12ex@C=\the\xlen{\ar[r]^-{#2}_-{#1} &}}}}}%
      {\mathrel{\xymatrix@M=.12ex@C=\the\xlen{\ar[r]^-{#2}_-{#1} &}}}%
   }
\makeatletter
\newcommand{\xla}[2][]{%
   \ifbool{@display}%
      {\settowidth{\olen}{$\overset{#2}{\longleftarrow}$}%
       \settowidth{\ulen}{$\underset{#1}{\longleftarrow}$}%
       \settowidth{\xlen}{$\xleftarrow[#1]{#2}$}%
       \ifdimgreater{\olen}{\xlen}%
          {\underset{#1}{\overset{#2}{\longleftarrow}}}%
          {\ifdimgreater{\ulen}{\xlen}%
             {\underset{#1}{\overset{#2}{\longleftarrow}}}
             {\xleftarrow[#1]{#2}}}}%
      {\xleftarrow[#1]{#2}}
   }
\newcommand{\isoarrow}{%
   \ifbool{@display}{\overset{\sim}{\longrightarrow}}{\xrightarrow\sim}%
   }
\renewcommand{\lra}{%
   \ifbool{@display}{\longleftrightarrow}{\leftrightarrow}%
   }

\newcommand{\Fb}{{\breve F}}
\newcommand{\Kb}{{\breve K}}
\newcommand{\OFb}{{O_{\breve F}}}
\newcommand{\OFnb}{{O_{\breve F_\nu}}}
\newcommand{\Herm}{\mathrm{Herm}}
\newcommand{\rd}{\mathrm{d}}
\newcommand{\Rep}{\mathrm{Rep}}

\newcommand{\Den}{\mathrm{Den}}
\newcommand{\Dene}{\mathrm{Den}^\varepsilon}
\newcommand{\Denf}{\Den^{\flat}}
\newcommand{\Denfe}{\Den^{\flat\varepsilon }}

\newcommand{\Nor}{\mathrm{Nor}}
\newcommand{\Nore}{\Nor^\varepsilon{}}
\newcommand{\Norfe}{\Nor^{\flat\varepsilon}}
\newcommand{\Has}{\epsilon}
\newcommand{\Hor}{\mathrm{Hor}}
\newcommand{\Hore}{\mathrm{Hor}^\varepsilon}
\newcommand{\pDen}{\partial\mathrm{Den}}
\newcommand{\pDene}{\pDen^\varepsilon}
\newcommand{\Inte}{\Int^\varepsilon}
\newcommand{\wt}{\mathfrak{m}}
\newcommand{\wte}{\wt^\varepsilon}
\newcommand{\wtf}{\wt^\flat}
\newcommand{\wtfe}{\wt^{\flat\varepsilon}}

\newcommand{\LZ}{{}^\BL\CZ}
\newcommand{\Tor}{\underline{\mathrm{Tor}}}
\newcommand{\HOM}{\underline{\mathrm{Hom}}}
\newcommand{\Intp}{\Int_{L^\flat,\sV}^\perp}
\newcommand{\pDenp}{\pDen_{L^\flat,\sV}^\perp}
\newcommand{\Ss}{\sS}
\newcommand{\Dd}{\sD}
\newcommand{\Lloc}{\mathrm{L}^1_\mathrm{loc}}
\newcommand{\VV}{W}
\newcommand{\OFp}{O_{F,(p)}}

\newcommand{\ch}{\mathrm{ch}}
\newcommand{\Tate}{\mathrm{Tate}}
\newcommand{\barK}{\overline{K}_0(\CV(\Lambda))}
\newcommand{\barCh}{\overline{\Ch}}
\newcommand{\Intch}{c_{1,{\mathcal{V}(\Lambda)}}}
\newcommand{\Intc}{c_{{\mathcal{V}(\Lambda)}}}
\newcommand{\IntK}{\mathrm{K}_{\CV(\Lambda)}}

\newcommand{\Zss}{\widehat{\mathcal{Z}}^\mathrm{ss}}
\newcommand{\Mss}{\widehat{\mathcal{M}^\mathrm{ss}}}
\newcommand{\KG}{K}

\newcommand{\pEis}{\partial\mathrm{Eis}}
\newcommand{\Diff}{\mathrm{Diff}}

\newcommand{\sz}{\mathsf{z}}
\newcommand{\sx}{\mathsf{x}}
\newcommand{\sy}{\mathsf{y}}

\DeclareFontFamily{U}{matha}{\hyphenchar\font45}
\DeclareFontShape{U}{matha}{m}{n}{
      <5> <6> <7> <8> <9> <10> gen * matha
      <10.95> matha10 <12> <14.4> <17.28> <20.74> <24.88> matha12
      }{}
\DeclareSymbolFont{matha}{U}{matha}{m}{n}
\DeclareFontFamily{U}{mathx}{\hyphenchar\font45}
\DeclareFontShape{U}{mathx}{m}{n}{
      <5> <6> <7> <8> <9> <10>
      <10.95> <12> <14.4> <17.28> <20.74> <24.88>
      mathx10
      }{}
\DeclareSymbolFont{mathx}{U}{mathx}{m}{n}

\DeclareMathSymbol{\obot}         {2}{matha}{"6B}

\newtheorem{theorem}[subsubsection]{Theorem}
\newtheorem{proposition}[subsubsection]{Proposition}
\newtheorem{lemma}[subsubsection]{Lemma}

\newtheorem{corollary}[subsubsection]{Corollary}

\theoremstyle{definition}
\newtheorem{definition}[subsubsection]{Definition}
\newtheorem{example}[subsubsection]{Example}

\newtheorem{remark}[subsubsection]{Remark}

\numberwithin{equation}{subsubsection}

\newcommand{\QHom}{\mathrm{QHom}}
\newcommand{\ansm}{\mathrm{ANilp}^\mathrm{fsm}_\OFb}
\newcommand{\alg}{\mathrm{Alg}_\OFb}
\newcommand{\kb}{{\bar\kappa}}
\newcommand{\bpi}{\boldsymbol{\pi}}

\newcommand{\etale}{\'etale }
\newcommand{\crys}{\mathrm{crys}}
\newcommand{\et}{{\acute{\mathrm{e}}\mathrm{t}}}
\newcommand{\Endo}{\mathrm{End}^\circ}
\newcommand{\SEnd}{\mathrm{SEnd}}
\newcommand{\SEndo}{\mathrm{SEnd}^\circ}

\DeclareMathOperator{\RRZ}{RZ}

\title{On the arithmetic Siegel--Weil formula for GSpin Shimura varieties}

\author[Chao Li]{Chao Li}
\address{Columbia University, Department of Mathematics, 2990 Broadway,	New York, NY 10027, USA}
\email{chaoli@math.columbia.edu} 

\author[Wei Zhang]{Wei Zhang}
\address{Massachusetts Institute of Technology, Department of Mathematics, 77 Massachusetts Avenue, Cambridge, MA 02139, USA}
\email{weizhang@mit.edu}

\date{\today}

\begin{document}

\maketitle{}

\begin{abstract}
  We formulate and prove a local arithmetic Siegel--Weil formula for GSpin Rapoport--Zink spaces, which is a precise identity between the arithmetic intersection numbers of special cycles on GSpin Rapoport--Zink spaces and the derivatives of local representation densities of quadratic forms. As a first application, we prove a semi-global arithmetic Siegel--Weil formula as conjectured by Kudla, which relates the arithmetic intersection numbers of special cycles on GSpin Shimura varieties at a place of good reduction and the central derivatives of nonsingular Fourier coefficients of incoherent Siegel Eisenstein series. 
\end{abstract}

\tableofcontents{}

\newpage

\section{Introduction}

\subsection{Background}

The classical \emph{Siegel--Weil formula} (\cite{Sie35,Siegel1951,Weil1965}) relates certain Siegel Eisenstein series to the arithmetic of quadratic forms, namely it expresses special \emph{values} of these series as theta functions --- generating series of representation numbers of quadratic forms. Kudla (\cite{Kudla1997a, Kudla2004}) initiated an influential program to establish the \emph{arithmetic Siegel--Weil formula}. In particular, the nonsingular part of Kudla's conjectural formula relates the \emph{central derivative} of nonsingular Fourier coefficients of Siegel Eisenstein series to the arithmetic intersection number of $n$ special divisors on orthogonal Shimura varieties associated to $\GSpin(n-1,2)$. The arithmetic Siegel--Weil formula was established by Kudla, Rapoport and Yang (\cite{Kudla1999, Kudla1997a, Kudla2000, Kudla2006}) for $n=1,2$. The \emph{archimedean} part of the formula for all $n$ was also known by Garcia--Sankaran \cite{Garcia2019} and Bruinier--Yang \cite{Bruinier2018}.  However, for the nonarchimedean part for higher $n$, the only known cases were when $n=3$ due to Gross--Keating \cite{Gross1993} (cf. \cite{VGW+07}) and Terstiege \cite{Terstiege2011}, and some partial results when the arithmetic intersection has dimension $0$ (\cite{Kudla1999a,Kudla2000a, Bruinier2018}).

The main result of this paper proves a semi-global (at a prime $p$) version of the arithmetic Siegel--Weil formula for all $n$. To achieve this, we formulate and prove a local arithmetic Siegel--Weil formula for GSpin Rapoport--Zink spaces, which is a precise identity between the arithmetic intersection numbers of special cycles on GSpin Rapoport--Zink spaces and the derivatives of local representation densities of quadratic forms. Such a local formula is an orthogonal analogue of the Kudla--Rapoport conjecture for unitary Rapoport--Zink spaces (\cite[Conjecture 1.3]{Kudla2011}) recently proved in our companion paper \cite{LZ2019}. Compared to the unitary case in \cite{LZ2019}, several new difficulties arise in the orthogonal case and we highlight some of them in \S\ref{sec:whats-new}. In fact, the geometric difficulty in the higher dimensional orthogonal case was one of the reasons Kudla and Rapoport shifted their perspective to the unitary case \cite{Kudla2011}. 

Via the doubling method of Piatetski-Shapiro and Rallis, the arithmetic Siegel--Weil formula is intimately tied to the \emph{arithmetic inner product formula}, which relates the central derivative of the standard $L$-function of cuspidal automorphic representations on metaplectic/orthogonal groups to the height pairing of cycles on orthogonal Shimura varieties constructed from arithmetic theta liftings. It can be viewed as a higher dimensional generalization of the Gross--Zagier formula \cite{Gross1986}, and an arithmetic analogue of the Rallis inner product formula. We hope to apply the main results in this paper to study the arithmetic inner product formula in the future (cf. \cite{LL2020, LL2021} by Liu and one of us in the unitary case). It is also worth mentioning several other recent advances in arithmetic Siegel--Weil formula in the unitary case, including cases of the singular term formula \cite{BH21} by Bruinier--Howard and the higher derivative formula over function fields  \cite{FYZ21} by Feng, Yun and one of us, and it would be interesting to study the (arguably more difficult) orthogonal analogues.

\subsection{Local arithmetic Siegel--Weil formula} Let $p$ be an odd prime. Let $F=\mathbb{Q}_p$ with residue field $\kappa=\mathbb{F}_p$ and a uniformizer $\varpi$. Let $\breve F$ be the completion of the maximal unramified extension of $F$. Let $m=n+1\geq 3$ be an integer. Let $\varepsilon\in\{\pm1\}$. Let $V=H_m^\varepsilon$ be a self-dual quadratic $O_F$-lattice of rank $m$ with $\chi(V)=\varepsilon$. Here $\chi(V)=+1$ (resp. $-1$) if the the discriminant $\disc(V)$ is a square (resp. nonsquare) in $O_F^\times$ (see Notations \S\ref{sec:notat-quadr-latt-1}). Associated to $V$ we have a \emph{local unramified Shimura-Hodge data} $(G, b, \mu, C)$, where $G=\GSpin(V)$, $b\in G(\Fb)$ is a basic element, $\mu: \mathbb{G}_m\rightarrow G$ is a certain cocharacter, and $C=C(V)$ is the Clifford algebra of $V$. Associated to this local unramified Shimura-Hodge data, we have a \emph{GSpin Rapoport--Zink space} $\RRZ_G=\RRZ(G, b, \mu, C)$ of Hodge type constructed by Howard--Pappas \cite{Howard2015} and Kim \cite{Kim2018}. The space $\RRZ_G$ is a formal scheme over $\Spf\OFb$, formally locally of finite type and formally smooth of relative dimension $n-1$ over $\Spf\OFb$ (see \S\ref{sec:gspin-rapoport-zink} for more details), and admits a decomposition $\RRZ_G=\bigsqcup_{\ell\in \mathbb{Z}}\RRZ_G^{(\ell)}$ into (isomorphic) connected components. Define $\mathcal{N}=\mathcal{N}_{n}^\varepsilon:=\RRZ_G^{(0)}$ to be a connected component of $\RRZ_G$, a formal scheme of (total) dimension $n$ (Definition \ref{def:conn-gspin-rapop-1}).

Let $\mathbb{V}=\mathbb{V}_m^\varepsilon$ be the unique (up to isomorphism) quadratic space over $F$ of dimension $m$, Hasse invariant $\Has(\mathbb{V})=-1$ and $\chi(\mathbb{V})=\varepsilon$. Then $\mathbb{V}$ can be identified as the space of \emph{special quasi-endomorphisms} $\mathbb{V}\subseteq \End(\mathbb{X}) \otimes \mathbb{Q}$, where $\mathbb{X}$ is the framing $p$-divisible group over $\kb$ for $\RRZ_G$ (\S\ref{sec:space-special-quasi}). For any subset $L\subseteq \mathbb{V}$, the \emph{special cycle} $\mathcal{Z}(L)$ (Definition \ref{def:kudla-rapop-cycl-2}) is a closed formal subscheme of $\mathcal{N}$, over which all special quasi-endomorphisms $x\in L$ deforms to endomorphisms.

Let $L\subseteq \mathbb{V}$ be an $O_F$-lattice of rank $n$ (always assumed to be non-degenerate throughout the paper, see Notations \S\ref{sec:notat-quadr-latt-1}). We now associate to $L$ two integers: the \emph{arithmetic intersection number} $\Inte(L)$ and the \emph{derived local density} $\pDene(L)$.

Let $x_1,\ldots, x_n$ be an $O_F$-basis of $L$. Define the \emph{arithmetic intersection number}
\begin{equation}
  \label{eq:intL}
  \Inte(L)\coloneqq \chi(\mathcal{N},\mathcal{O}_{\mathcal{Z}(x_1)} \otimes^\mathbb{L}\cdots \otimes^\mathbb{L}\mathcal{O}_{\mathcal{Z}(x_n)} ),
\end{equation}
 where $\mathcal{O}_{\mathcal{Z}(x_i)}$ denotes the structure sheaf of the special divisor $\mathcal{Z}(x_i)$, $\otimes^\mathbb{L}$ denotes the derived tensor product of coherent sheaves on $\mathcal{N}$, and $\chi$ denotes the Euler--Poincar\'e characteristic. It is independent of the choice of the basis $x_1,\ldots, x_n$ and hence is a well-defined invariant of $L$ itself (Definition \ref{def:Int(L)}).

 For $M$ another quadratic $O_F$-lattice (of arbitrary rank), define $\Rep_{M, L}$ to be the \emph{scheme of integral representations of $M$ by $L$}, an $O_{F}$-scheme such that for any $O_{F}$-algebra $R$, $\Rep_{M,L}(R)=\QHom(L \otimes_{O_{F}}R, M \otimes_{O_{F}}R)$, where $\QHom$ denotes the set of quadratic module homomorphisms. The \emph{local density} of integral representations of $M$ by $L$ is defined to be $$\Den(M,L)\coloneqq \lim_{N\rightarrow +\infty}\frac{\# \Rep_{M,L}(O_{F}/\varpi^N)}{p^{N\cdot\dim (\Rep_{M,L})_{F}}}.$$ Then $\Den(H_{m+2k}^\varepsilon, L)$ is a polynomial in $p^{-k}$ with $\mathbb{Q}$-coefficients.  Define the (normalized) \emph{local Siegel series} of $L$ to be the polynomial $\Dene(X,L)\in \mathbb{Z}[X]$ (Theorem \ref{thm: Den(X)}) such that for all $k\ge0$, $$\Dene(p^{-k},L)=\frac{\Den(H_{m+2k}^\varepsilon, L)}{\Den(H_{m+2k}^\varepsilon, H_n^\varepsilon)}.$$ Since $L\subseteq \mathbb{V}$ is an $O_F$-lattice of rank $n$, it satisfies a functional equation relating $X\leftrightarrow\frac{1}{X}$ (Theorem~\ref{thm:ikeda}, Lemma~\ref{lem:latticemebed} (\ref{item:embed1})),
\begin{equation*}
  \label{eq:FEmain}
  \Dene(X,L)=-X^{\val(L)}\cdot\Dene\left(\frac{1}{X},L\right).
\end{equation*} Here $\val(L)$ is the valuation of $L$, see Notations \S\ref{sec:notat-quadr-latt-1}. We thus consider the \emph{derived local density} $$\pDene(L)\coloneqq-\frac{\rd}{\rd X}\bigg|_{X=1}\Dene(X,L).$$ 

Our main theorem in Part \ref{part:local-kudla-rapoport} is the proof of a local arithmetic Siegel--Weil formula, which asserts an exact identity between the two integers just defined.
\begin{theorem}[Theorem \ref{thm: main}]\label{thm:intro1}
  Let $L\subseteq \mathbb{V}$ be an $O_F$-lattice of rank $n$. Then $$\Inte(L)=\pDene(L).$$
\end{theorem}
We refer to $\Inte(L)$ as the \emph{geometric side} of the identity (related to the geometry of Rapoport--Zink spaces and Shimura varieties) and $\pDene(L)$ the \emph{analytic side} (related to the derivatives of Eisenstein series and $L$-functions).

\subsection{Semi-global arithmetic Siegel--Weil formula} Next let us describe a semi-global application of our local theorem. We now switch to global notations. Let $F=\mathbb{Q}$ and $\mathbb{A}=\mathbb{A}_F$ its ring of ad\`eles. Let $m=n+1\ge3$ be an integer. Let $V$ be a quadratic space over $F$ of dimension $m$ and signature $(n-1,2)$. Let $G=\GSpin(V)$. Associated to $G$ there is a Shimura datum $(G, \{h_G\})$ of Hodge type. Let $K=\prod_vK_v\subseteq G(\mathbb{A}_f)$ be an open compact subgroup. Then the associated Shimura variety $\Sh_{K}=\Sh_{K}(G,\{h_{G}\})$ is of dimension $n-1$ and has a canonical model over its reflex field $F=\mathbb{Q}$.

Assume that $p$ is an odd prime such that $K_p$ is a hyperspecial subgroup of $G(F_p)$, or equivalently $K_p=\GSpin(\Lambda_p)$ for a self-dual lattice $\Lambda_p\subseteq V_p:=V \otimes_F F_p$. Then by Kisin \cite{Kisin2010}, there exists a smooth integral canonical model $\mathcal{M}_K$ of $\Sh_K$ over the localization $O_{F,(p)}$.

Let $\mathbb{V}$ be the \emph{incoherent} quadratic space over $\mathbb{A}$ of rank $m$ nearby $V$, namely $\mathbb{V}$ is positive definite and $\mathbb{V}_v \cong V_v$ for all finite places $v$. Let $\varphi_K\in \sS(\mathbb{V}^n_f)$ be a factorizable Schwartz function. We say that $\varphi_K$ is \emph{$p$-admissible} if $\varphi_K$ is $K$-invariant and $\varphi_{K,p}=\mathbf{1}_{(\Lambda_p)^n}$. Let $T\in \Sym_n(F)_{>0}$ be a positive definite symmetric matrix of size $n$. Associated to $(T,\varphi_K)$ we construct semi-global special cycles $\mathcal{Z}(T,\varphi_K)$ over $\mathcal{M}_K$ (\S\ref{sec:semi-global-kudla}). Analogous to the local situation (\ref{eq:intL}), we may define its semi-global arithmetic intersection numbers $\Int_{T,p}(\varphi_K)$ at  $p$ (\S\ref{sec:local-arithm-inters}). 

On the other hand, associated to $\varphi=\varphi_K \otimes \varphi_\infty\in\sS(\mathbb{V}^n)$, where $\varphi_{\infty}$ is the Gaussian function, there is a classical \emph{incoherent Eisenstein series} $E(\sz, s,\varphi_K)$ (\S\ref{sec:incoh-eisenst-seri-1}) on the Siegel upper half space $$\mathbb{H}_n=\{\sz=\sx+i\sy:\ \sx\in\Sym_n(F_\infty),\ \sy\in\Sym_n(F_\infty)_{>0}\},$$ where $F_\infty= F \otimes_\mathbb{Q} \mathbb{R}\simeq \mathbb{R}$. This is essentially the Siegel Eisenstein series associated to a standard Siegel--Weil section of the degenerate principal series (\S\ref{sec:sieg-eisenst-seri}). The Eisenstein series here has a meromorphic continuation and a functional equation relating $s\leftrightarrow -s$.  The central value $E(\sz, 0, \varphi_K)=0$ by the incoherence. We thus consider its \emph{central derivative} $$\pEis(\sz, \varphi_K)\coloneqq \frac{\rd}{\rd s}\bigg|_{s=0}E(\sz, s,\varphi_K).$$ Associated to the standard additive character $\psi: \mathbb{A}/F\rightarrow \mathbb{C}^\times$, it has a decomposition into the central derivative of the Fourier coefficients $$\pEis(\sz,\varphi_K)=\sum_{T\in \Herm_n(F)}\pEis_T(\sz,\varphi_K).$$ When $T$ is nonsingular, the Euler factorization of $T$-th Fourier coefficients  further gives a decomposition (\S\ref{sec:incoh-eisenst-seri}) $$\pEis_T(\sz,\varphi_K)=\sum_v \pEis_{T,v}(\sz,\varphi_K).$$

Now we can state our application to the semi-global arithmetic Siegel--Weil formula, which asserts an identity between the semi-global arithmetic intersection number of special cycles and the derivative of nonsingular Fourier coefficients of the incoherent Eisenstein series. 

\begin{theorem}[Theorem \ref{thm:semi-global-identity}]\label{thm:intro2}
Assume that $\varphi_\KG\in \sS(\mathbb{V}^n_f)$ is $p$-admissible. Then for any $T\in \Sym_n(F)_{>0}$, $$\Int_{T,p}(\varphi_{\KG})q^T=c_K\cdot \pEis_{T,p}(\sz,\varphi_{\KG}),$$ where $q^T:=\psi_\infty(\tr T\sz)=e^{2\pi i \tr T\sz}$, $c_K=\frac{(-1)^n}{\vol(K)}$ is a nonzero constant independent of $T$ and $\varphi_K$, and $\vol(K)$ is the volume of $K$ under a suitable Haar measure on $G(\mathbb{A}_f)$.
\end{theorem}

\begin{remark}
In the unitary case, we also proved a global version (including terms for all $T$ and all places $v$) of the arithmetic Siegel--Weil formula \cite[Theorem 1.3.2]{LZ2019}, at least for test functions $\varphi_K$ with nonsingular support at two split places. This global version is more difficult in the orthogonal case due to several complications, most notably the lack of the analogue of split places and the inevitability to treat the place $p=2$. We hope to return to these questions in the future. 
\end{remark}

\begin{remark}
The assumption $F=\mathbb{Q}_p$ ($p$ odd) in Theorem \ref{thm:intro1} is required to apply results from  \cite{Howard2015} and \cite{Kim2018}. It would be very interesting to relax this assumption  to more general $p$-adic fields $F$ by generalizing \cite{Howard2015} and \cite{Kim2018} to ``relative'' GSpin Rapoport--Zink spaces and proving a comparison between the relative and absolute GSpin Rapoport--Zink spaces of Weil-restricted groups $\Res_{F/\mathbb{Q}_p}\GSpin$. Once this is done, one should also be able to relax the assumption $F=\mathbb{Q}$ in Theorem \ref{thm:intro2} to more general totally real fields.
\end{remark}

\subsection{Strategy and novelty of the proof of the Main Theorem \ref{thm:intro1}}\label{sec:whats-new} Our general strategy is parallel to the unitary case proved in \cite{LZ2019} (see also several simplifications in \cite{LL2021}).   More precisely, fix an $O_F$-lattice $L^\flat\subseteq \mathbb{V}$ of rank $n-1$ and denote by $\mathbb{W}=(L^\flat_F)^\perp\subseteq \mathbb{V}$. Consider functions on $L^\flat\times \mathbb{W}^\mathrm{an}$, $$\Int_{L^\flat}(x)\coloneqq \Inte(L^\flat+\langle x\rangle),\quad \pDen_{L^\flat}(x)\coloneqq \pDene(L^\flat+\langle x\rangle),$$ where $\mathbb{W}^\mathrm{an}:=\{x \in \mathbb{W}: (x,x)\ne0\}$ is the set of anisotropic vectors. Then it remains to show the equality of the two functions $\Int_{L^\flat}=\pDen_{L^\flat}$. To show this equality, we find a decomposition $$\Int_{L^\flat}=\Int_{L^\flat,\sH}+\Int_{L^\flat,\sV},\quad \pDen_{L^\flat}=\pDen_{L^\flat,\sH}+\pDen_{L^\flat,\sV}$$ into ``horizontal'' and ``vertical'' parts such that the horizontal identity $\Int_{L^\flat,\sH}=\pDen_{L^\flat,\sH}$ holds and the vertical parts $\Int_{L^\flat,\sV}$ and $\pDen_{L^\flat,\sV}$ behaves well under Fourier transform.

In the unitary case, the hermitian space $\mathbb{W}$ is 1-dimensional over $F$ and we found that both $\Int_{L^\flat,\sV}$ and $\pDen_{L^\flat,\sV}$ are Schwartz functions on $\mathbb{V}$, and satisfy the remarkable properties that \begin{equation}\label{eq:FCmain}
\widehat{\Int_{L^\flat,\sV}}=- \Int_{L^\flat,\sV},\quad \supp(\widehat{\pDen_{L^\flat,\sV}})\subseteq \mathbb{V}^\circ,
\end{equation} where $\wh{\ }$ denotes the Fourier transform on $\mathbb{V}$, and $\mathbb{V}^\circ:=\{x\in \mathbb{V}: (x,x)\in O_F\}$ is the integral cone.  Equation (\ref{eq:FCmain}), together with the induction on the valuation of $L^\flat$ and the uncertainty principle, allows us to conclude that $\Int_{L^\flat,\sV}=\pDen_{L^\flat,\sV}$.

In the orthogonal case, the quadratic space $\mathbb{W}$ is 2-dimensional over $F$. When $\mathbb{W}$ is anisotropic, both $\Int_{L^\flat,\sV}$ and $\pDen_{L^\flat,\sV}$ are still Schwartz functions on $\mathbb{V}$. But when $\mathbb{W}$ is isotropic, both $\Int_{L^\flat,\sV}$ and $\pDen_{L^\flat,\sV}$ in general have singularities near the isotropic cone of $\mathbb{W}$ and are no longer Schwartz functions on $\mathbb{V}$.  On the geometric side this reflects the fact that $\mathcal{Z}(L^\flat)$ is no longer quasi-compact when $\mathbb{W}$ is isotropic. Nevertheless we may still show that $\Int_{L^\flat,\sV}$ is locally integrable and show its Fourier invariance (up to a sign) as in (\ref{eq:FCmain}), but now understood as a distribution on $\mathbb{V}$. However, on the analytic side the singularities seem to cause essential difficulty in directly extending the argument in \cite{LZ2019} or \cite{LL2021} for controlling $\supp(\widehat{\pDen_{L^\flat,\sV}})$ as in (\ref{eq:FCmain}).

To overcome this difficulty, we instead perform a \emph{partial Fourier transform} along $L^\flat_F$ and consider new functions on $\mathbb{W}^\mathrm{an}$, $$\Intp(x):=\int_{L^\flat_F}\Int_{L^\flat, \sV}(y+x)\rd y,\quad \pDenp(x):=\int_{L^\flat_F}\pDen_{L^\flat, \sV}(y+x)\rd y.$$ On the geometric side, the Fourier transform $\wh{\Intp}$ of $\Intp$ on $\mathbb{W}$ agrees with the restriction of $\widehat{\Int_{L^\flat,\sV}}$ to $\mathbb{W}$. Now the advantage is that the new function $\Intp$ enjoys extra invariance under the action of the orthogonal group $\O(\mathbb{W})(F)$ and the scalars $O_F^\times$. Using the Weil representation and the theory of newforms for $\SL_2(F)$, we completely classify certain subspaces of \emph{invariant distributions} on $\mathbb{W}$ to which $\Intp$ belongs (Proposition \ref{prop:distributionbasis}, which may be of independent interest). In particular, we observe that a certain recurrence relation is satisfied for the values of $\Intp$ (Proposition \ref{prop:Intp}). On the analytic side, we directly verify that the same recurrence relation is also satisfied for $\pDenp$ (Proposition \ref{prop:pDenp})  via involved lattice-theoretic calculations which occupy \S\ref{s:FT ana}. Finally, the induction on the valuation of $L^\flat$ supplies the same initial values for the recursions on both sides, and allows us to conclude that $\Int_{L^\flat,\sV}=\pDen_{L^\flat,\sV}$. 

The strategy outlined above is a reminiscence of the uncertainty principle when $\mathbb{W}$ is anisotropic, but is more refined when $\mathbb{W}$ is isotropic.  We end by mentioning several technical complications compared to the unitary case when executing this strategy.  There are two (instead of one) relevant ambient quadratic spaces $\mathbb{V}=\mathbb{V}_m^\varepsilon$ of a given dimension $m$, which we have emphasized with the superscript $\varepsilon\in\{\pm1\}$. On the analytic side, we extend the results of \cite{CY} and \cite{Ikeda2017} to treat both cases $\varepsilon\in\{\pm1\}$ uniformly in \S\ref{sec:local-dens-quadr}, and the numerology is more complicated than the unitary case. On the geometric side, the GSpin Rapoport--Zink spaces is of Hodge type (instead of PEL type) which makes several proofs more technical. In particular, we provide proofs of two results on the vertical parts which are even new for the unitary case (see Remarks \ref{sec:supp-vert-part-1} and \ref{sec:comp-int_m}). Also the horizontal parts of special cycles are indexed by certain lattices of type $\le2$ (instead of type $\le1$), which we call \emph{horizontal lattices} (see Definition~\ref{def:horizontallattice}), and cause more complicated numerology as well. In particular, the horizontal identity eventually reduces to the case $n=3$ and $\varepsilon=+1$ (instead of $n=2$).

\subsection{The structure of the paper} In Part~\ref{part:local-kudla-rapoport}, we first prove necessary background results on both the analytic side (\S\ref{sec:local-dens-quadr}) and the geometric side (\S\ref{sec:kudla-rapop-cycl}--\S\ref{sec:vert-comp-kudla}) of the local arithmetic Siegel--Weil formula.   The Fourier invariance on the geometric side is proved in \S\ref{sec:four-transf-geom}. The recurrence relations satisfied by the partial Fourier transform on the analytic side is proved in \S\ref{s:FT ana}.  Finally in \S\ref{sec:invar-distr-proof}, we establish results on invariant distributions on 2-dimensional quadratic spaces and prove the main Theorem  \ref{thm:intro1}.  In Part~\ref{part:semi-global-global}, we first review incoherent Eisenstein series (\S\ref{sec:incoh-eisenst-seri-3}), semi-global integral models of GSpin Shimura varieties and their special cycles (\S\ref{sec:kudla-rapop-cycl-2}).  We then apply the local results in Part~\ref{part:local-kudla-rapoport} to prove the semi-global arithmetic Siegel--Weil formula (Theorem \ref{thm:intro2}).

For the sake of readability, we make an effort to ensure that the notations and the structure of this paper are parallel to those of the companion paper \cite{LZ2019} in the unitary case. We always write out the complete statements, and include details when a proof differs from the parallel proof in \cite{LZ2019} or point it out when the same proof of \cite{LZ2019} applies verbatim.

\subsection{Acknowledgments} C.~L.~was partially supported by the NSF grant DMS-1802269.  W.~Z.~was partially supported by the NSF grant DMS-1901642.

\part{Local arithmetic Siegel--Weil formula}\label{part:local-kudla-rapoport}

\section{Notations and Conventions}

\subsection{Notations on quadratic lattices}\label{sec:notat-quadr-latt-1}

Let $p$ be an odd prime. Let $F$ be a non-archimedean local field of residue characteristic $p$, with ring of integers $O_F$, residue field $\kappa=\mathbb{F}_q$ of size $q$, and uniformizer $\varpi$.  Let $\val:F\to \BZ\cup\{\infty\}$ be the valuation on $F$ and $|\cdot|:F\rightarrow \mathbb{R}_{\ge0}$ be the normalized absolute value on $F$. Let $\Fb$ be the completion of the maximal unramified extension of $F$, and $\OFb$ its ring of integers. Let $\sigma\in \Aut(\OFb)$ be the lift of the absolute $q$-Frobenius on $\kb$. We further assume that $F=\mathbb{Q}_p$ when dealing with the geometric side (the exceptions are  \S\ref{sec:local-dens-quadr}, \S\ref{s:FT ana}, \S\ref{sec:invar-distr-2}, which concern only the analytic side).

Let $L$ be a quadratic $O_F$-lattice of rank $n$ with symmetric bilinear form $(\ ,\ )$. We say $L$ is \emph{non-degenerate} if the extension $(\ ,\ )$ on the quadratic space $L_F:=L \otimes_{O_F}F$ is non-degenerate. Unless otherwise specified, all quadratic $O_F$-lattices are assumed to be non-degenerate throughout the paper. We denote by $L^\vee:=\{x\in L_F: (x,L)\subseteq O_F\}$ its dual lattice under $(\ ,\ )$. We say that $L$ is \emph{integral} if $L\subseteq L^\vee$. If $L$ is integral, define its \emph{fundamental invariants} to be the unique sequence of integers $(a_1,\ldots, a_n)$ such that $0\leq a_1\le \cdots\le a_n$, and $L^\vee/L\simeq \oplus_{i=1}^n O_F/\varpi^{a_i}$ as $O_F$-modules; define its \emph{valuation} to be $\val(L)\coloneqq \sum_{i=1}^na_i$; and define its \emph{type}, denoted by $t(L)$, to be the number of nonzero terms in its invariant $(a_1,\ldots, a_n)$. Denote by $\val(x):=\val((x,x))$ for any $x\in L$.  A \emph{standard orthogonal basis} of $L$ is an orthogonal $O_F$-basis $\{e_1,\ldots,e_n\}$ of $L$ such that $\val(e_i)=a_i$, which always exists. Say $L$ is \emph{self-dual} if $L=L^\vee$. Say $L$ is \emph{minuscule} or a \emph{vertex lattice} if it is integral and $L^\vee\subseteq \varpi^{-1}L$. Note that $L$ is a vertex lattice of type $t$ if and only if it has fundamental invariants $(0^{(n-t)},1^{(t)})$, if and only if $L\subseteq^t L^\vee\subseteq \varpi^{-1}L$, where $\subseteq^t$ indicates that the $O_F$-colength is equal to $t$.  Notice that $L$ is self-dual if and only if $L$ is a vertex lattice of type 0. 

The \emph{determinant} of $L$ is defined to be $$\det(L):=\det((x_i, x_j)_{i,j=1}^n)\in F^\times/(O_F^\times)^2$$ where $\{x_1,\ldots,x_n\}$ is an $O_F$-basis of $L$, and the \emph{discriminant} of $L$ is defined to be $$\disc(L):=(-1)^{n \choose 2}\cdot\det(L)\in F^\times/(O_F^\times)^2.$$ Notice that $\val(L)=\val(\disc(L))$. 

Let $\chi=(\frac{\cdot}{\varpi})_F: F^\times/(F^\times)^2\rightarrow\{\pm 1,0\}$ be the quadratic residue symbol. Define $$\chi(L):=\chi(\disc(L))\in\{\pm 1,0\}.$$ Notice that $\chi(L)=0$ if and only if $\val(L)$ is odd. If $L$ is self-dual, then $\chi(L)\in\{\pm1\}$ is the image of $\disc(L)$ under the isomorphism $O_F^\times/(O_F^\times)^2\cong \{\pm1\}$.

Let $\VV$ be a (non-degenerate) quadratic space over $F$ of dimension $m$ with symmetric bilinear form $(\ ,\ )$. Similar invariants are defined for quadratic spaces over $F$, and  we denote them by $$\det(\VV)\in F^\times/(F^\times)^2, \quad \disc(\VV)\in F^\times/(F^\times)^2,\quad \chi(\VV)\in\{\pm1,0\}.$$ Then $\chi(\VV)=+1$ if and only if $\disc(\VV)=1$. Also $\chi(L)=\chi(W)$ for any $O_F$-lattice $L\subseteq W$ of full rank. Define the \emph{Hasse invariant} of $\VV$ to be $$\Has(\VV):=\prod_{1\le i<j\le m}(u_i,u_j)_F\in\{\pm1\},$$ where $(\ , \ )_F: F^\times/(F^\times)^2 \times F^\times/(F^\times)^2\rightarrow\{\pm1\}$ is the Hilbert symbol, and $u_i=(e_i,e_i)$ for an orthogonal basis $\{e_1,\ldots e_m\}$ of $\VV$. 

Recall that quadratic spaces $\VV$ over $F$ are classified by its dimension $m\ge1$, its discriminant $\disc(\VV)$ and its Hasse invariant $\epsilon(\VV)$. When $m\ge3$, $\disc(\VV)$ and $\Has(\VV)$ can take arbitrary values. When $m=2$, the case $\disc(\VV)=1$ and $\Has(\VV)=-1$ is excluded. When $m=1$, the case $\Has(\VV)=-1$ is excluded. The space $\VV$ admits a self-dual lattice if and only if $\Has(\VV)=+1$ and $\chi(\VV)\ne0$.

For $k\in \mathbb{Z}$, denote by $$\VV^{\ge k}:=\{x\in \VV: \val(x)\ge k\}, \quad \VV^{=k}:=\{x\in \VV: \val(x)=k\}.$$ Denote by $$\VV^{\circ}:=\VV^{\ge0},\quad \VV^{\circ\circ}:=\VV^{\ge1},\quad \VV^\mathrm{an}:=\{x\in \VV: (x,x)\ne0\},$$ the integral cone,  the positive cone and the set of anisotropic vectors of $\VV$ respectively. For a quadratic $O_F$-lattice $L$, define $L^{\circ}:=L\cap (L_F)^{\circ}$ and $L^{\circ\circ}:=L\cap (L_F)^{\circ\circ}$. 

The set of vertex lattices of type $t$ in $\VV$ is denoted by $\Ver^t(\VV)$. 

For $\varepsilon\in\{\pm1\}$ and $m\ge1$, denote by $H_m^\varepsilon$ the self-dual $O_F$-lattice of rank $m$ with $\chi(H_m^\varepsilon)=\varepsilon$ (for $m=0$, only $H_m^+=0$ is allowed by convention). Denote by $\mathbb{V}_{m}^\varepsilon$ the quadratic space over $F$ with
\begin{center}
  $\chi(\mathbb{V}_{m}^\varepsilon)=\varepsilon$ and
  $\Has(\mathbb{V}_{m}^\varepsilon)=-1$.
\end{center}

\subsection{Notations on functions}\label{sec:notations-functions}

Let $\VV$ be a (non-degenerate) quadratic space over $F$. Fix an unramified additive character $\psi: F\to \BC^\times$. For an integrable function $f$ on $\VV$, we define its Fourier transform $\wh f$ to be $$\wh f(x):=\int_\VV f(y)\psi ((x,y))\rd y,\quad x\in \VV.$$ We normalize the Haar measure  on $\VV$  to be self-dual, so $\hat{\hat{f}}(x)=f(-x)$. For an $O_F$-lattice $\Lambda\subseteq \VV$ of full rank, we have
\begin{equation}\label{eq:latticefourier}
\wh{\bf 1}_{\Lambda}=\vol(\Lambda){\bf 1}_{\Lambda^\vee},\quad\text{and} \quad\vol(\Lambda)=[\Lambda^\vee:\Lambda]^{-1/2}=q^{-\val(\Lambda)/2}.
\end{equation}
Note that $\val(\Lambda)$ can be defined for any lattice $\Lambda$ (not necessarily integral) so that the above equality for $\vol(\Lambda)$ holds.

Denote by $\Ss(\VV)$  the space of Schwartz functions (i.e., compactly supported locally constant functions) on $\VV$. Denote by $\Dd(\VV):=\Hom(\Ss(\VV), \mathbb{C})$ the space of distributions on $\VV$ (the linear dual of $\Ss(\VV)$). Any Schwartz function is integrable. The Fourier transform preserves $\Ss(\VV)$ and induces a Fourier transform on $\Dd(\VV)$ such that $\hat T(f)=T(\hat f)$ for any $T\in \Dd(\VV)$, $f\in \Ss(\VV)$. Denote by $\supp(T)$ the support of the distribution $T$ (the complement of the largest open subset on which $T=0$). 

For any open dense subset $\Omega\subseteq \VV$, denote by $\Lloc(\Omega)$ the space of locally integrable functions on $\Omega$ (i.e., integrable on any compact open subset of $\Omega$). A function $\phi\in \Lloc(\VV)$ gives a distribution $T_\phi\in \Dd(\VV)$ represented by $\phi$, i.e., $T_\phi(f)=\int_{\Omega} \phi(x)f(x)\rd x$ for any $f\in \Ss(\VV)$. By abuse of notation, we often view $\phi\in\Lloc(\Omega)$ as a distribution on $\VV$ and write $\phi$ (resp. $\hat \phi$) instead of $T_\phi$ (resp. $\hat T_\phi$).

\subsection{Notations on formal schemes}\label{sec:notat-form-schem}

Denote by $\mathrm{ANilp}_{\OFb}$ the category of noetherian $\OFb$-algebras in which $\varpi$ is nilpotent. Denote by $\mathrm{ANilp}_{\OFb}^\mathrm{f}$ the category of noetherian adic $\OFb$-algebras in which $\varpi$ is nilpotent. Denote by $\ansm\subseteq \mathrm{ANilp}_{\OFb}^\mathrm{f}$ the full subcategory consisting of $\OFb$-algebras which are formally finitely generated and formally smooth over $\OFb/\varpi^k$ for some $k\ge1$. Denote by $\alg$ the category of noetherian $\varpi$-adically complete $\OFb$-algebras.

Let $X$ be a formal scheme.  Denote by $X^{\red}$ the underlying reduced scheme. For closed formal subschemes $\CZ_1,\cdots,\CZ_m$ of $X$, denote by $\cup_{i=1}^m\mathcal{Z}_i$ the formal scheme-theoretic union, i.e., the closed formal subscheme with ideal sheaf $\cap_{i=1}^m\mathcal{I}_{\mathcal{Z}_i}$, where $\mathcal{I}_{\mathcal{Z}_i}$ is the ideal sheaf of $\mathcal{Z}_i$.  A closed formal subscheme on $X$ is called a Cartier divisor if it is defined by an {\em invertible} ideal sheaf.

Let  $X$ be a formal scheme over $\Spf \OFb$. Then $X$ defines a functor on the category of $\Spf \OFb$-schemes (i.e. $\OFb$-schemes on which $\varpi$ is locally nilpotent). For $R\in\mathrm{ANilp}_{\OFb}^\mathrm{f}$ with ideal of definition $I$, write $X(R):=\varprojlim_nX(\Spec R/I^n)$. For $R\in\alg$, write $X(R):=\varprojlim_nX(\Spec R/\varpi^n)$.

 When $X$ is noetherian, denote by $K_0^Y(X)$ the Grothendieck group (modulo quasi-isomorphisms) of finite complexes of coherent locally free $\mathcal{O}_X$-modules, acyclic outside $Y$ (i.e., the homology sheaves are formally supported on $Y$). As defined in \cite[(B.1), (B.2)]{Zhang2019},  denote by $\mathrm{F}^i K_0^Y(X)$ the (descending) codimension filtration on $K_0^Y(X)$, and denote by $\Gr^{i}K_0^Y(X)$ its $i$-th graded piece. As in \cite[Appendix B]{Zhang2019}, the definition of $K_0^Y(X)$, $\mathrm{F}^i K_0^Y(X)$ and $\Gr^{i}K_0^Y(X)$ can be extended to locally noetherian formal schemes $X$ by writing $X$ as an increasing union of open noetherian formal subschemes. Similarly, we let $K_0'(X)$ denote the Grothendieck group of coherent sheaves of $\mathcal{O}_X$-modules. Now let $X$ be regular. Then there is a natural isomorphism $K_0^Y(X)\simeq K_0'(Y)$. For closed formal subschemes $\CZ_1,\cdots,\CZ_m$ of $X$, denote by $\CZ_1\jiao_X\cdots\jiao_X\CZ_m$ (or simply $\CZ_1\jiao\cdots\jiao\CZ_m$) the derived tensor product $\CO_{\CZ_1}\otimes^\BL_{\CO_X}\cdots \otimes^\BL_{\CO_X}\CO_{\CZ_m}$, viewed as an element in $K_0^{\CZ_1\cap\cdots\cap \CZ_m}(X)$.

For $\mathcal{F}$ a finite complex of coherent $\CO_X$-modules, we define its Euler--Poincar\'e characteristic $$\chi(X, \mathcal{F}):=\sum_{i,j}(-1)^{i+j}\length_{\OFb}H^i(X,H_j(\mathcal{F}))$$
if the lengths are all finite. Assume that $X$ is regular with pure dimension $n$. If $\mathcal{F}_i\in \mathrm{F}^{r_i}K_0^{\mathcal{Z}_i}(X)$ with $\sum_i r_i\geq n$, then by \cite[(B.3)]{Zhang2019} we know that $\chi(X, \bigotimes_i^\mathbb{L}\mathcal{F}_i)$ depends only on the image of $\mathcal{F}_i$ in $\Gr^{r_i}K_0^{\mathcal{Z}_i}(X)$.  In fact, we will only need this assertion when $X$ is a scheme (cf. Remark \ref{rem:avoid B3}). When $X$ is a formal scheme, the assertion holds trivially when one of the $r_i$ is $\dim X$; this special case will be used repeatedly. 

\subsection{Conventions}\label{sec:conventions}

Unless otherwise specified, we will denote by $L$ an  $O_F$-lattice of rank $n$, $L^\flat$ an $O_F$-lattice of rank $n-1$, and $\Lambda$ an $O_F$-lattice of full rank $m$ in a quadratic space of dimension $m$. Unless otherwise specified, all $O_F$-lattices are assumed to be non-degenerate.

Starting from \S\ref{sec:kudla-rapop-cycl},  we fix $m=n+1\ge3$ and $\varepsilon\in\{\pm1\}$.  For brevity we will suppress the superscript $\varepsilon$ and the subscripts $m$ and $n$ when there is no confusion, so $\mathbb{V}=\mathbb{V}_m^\varepsilon$, $\mathcal{N}=\mathcal{N}_n^\varepsilon$, $\Int(L)=\Int^\varepsilon(L)$, $\pDen(L)=\pDen^\varepsilon(L)$, $\Denf(1, L)=\Denfe(1,L)$, $\Hor(L^\flat)=\Hore(L^\flat)$ and so on.

\section{Local densities of quadratic lattices}\label{sec:local-dens-quadr}

\subsection{Local densities for quadratic lattices}\label{ss:loc den}
\begin{definition}
Let $L, M$ be two quadratic $O_F$-lattices. Let $\Rep_{M,L}$ be the \emph{scheme of integral representations of $M$ by $L$}, an $O_{F}$-scheme such that for any $O_{F}$-algebra $R$, \begin{align}\label{def: Rep}
\Rep_{M,L}(R)=\QHom(L \otimes_{O_{F}}R, M \otimes_{O_{F}}R),\end{align} where $\QHom$ denotes the set of quadratic module homomorphisms. The \emph{local density} of integral representations of $M$ by $L$ is defined to be $$\Den(M,L)\coloneqq \lim_{N\rightarrow +\infty}\frac{\#\Rep_{M,L}(O_{F}/\varpi^N)}{q^{N\cdot\dim (\Rep_{M,L})_{F}}}.$$  
\end{definition}
 Note that if $L, M$ have rank $n, m$ respectively and the generic fiber $(\Rep_{M,L})_{F}\ne\varnothing$, then $n\le m$ and
\begin{equation}
  \label{eq:dimRep}
\dim (\Rep_{M,L})_{F}=\dim \O_m-\dim\O_{m-n}={m\choose 2}-{m-n\choose 2}=\frac{n(2m-n-1)}{2}.  
\end{equation}

Our next goal is to obtain an explicit formulas for $\Den(H_m^\varepsilon, L)$ (Lemma \ref{lem:localden}). To do so we need some preliminaries on quadratic spaces over finite fields.

\subsection{Quadratic spaces over finite fields}

Let $V$ be a non-degenerate quadratic space over $\kappa$ of dimension $m$. Let $\chi(V)\in\{\pm1\}$ be the image of its discriminant $\disc(V)\in \kappa^\times/(\kappa^\times)^2$ under the isomorphism $\kappa^\times/(\kappa^\times)^2\cong\{\pm1\}$. By convention, if $V=0$ then $\det(V)=\disc(V)=1$ and $\chi(V)=+1$. Denote by $\O(V)$ the orthogonal group of $V$. Then we have the well-known formula $$\#\O(V)(\kappa)=
  \begin{cases}
    2 q^{m \choose 2}(1-\chi(V)\cdot q^{-m/2})\prod\limits_{i=1}^{m/2-1}(1-q^{-2i}), & m\text{ is even},\\
    2q^{m \choose 2}\prod\limits_{i=1}^{(m-1)/2}(1-q^{-2i}), & m\text{ is odd}.
  \end{cases}
$$ This can be uniformly written as
\begin{align}
  \#\O(V)(\kappa)&=2q^{m \choose 2}(1-\sgn(V)q^{-m/2})\prod_{1\le i<m/2}(1-q^{-2i}) \\
  &=2q^{m \choose 2}(1+\sgn(V)q^{-m/2})^{-1}\prod_{1\le i\le m/2}(1-q^{-2i})\label{eq:O(V)}
\end{align}
where we define $$\sgn(V):=
\begin{cases}
  \chi(V), & m\text{ is even}, \\
  0, & m\text{ is odd}.
\end{cases}$$ Notice the formula is true even for $m=0$ when interpreted as (\ref{eq:O(V)}).

More generally, for a possibly degenerate quadratic space $U$ over $\kappa$, we take an orthogonal decomposition
\begin{equation}\label{eq:U0U1}
U=U_0\obot U_1,  
\end{equation}
where $U_0$ is non-degenerate and $U_1$ is the radical of $U$, and define
\begin{equation}
  \label{eq:sgn}
\sgn(U):=\sgn(U_0)=\begin{cases}
  \chi(U_0), & \dim U_0\text{ is even}, \\
  0, & \dim U_0\text{ is odd}.
\end{cases}   
\end{equation}
This is independent of the decomposition (\ref{eq:U0U1}). Similarly, define
\begin{equation}
  \label{eq:sgn'}
  \sgn'(U):=
\begin{cases}
  0, & \dim U_0\text{ is even}, \\
  \chi(U_0), & \dim U_0\text{ is odd}. 
\end{cases}
\end{equation}
These two definitions can be written uniformly: for any integer $m$, define
\begin{equation}
  \label{eq:sgnm}
\sgn_m(U):=
\begin{cases}
  \chi(U_0), & \dim U_0\equiv m\pmod{2}, \\
  0, & \dim U_0\equiv m+1\pmod{2}.
\end{cases}
\end{equation}
Then $\sgn=\sgn_\mathrm{even}$ and $\sgn'=\sgn_\mathrm{odd}$.

The following lemma is a generalization of \cite[\S 5.6 Exercise 4]{Kitaoka1993}.
\begin{lemma}\label{lem:isometries}
  Let $U$ be a quadratic space over $\kappa$ of dimension $n$ whose radical has dimension $t$. Let $V$ be a non-degenerate quadratic space over $\kappa$ of dimension $m\ge n$. Let $\O(U,V)$ be the set of isometries from $U$ into $V$.  Then $\O(U,V)$ has size
  $$\#\O(U,V)=q^{\frac{n(2m-n-1)}{2}}    (1-\sgn(V)\cdot q^{-m/2})  
    (1-\chi(V)\sgn_m(U)\cdot q^{-(m-n-t)/2})^{-1}\cdot\prod\limits_{(m-n-t)/2\le i <m/2}(1-q^{-2i}).
$$ Namely,
\begin{multline*}
  \#\O(U,V)=q^{\frac{n(2m-n-1)}{2}}    (1-\sgn(V)\cdot q^{-m/2})\\
  \cdot\begin{cases}
    (1+\chi(V)\sgn(U)\cdot q^{-(m-n-t)/2})\cdot\prod\limits_{1\le i <(n+t)/2}(1-q^{2i-m}), & m \text{ is even},\\ 
    (1+\chi(V)\sgn'(U)\cdot q^{-(m-n-t)/2})\cdot\prod\limits_{1\le i <(n+t+1)/2}(1-q^{2i-(m+1)}), & m\text{ is odd}.
  \end{cases}
\end{multline*}
\end{lemma}

\begin{proof}
  The group $\O(V)(\kappa)$ acts transitively on $\O(U,V)$. Fixing an isometry $\phi\in \O(U,V)$ and identifying $U$ as a quadratic subspace of $V$ using $\phi$, we find the stabilizer of $\O(V)$ on $\phi$ is isomorphic to $$\{g\in \O(V): g|_{U}=1\}\cong \{g\in \O(U_0^\perp): g|_{U_1}=1\}=:H.$$ Here $U_0, U_1$ are as in decomposition (\ref{eq:U0U1}) and $U_0^\perp\subseteq V$ is the orthogonal complement of $U_0$. Notice that $U_1\subseteq U_0^\perp$ is totally isotropic. Let $U_2$ be the orthogonal complement of $U_1$ in $U_0^\perp$. Let $P=MN\subseteq \O(U_0^\perp)$ be the parabolic subgroup stabilizing the flag $0\subseteq U_1\subseteq U_2\subseteq U_0^\perp$. Then $H=M'N\subseteq P$, where $M'\subseteq M\simeq\GL(U_1)\times \O(U_2/U_1)$ is the subgroup $1\times \O(U_2/U_1)$. Notice that we have an isomorphism as affine varieties $N\cong \Hom(U_1, U_2/U_1)\times \wedge^2(U_1)$. It follows that the number of isometries is equal to $$\#\O(U,V)=\frac{\#\O(V)(\kappa)}{\#H(\kappa)}=\frac{\#\O(V)(\kappa)}{\#\O(U_2/U_1)(\kappa)\cdot \#\Hom(U_1, U_2/U_1)(\kappa)\cdot \#\wedge^2(U_1)(\kappa)}.$$  Notice that $\dim V=m$, $\dim U_2/U_1=m-n-t$, and $$\#\Hom(U_1, U_2/U_1)(\kappa)=q^{t(m-n-t)}, \quad \#\wedge^2(U_1)(\kappa)=q^{t\choose 2}.$$ We compute
 \begin{multline*}
   \frac{\#\O(V)(\kappa)}{\#H(\kappa)}=   q^{\frac{n(2m-n-1)}{2}}\cdot \frac{(1-\sgn(V)\cdot q^{-m/2})\prod\limits_{1\le i< m/2}(1-q^{-2i})}{(1-\sgn(U_2/U_1)\cdot q^{-(m-n-t)/2})\prod\limits_{1\le i<(m-n-t)/2}(1-q^{-2i})} \\
   = q^{\frac{n(2m-n-1)}{2}}\cdot (1-\sgn(V)\cdot q^{-m/2})(1+\sgn(U_2/U_1)\cdot q^{-(m-n-t)/2})\prod\limits_{(m-n-t)/2<i<m/2}(1-q^{-2i})
 \end{multline*}
 Notice that
 \begin{align*}
   \prod\limits_{(m-n-t)/2<i<m/2}(1-q^{-2i})=
   \begin{cases}
     \prod\limits_{1\le i<(n+t)/2}(1-q^{2i-m}), & m \text{ is even}, \\
     \prod\limits_{1\le i<(n+t+1)/2}(1-q^{2i-(m+1)}), & m \text{ is odd}.
   \end{cases}
 \end{align*}
Moreover, $$\sgn(U_2/U_1)=\sgn(U_0^\perp)=
 \begin{cases}
   \chi(V)\cdot \sgn(U), & m \text{ is even}, \\
   \chi(V)\cdot \sgn'(U), & m \text{ is odd}.
 \end{cases}
$$ This completes the proof.
\end{proof}

We deduce the following (well-known) counting formula for totally isotropic subspaces.

\begin{lemma}[number of totally isotropic subspaces]\label{lem:Smb} 
  Let $V$ be a non-degenerate quadratic space over $\kappa$ of dimension $m$. Let $\mathcal{S}_b(V)$ be the set of totally isotropic $\kappa$-subspaces of dimension $b$ in $V$, and $S_b(V):=\#\mathcal{S}_b(V)$. Then $$S_b(V)=
  \begin{cases}
    \displaystyle\frac{(q^{m/2}-\chi(V))(q^{m/2-b}+\chi(V))\prod_{i=1}^{b-1}(q^{m-2i}-1)}{\prod_{i=1}^b(q^{i}-1)},    & m\text{ is even},\\
    \displaystyle\frac{\prod_{i=0}^{b-1}(q^{m-1-2i}-1)}{\prod_{i=1}^b(q^{i}-1)},& m\text{ is odd}.
  \end{cases}$$
\end{lemma}

\begin{proof}
  The group $\O(V)(\kappa)$ acts transitively on the set $\mathcal{S}_b(V)$. Fix a totally isotropic subspace $U\in \mathcal{S}_b(V)$, then we have a surjection $$\O(U,V)\rightarrow \mathcal{S}_b(V),\quad \phi\mapsto \phi(U)$$ with each fiber in bijection with $\GL(U)(\kappa)$. Hence     $S_b(V)=\frac{\# \O(U, V)}{\#\GL(U)(\kappa)}$. Therefore by Lemma \ref{lem:isometries}, we know that $S_b(V)$ equals 
  \begin{equation*}
\frac{q^{\frac{b(2m-b-1)}{2}}    (1-\sgn(V)\cdot q^{-m/2})  
    (1-\chi(V)\sgn_m(U)\cdot q^{-(m-2b)/2})^{-1}\cdot\prod_{(m-2b)/2\le i <m/2}(1-q^{-2i})}{q^{b^2}\prod_{i=1}^b(1-q^{-i})},
  \end{equation*}
   which simplifies to the desired formula, as in this case $\chi(V)\sgn_m(U)=\sgn(V)$ is equal to $\chi(V)$ when $m$ is even and is equal to 0 when $m$ is odd.
 \end{proof}

 We record some consequences of Lemma \ref{lem:Smb}, which will be used throughout this article often without explicit reference.

\begin{corollary} Let $V$ be a non-degenerate quadratic space over $\kappa$ of dimension $m$.
  \begin{altenumerate}
  \item  The number of isotropic lines in $V$ is equal to $$S_1(V)=
    \begin{cases}
      \displaystyle\frac{(q^{m/2}-\chi(V))(q^{m/2-1}+\chi(V))}{q-1}, & m\text{ is even}, \\
      \displaystyle\frac{q^{m-1}-1}{q-1}, & m\text{ is odd}. 
    \end{cases}
$$ In particular, $$    S_1(V)=
\begin{cases}
  2, & m=2,\ \chi(V)=+1,\\
  0, & m=2,\ \chi(V)=-1, \\
  q+1, & m=3, \\
  (q+1)^2, & m=4,\ \chi(V)=+1,\\
  q^2+1, & m=4, \ \chi(V)=-1.
    \end{cases}
$$
\item The number of isotropic vectors in $V$ is equal to $$(q-1)S_1(V)+1=\begin{cases}
      \displaystyle q^{m-1}+\chi(V)q^{m/2}-\chi(V)q^{m/2-1}, & m\text{ is even}, \\
      \displaystyle q^{m-1}, & m\text{ is odd}. 
    \end{cases}$$    
  \end{altenumerate}
\end{corollary}

\subsection{Formulas in terms of weighted lattice counting: Theorem of Cho--Yamauchi}

\begin{definition}
Let $L$ be a quadratic $O_F$-lattice of rank $n$. Denote by $L_\kappa:=L \otimes_{O_F}\kappa$, a (possibly degenerate) quadratic space over $\kappa$, and $$\sgn(L):=\sgn(L_\kappa),\quad \sgn'(L):=\sgn'(L_\kappa),\quad \sgn_m(L):=\sgn_m(L_\kappa)$$ as in Equations (\ref{eq:sgn}), (\ref{eq:sgn'}) and (\ref{eq:sgnm}). By definition, if $L$ is self-dual, then $\sgn_n(L)=\chi(L)$.  
\end{definition}

We have the following explicit formula for local densities in terms of weighted lattice counting, generalizing the theorem of Cho--Yamauchi.

\begin{lemma}[Cho--Yamauchi]\label{lem:localden} Let $L$ be a quadratic $O_F$-lattice of rank $n$. Then
  \begin{multline}\label{eq:generalDen}
    \Den(H_{m}^\varepsilon, L)
    =\sum_{L\subseteq L'\subseteq L'^\vee}q^{(n+1-m)\ell(L'/L)}\cdot (1-\sgn(H_{m}^\varepsilon)\cdot q^{-m/2})\\\quad(1+\varepsilon \sgn_m(L')\cdot q^{-(m-n-t(L'))/2})\cdot\prod\limits_{(m-n-t(L'))/2< i<m/2}(1-q^{-2i}).
  \end{multline}
Here the sum runs over all integral lattices $L'\subseteq L_F$ such that $L\subset L'$, and $$\ell(L'/L)\coloneqq {\rm length}_{O_F}\,L'/L.$$  
\end{lemma}

\begin{proof}
  By \cite[Equation (3.4)]{CY} (replacing $2k$ there by $m$), we have
  \begin{align*}
    \Den(H_{m}^{\varepsilon}, L)=q^{-\frac{n(2m-n-1)}{2}}\cdot \sum_{L\subseteq L'\subseteq L'^\vee}q^{(n+1-m)\ell(L'/L)} \#\O(L'_\kappa, (H_{m}^{\varepsilon})_{\kappa}).
  \end{align*}
  Here $\O(U,V)$ denotes the set of isometries from $U$ to $V$ as in Lemma \ref{lem:isometries}.  Strictly speaking, \cite{CY} only treats the case $m$ is even and $\varepsilon=+1$, but the same proof goes through as $H_{m}^\varepsilon$ is self-dual.

  It follows from Lemma \ref{lem:isometries} that \begin{align*}
    \Den(H_{m}^\varepsilon, L)
    &=\sum_{L\subseteq L'\subseteq L'^\vee}q^{(n+1-m)\ell(L'/L)}\cdot (1-\sgn(H_{m}^\varepsilon)\cdot q^{-m/2})\\&\quad (1-\varepsilon \sgn_m(L')\cdot q^{-(m-n-t(L'))/2})^{-1}\cdot\prod\limits_{(m-n-t(L'))/2\le i <m/2}(1-q^{-2i})\\
    &=\sum_{L\subseteq L'\subseteq L'^\vee}q^{(n+1-m)\ell(L'/L)}\cdot (1-\sgn(H_{m}^\varepsilon)\cdot q^{-m/2})\\&\quad(1+\varepsilon \sgn_m(L')\cdot q^{-(m-n-t(L'))/2})\cdot\prod\limits_{(m-n-t(L'))/2< i<m/2}(1-q^{-2i}).
  \end{align*} This completes the proof.
\end{proof}

We have the following induction formula for local densities, generalizing the results of Cho--Yamauchi and Katsurada. 

\begin{lemma}[Induction formula]\label{lem:naiveinduction}
Let $L^\flat$ be a quadratic $O_F$-lattice of rank $n-1$ with fundamental invariants $(a_1,\cdots,a_{n-1})$. Let $L=L^\flat+\pair{x}$ and  $\wit L=L^\flat+\pair{\varpi^{-1} x}$ where $x\perp L^\flat$ with $\val(x)>a_{n-1}$. If $m$ is even, then $$\Den(H_m^\varepsilon, L)=q^{n+1-m}\cdot\Den(H_m^\varepsilon, \wit L)+ (1-\varepsilon q^{-m/2})(1+\varepsilon q^{-(m-2)/2})\cdot\Den(H_{m-2}^\varepsilon, L^\flat).$$ If $m$ is odd, then $$\Den(H_m^\varepsilon, L)=q^{n+1-m}\cdot\Den(H_m^\varepsilon, \wit L)+ (1-q^{-(m-1)})\cdot\Den(H_{m-2}^\varepsilon, L^\flat).$$
\end{lemma}

\begin{proof}
  When $m$ is even and $\varepsilon=+1$, this is proved in \cite[Corollary 4.10]{CY} (see also \cite[Theorem 2.6 (1)]{Katsurada1999}). The general case can be proved similarly. More precisely, consider the terms in (\ref{eq:generalDen}) indexed by lattices $L\subseteq L'\subseteq L'^\vee$ depending on $\wit L\subseteq L'$ or not.

  The sum of terms in (\ref{eq:generalDen}) with $L'$ satisfying $\wit L\subseteq L'$ evaluates to $$q^{(n+1-m)(\ell(L'/L)-\ell(L'/\tilde L))}\cdot \Den(H_m^\varepsilon, \tilde L)=q^{n+1-m}\cdot \Den(H_m^\varepsilon, \tilde L).$$

Now we consider those $L'$'s satisfying $\wit L\not\subseteq L'$. In this case the image of $x$ in $L'_\kappa$ is nonzero, hence by Nakayama's lemma, there exists a quadratic $O_F$-lattice $M'$ of rank $n-1$ such that $L'=M'+\langle x\rangle$ and $L=M+\langle x\rangle$, where $M=L\cap M'$. Since $\val(x)>a_{n-1}$, we know that $M$ and $L^\flat$ has the same fundamental invariants and moreover $\Den(H_m^\varepsilon, M)=\Den(H_m^\varepsilon, L^\flat)$ (\cite[Remark 4.3 (5)]{CY}). So by \cite[Proposition 4.8]{CY} (specialized to $d=1$ in the notation there), the sum of terms in (\ref{eq:generalDen}) with  $L'$ satisfying $\wit L\not\subseteq L'$ evaluates to
    \begin{multline*}
    \sum_{L^\flat\subseteq M'\subseteq M'^\vee\atop L'=M'+\langle x\rangle}q^{\ell(L'/L)}\cdot q^{(n+1-m)\ell(L'/L)}\cdot (1-\sgn(H_{m}^\varepsilon)\cdot q^{-m/2})\\\quad(1+\varepsilon \sgn_{m}(L')\cdot q^{-(m-n-t(L'))/2})\cdot\prod\limits_{(m-n-t(L'))/2< i<m/2}(1-q^{-2i}).
  \end{multline*}
Since $t(M')=t(L')-1$, $\ell(M'/L^\flat)=\ell(L'/L)$ and $\sgn_{m-2}(M')=\sgn_m(L')$, this evaluates to $$\frac{(1-\sgn(H_{m}^\varepsilon)\cdot q^{-m/2})\prod\limits_{(m-2)/2\le i<m/2}(1-q^{-2i})}{(1-\sgn(H_{m-2}^\varepsilon)\cdot q^{-(m-2)/2})}\cdot\Den(H_{m-2}^\varepsilon, L^\flat).$$ Notice that the extra factor evaluates to $$\frac{(1-\sgn(H_{m}^\varepsilon)\cdot q^{-m/2})\prod\limits_{(m-2)/2\le i<m/2}(1-q^{-2i})}{(1-\sgn(H_{m-2}^\varepsilon)\cdot q^{-(m-2)/2})}=
\begin{cases} 
  (1-\varepsilon q^{-m/2})(1+\varepsilon q^{-(m-2)/2}),& m\text{ is even},\\
  (1-q^{-(m-1)}), & m \text{ is odd}.
\end{cases}$$
This finishes the proof.
\end{proof}

Our next goal is to define normalized local Siegel series and derive explicit formulas for them (Theorems \ref{thm: Den(X)}, \ref{thm: Denf(X)} and \ref{thm:induction}). We distinguish two cases depending on the parity of the corank of $L$ in $H_m^\varepsilon$.

\subsection{Odd corank case}

\begin{definition}
  Let $L$ be a quadratic $O_F$-lattice of rank $n$. Define the \emph{normalizing polynomial} (in the odd corank case) $\Nore(X,L)\in \mathbb{Q}[X]$ to be
  \begin{equation}
    \label{eq:norpoly}
    \Nore(X,L)=(1-\sgn(H_{n+1}^\varepsilon)\cdot q^{-(n+1)/2}X)\prod_{1\le i< (n+1)/2}(1-q^{-2i}X^2).
  \end{equation}
 Notice that the dependence of $\Nore(X,L)$ on $L$ is only its rank $n$. By Lemma \ref{lem:localden}, we have for all $k\ge0$, $$\Nore(q^{-k},L)=\Den(H_{n+1+2k}^\varepsilon, H_n^\varepsilon).$$
\end{definition}

\begin{definition}\label{def:Dene}
  Define the (normalized) \emph{local Siegel series} of $L$ (in the odd corank case) to be the polynomial $\Dene(X,L)\in \mathbb{Z}[X]$ such that for all $k\ge0$, $$\Dene(q^{-k},L)=\frac{\Den(H_{n+1+2k}^\varepsilon, L)}{\Nor^\varepsilon(q^{-k},L)}=\frac{\Den(H_{n+1+2k}^\varepsilon, L)}{\Den(H_{n+1+2k}^\varepsilon, H_n^\varepsilon)}.$$ Define the \emph{central derivative of the local density} or \emph{derived local density} to be $$\pDene(L)\coloneqq-\frac{\rd}{\rd X}\bigg|_{X=1}\Dene(X,L).$$ Notice that if $L$ is not integral then $\Dene(X,L)=0$ and hence $\pDene(L)=0$.
\end{definition}

\begin{remark}
By definition $\Dene(M,L)$ only depends on the isometry classes of $M$ and $L$, and hence $\Dene(X,L)$ and $\pDene(L)$ only depends on the isometry class of $L$.   
\end{remark}

\begin{definition}
  Let $t\ge1$, $s\in \{\pm1,0\}$, $\varepsilon\in\{\pm1\}$ such that if $t-1$ is odd then $s=0$.  Define weight polynomials $$\wt(t, s; X):=(1+s\cdot q^{(t-1)/2}X)\prod_{0\le i<(t-1)/2}(1-q^{2i}X^2), \quad \wte(t, s; X):=\wt(t, s; \varepsilon X).$$ By convention, define $\wte(0, s; X)=1$. Define weight factors
  \begin{align*}
    \wte(t,s)&\coloneqq -\frac{\rd}{\rd X}\bigg|_{X=1}\wte(t,s; X)
    =
    \begin{cases}
      0, & t=0,\\
      -\varepsilon s, & t=1,\\
      2(1+\varepsilon s\cdot q^{(t-1)/2})\prod_{1\le i <(t-1)/2}(1-q^{2i}), & t\ge 2.
    \end{cases}
  \end{align*}
\end{definition}
Now we have the following explicit formula for the local Siegel series $\Dene(X, L)$, generalizing \cite[Corollary 3.16]{CY} (the case $n$ is odd and $\varepsilon=+1$).

\begin{theorem}\label{thm: Den(X)}
Let $L$ be a quadratic $O_F$-lattice of rank $n$. Then $$\Dene(X,L)=
  \sum_{L\subset L'\subset L'^\vee} X^{2\ell( L'/L)}\cdot  \wte(t(L'), \sgn_{n+1}(L'); X).
$$

\end{theorem}
\begin{proof} 
Take $m=n+1+2k$ in Lemma \ref{lem:localden}, we obtain that \begin{align*}
    \Den(H_{n+1+2k}^\varepsilon, L)
    &=\sum_{L\subseteq L'\subseteq L'^\vee}q^{-2k\cdot\ell(L'/L)}\cdot (1-\sgn(H_{n+1}^\varepsilon)\cdot q^{-(n+1)/2-k})\\&\quad(1+\varepsilon\sgn_{n+1}(L')\cdot q^{(t(L')-1)/2-k})\prod_{-(t(L')-1)/2+k< i< (n+1)/2+k}(1-q^{-2i}).
  \end{align*}
  Taking the ratio we obtain
  \begin{align*}
    &\phantom{=}\frac{\Den(H_{n+1+2k}^\varepsilon, L)}{\Nore(q^{-k}, L)}\\
    &=\sum_{L\subseteq L'\subseteq L'^\vee}q^{-2k\cdot\ell(L'/L)}(1+\varepsilon\sgn_{n+1}(L')\cdot q^{(t(L')-1)/2-k})\prod_{-(t(L')-1)/2+k< i\le k }(1-q^{-2i})\\
    &=\sum_{L\subseteq L'\subseteq L'^\vee}q^{-2k\cdot\ell(L'/L)}(1+\varepsilon\sgn_{n+1}(L')\cdot q^{(t(L')-1)/2-k})\prod_{0\le i< (t(L')-1)/2 }(1-q^{2i-2k})\\
    &=\sum_{L\subseteq L'\subseteq L'^\vee} X^{2\ell( L'/L)}\cdot  \wte(t(L'), \sgn_{n+1}(L');X)\bigg|_{X=q^{-k}}.
  \end{align*}
  This completes the proof.
  \end{proof}

  We have the following functional equation for $\Dene(X,L)$.
  
\begin{theorem}[Ikeda]\label{thm:ikeda}
Let $L$ be a quadratic $O_F$-lattice of rank $n$. Then $$\Dene(X,L)=w^\varepsilon(L)\cdot X^{\val(L)}\cdot \Dene\left(\frac{1}{X},L\right),$$ where the sign of functional equation is equal to
\begin{equation}
  \label{eq:signwe}
  w^\varepsilon(L):= (\det L, -(-1)^{n+1\choose 2}u)_F\cdot\Has(L_F)\in\{\pm1\},
\end{equation} where $u\in O_F^\times$ such that $\chi(u)=\varepsilon$.
\end{theorem}

\begin{proof}
  This is \cite[Theorem 4.1 (2)]{Ikeda2017} when $n$ is odd and $\varepsilon=+1$. The same proof works in general.
\end{proof}

\begin{corollary}\label{cor: pDen} If $\Dene(X,L)$ has sign of functional equation $w^\varepsilon(L)=-1$, then
$$\pDene(L)=
  \sum_{L\subset L'\subset L'^\vee}  \wte(t(L'),\sgn_{n+1}(L')).
  $$
\end{corollary}

\begin{remark}[A cancellation law for $\pDene(L)$]
    For  a self-dual $O_F$-lattice $M$ of rank $r$ and an $O_F$-quadratic lattice $L^\flat$ of rank $n-r$, by Theorem \ref{thm: Den(X)} we have
$$\Dene(X,L^\flat\obot M)= \Den^{\varepsilon'}(X,L^\flat)$$ for the unique $\varepsilon'\in\{\pm1\}$ such that $H^{\varepsilon'}_{n-r+1}\obot M\cong H^\varepsilon_{n+1}$.  Therefore we obtain a cancellation law:
\begin{align}
\label{eq:cancel den}
 \pDene(L^\flat\obot M)=\pDen^{\varepsilon'}(L^\flat).
\end{align}
\end{remark}

\begin{remark}[Relation with local Whittaker functions]
  \label{sec:relation-with-local} Let $\Lambda=H_m^\varepsilon$ be a self-dual quadratic $O_F$-lattice of rank $m$. Let $L$ be a quadratic $O_F$-lattice of rank $n$. Let $T=((x_i, x_j))_{1\le i,j\le n}$ be the fundamental matrix of an $O_F$-basis $\{x_1,\ldots, x_n\}$ of $L$, an $n\times n$ symmetric matrix over $F$. Associated to the standard Siegel--Weil section of the characteristic function $\varphi_0=\mathbf{1}_{\Lambda^n}$ and the unramified additive character $\psi: F\rightarrow \mathbb{C}^\times$, there is a local (generalized) Whittaker function $W_T(g, s, \varphi_0)$ (see \S\ref{sec:four-coeff-deriv}, \S\ref{sec:incoh-eisenst-seri} for the precise definition). By \cite[Proposition A.6]{Kudla1997a}, when $g=1$, it satisfies the interpolation formula for integers $s=k\ge0$ (notice $\gamma(V)=1$ in the notation there), $$W_T(1, k,\varphi_0)=\Den(\Lambda \obot H_{2k}^+, L).$$ Assume $m=n+1$.  By Definition \ref{def:Dene}, it follows that its value at $s=k$ is $$W_T(1, k, \varphi_0)=\Den(H_{n+1+2k}^\varepsilon, L)=\Den^\varepsilon(q^{-k}, L)\cdot \Nor^\varepsilon(q^{-k},L),$$ and when $w^\varepsilon(L)=-1$, its derivative at $s=0$ is $$W_T'(1, 0, \varphi_0)=\pDen^\varepsilon(L)\cdot \Nor^\varepsilon(1, L)\cdot\log q.$$ Plugging in (\ref{eq:norpoly}), we obtain \begin{align}
W_T(1, 0, \varphi_0)&=\Den^\varepsilon(1,L)\cdot (1-\sgn(H_{n+1}^\varepsilon)\cdot q^{-(n+1)/2})\prod_{1\le i< (n+1)/2}(1-q^{-2i}),\label{eq:localWhittaker0}\\ W_T'(1, 0, \varphi_0)&=\pDen^\varepsilon(L)\cdot(1-\sgn(H_{n+1}^\varepsilon)\cdot q^{-(n+1)/2})\prod_{1\le i< (n+1)/2}(1-q^{-2i})\cdot \log q.    \label{eq:localWhittaker1}
\end{align}
\end{remark}

\subsection{Even corank case}

\begin{definition}
Let $L$ be a quadratic $O_F$-lattice of rank $n$. Define the \emph{normalizing polynomial} (in the even corank case) $\Nore(X,L)\in \mathbb{Q}[X]$ to be $$\Norfe(X,L)=(1-\sgn(H_n^\varepsilon)\cdot q^{-n/2}X)(1-\varepsilon\chi(L)X)^{-1}\prod\limits_{0\le i<n/2}(1-q^{-2i}X^2).$$ Notice that the dependence of $\Nore(X,L)$ on $L$ is only its rank $n$ and $\chi(L)$. 
\end{definition}

\begin{definition}\label{def:denfe}
 Define the (normalized) \emph{local Siegel series} of $L$ (in the even corank case) to be the polynomial $\Denfe(X, L)\in \mathbb{Z}[X]$ such that for all $k\ge0$, $$\Denfe(q^{-k},L):=\frac{\Den(H_{n+2k}^\varepsilon, L)}{\Norfe(q^{-k},L)}.$$
\end{definition}

\begin{definition}
  Let $t\ge1$, $s\in \{\pm1,0\}$, $\chi\in\{\pm1,0\}$, $\varepsilon\in\{\pm1\}$ such that if $t$ is odd then $s=0$. Define weight polynomials $$\wtf(t, s,\chi; X):=(1+s\cdot q^{t/2}X)(1-\chi \cdot X)\prod_{0< i<
    t/2}(1-q^{2i}X^2), \quad \wtfe(t, s,\chi; X):=\wtf(t, 
  s,\chi; \varepsilon X).$$ By convention, define $\wtfe(0, s,\chi; X)=1$. 
\end{definition}

Similar to Theorem \ref{thm: Den(X)}, we have the following explicit formula for $\Denfe(X,L)$, generalizing \cite[Corollary 3.16]{CY} (the case $n$ is even and $\varepsilon=+1$).

\begin{theorem}\label{thm: Denf(X)}
  Let $L$ be a quadratic $O_F$-lattice of rank $n$. Then $$\Denfe(X,L)=
  \sum_{L\subset L'\subset L'^\vee} (q^{1/2}X)^{2\ell( L'/L)}\cdot \wtfe(t(L'), \sgn_{n}(L'),\chi(L); X).$$
\end{theorem}

\begin{proof}
Take $m=n+2k$ in Lemma \ref{lem:localden}, we obtain  \begin{align*}
    \Den(H_{n+2k}^\varepsilon, L)&=\sum_{L\subseteq L'\subseteq L'^\vee}q^{(1-2k)\ell(L'/L)}\cdot(1-\sgn(H_{n}^\varepsilon)\cdot q^{-n/2-k})\\ &\quad\cdot(1+\varepsilon\sgn_n(L')\cdot q^{t(L')/2-k})\prod_{-t(L')/2+k< i< n/2+k}(1-q^{-2i}).
  \end{align*}
  Taking the ratio we obtain
  \begin{align*}
    &\phantom{=}\frac{\Den(H_{n+1+2k}^\varepsilon, L)}{\Norfe(q^{-k}, L)} \\
    &=\sum_{L\subseteq L'\subseteq L'^\vee}q^{(1-2k)\ell(L'/L)}(1+\varepsilon\sgn_{n}(L')\cdot q^{t(L')/2-k})(1-\varepsilon \chi(L)\cdot q^{-k})\prod_{-t(L')/2+k< i< k }(1-q^{-2i})\\
    &=\sum_{L\subseteq L'\subseteq L'^\vee}q^{(1-2k)\ell(L'/L)}(1+\varepsilon\sgn_{n}(L')\cdot q^{t(L')/2-k})(1-\varepsilon \chi(L)\cdot q^{-k})\prod_{0< i< t(L')/2 }(1-q^{2i-2k})\\
    &=\sum_{L\subseteq L'\subseteq L'^\vee} (q^{1/2}X)^{2\ell( L'/L)}\cdot \wtfe(t(L'), \sgn_{n}(L'),\chi(L);X)\bigg|_{X=q^{-k}}.
  \end{align*}
  This completes the proof.
\end{proof}

We have the following functional equation for $\Denfe(X,L)$.

\begin{theorem}[Ikeda]\label{thm:ikedaeven}
  Let $L$ be a quadratic $O_F$-lattice of rank $n$. Then $$\Denfe(X, L)=(q^{1/2}X)^{2\lfloor \frac{\val(L)}{2}\rfloor}\cdot \Denfe\left(\frac{1}{qX}, L\right).$$
\end{theorem}

\begin{proof}
  This is \cite[Theorem 4.1 (1)]{Ikeda2017} when $n$ is even and $\varepsilon=+1$. The same proof works in general.
\end{proof}

\begin{corollary}\label{cor:Denfe1L}
Let $L$ be a quadratic $O_F$-lattice of rank $n$. Then
  \begin{align*}
    \Denfe(1, L)&=    \sum_{L\subseteq L'\subseteq L'^\vee\atop t(L')\le2} q^{\lfloor \frac{\val(L')}{2}\rfloor}\cdot
     \begin{cases}
    1, & t(L')=0,\\
    (1-\varepsilon \chi(L) q^{-1}), & t(L')=1, \\
    (1+\varepsilon\sgn_n(L'))(1-\varepsilon \chi(L) q^{-1}), & t(L')=2.
  \end{cases}    
  \end{align*}
\end{corollary}

\begin{proof}
By Theorem \ref{thm:ikedaeven}, we have $$\Denfe(1, L)=q^{\lfloor \frac{\val(L)}{2}\rfloor}\Denfe(q^{-1}, L).$$ By Theorem \ref{thm: Denf(X)},  $$\Denfe(q^{-1}, L)=\sum_{L\subset L'\subset L'^\vee} q^{-\ell( L'/L)}\cdot \wtfe(t(L'), \sgn_{n}(L'),\chi(L); q^{-1}).$$ Notice that the weight polynomial evaluates to $$\wtfe(t, s,\chi; q^{-1})=
\begin{cases}
  1, & t=0,\\
  (1-\varepsilon\chi q^{-1}), & t=1,\\
  (1+\varepsilon s)(1-\varepsilon \chi q^{-1}), & t=2,\\
  0, & t>2.
\end{cases}$$
This completes the proof since $q^{\lfloor \frac{\val(L)}{2}\rfloor}\cdot q^{-\ell(L'/L)}=q^{\lfloor \frac{\val(L')}{2}\rfloor}$.
\end{proof}

\subsection{Induction formula}

\begin{theorem}[Induction formula]\label{thm:induction}
  Let $L^\flat$ be a quadratic $O_F$-lattice of rank $n-1$ with fundamental invariants $(a_1,\cdots,a_{n-1})$. Let $L=L^\flat+\pair{x}$ and  $\wit L=L^\flat+\pair{\varpi^{-1} x}$ where $x\perp L^\flat$ with $\val(x)>a_{n-1}$. Then
  \begin{equation}
    \label{eq:inductionDen}
    \Dene(X,L)=X^2\cdot \Dene(X, \wit L)+(1-\varepsilon\chi(L^\flat)X)^{-1}(1-X^2)\cdot \Denfe(X, L^\flat).
  \end{equation}
\end{theorem}

\begin{proof}
  It follows immediately from Lemma \ref{lem:naiveinduction} by evaluating both sides at $X=q^{-k}$ ($k\ge0$) using the definition of $\Dene(X,L)$ (Definition \ref{def:Dene}) and $\Denfe(X, L^\flat)$  (Definition \ref{def:denfe}).
\end{proof}

\begin{corollary}\label{cor:pDendiffgeneral}
  Assume the situation is as in Theorem \ref{thm:induction}. Assume that $w^\varepsilon(L)=-1$. Then $$\pDene(L)-\pDene(\wit L)=
  \begin{cases}
    -\varepsilon\chi(L^\flat)\cdot \Denfe(1, L^\flat), & \chi(L^\flat)\ne0\\
    2\cdot \Denfe(1, L^\flat), & \chi(L^\flat)=0.
  \end{cases}$$
\end{corollary}

\begin{proof} Take $-\frac{\rd}{\rd X}\big|_{X=1}$ on both sides of (\ref{eq:inductionDen}).
\end{proof}

\subsection{Lemmas on quadratic lattices}

\begin{lemma}\label{lem:directsum}
  Let $W=W_1 \obot W_2$, where $W_i$ is a quadratic space over $F$ of dimension $m_i$ ($i=1,2$). Then
  \begin{altenumerate}
  \item\label{item:l1} $\det(W)=\det(W_1)\det(W_2)$.
  \item\label{item:l2} $\disc(W)=(-1)^{m_1m_2}\disc(W_1)\disc(W_2)$. In particular, $\chi(W)=\chi(W_1)\chi(W_2)$ if $m_1m_2$ is even and at least one of $\chi(W_i)$ is nonzero.
  \item\label{item:l3} $\Has(W)=\Has(W_1)\Has(W_2)(\det (W_1), \det (W_2))_F$.
  \end{altenumerate}
\end{lemma}

\begin{proof}
  \begin{altenumerate}
  \item It follows from the definition.
  \item It follows from (\ref{item:l1}) and ${m_1+m_2 \choose 2}={m_1\choose 2}+{m_2 \choose 2}+ m_1m_2$.
  \item Choose an orthogonal basis $\{e_1,\ldots, e_{m_1}\}$ of $W_1$ and an orthogonal basis $\{e_{m_1+1},\ldots, e_{m_1+m_2}\}$ of $W_2$. Let $u_i=(e_i, e_i)$.  Then
    \begin{align*}
      \Has(W)&=\prod_{1\le i<j\le m_1+m_2}(u_i, u_j)_F\\
      & =\prod_{1\le i <j \le m_1}(u_i,u_j)_F\cdot \prod_{m_1+1\le i< j\le m_1+m_2}(u_i,u_j)_F\cdot \prod_{1\le i\le m_1\atop m_1+1\le j\le m_1+m_2} (u_i, u_j)_F\\ 
      &=\Has(W_1)\Has(W_2)\left(\prod\nolimits_{1\le i\le m_1}u_i,\prod\nolimits_{m_1+1\le j\le m_1+m_2} u_j\right)_F\\
      &=\Has(W_1)\Has(W_2)(\det (W_1), \det (W_2))_F.
    \end{align*}
    This completes the proof. \qedhere
  \end{altenumerate}
\end{proof}

\begin{lemma}\label{lem:latticemebed}Take $m=n+1$ and $\varepsilon\in\{\pm1\}$.
  \begin{altenumerate}
  \item\label{item:embed1} Let $L\subseteq \mathbb{V}_m^\varepsilon$ (resp. $L\subseteq H_{m,F}^\varepsilon$) be a quadratic $O_F$-lattice of rank $n$. Then the sign of functional equation \eqref{eq:signwe} $w^\varepsilon(L)=-1$ (resp. +1). In particular, a quadratic $O_F$-lattice of rank $n$ cannot be simultaneously embedded into $H_{m,F}^\varepsilon$ and $\mathbb{V}_m^\varepsilon$ as quadratic submodules.
  \item\label{item:embed2} Let $L^\flat$ be a quadratic $O_F$-lattice of rank $n-1$.   Assume that $L^\flat$ can be simultaneously embedded into $H_{m,F}^\varepsilon$ and $\mathbb{V}_{m}^\varepsilon$ as quadratic submodules. Then $\chi(L^\flat)=0$ or $\chi(L^\flat)=-\varepsilon$.
  \end{altenumerate}
\end{lemma}

\begin{proof}
  \begin{altenumerate}
  \item Let $L\subseteq \mathbb{V}_m^\varepsilon$ (the case $L\subseteq H_{m,F}^\varepsilon$ is parallel). Write $\mathbb{V}_m^\varepsilon=L_F \obot L_F^\perp$. Then by Lemma \ref{lem:directsum} (\ref{item:l1}) (\ref{item:l3}) and $(a,a)_F=(a,-1)_F$ for any $a\in F^\times$, we obtain $$-1=\Has(\mathbb{V}_m^\varepsilon)=\Has(L_F)\Has(L_F^\perp)(\det(L_F),\det(\mathbb{V}_m^\varepsilon)\det(L_F))=\Has(L_F)\Has(L_F^\perp)(\det(L_F),-\det(\mathbb{V}_m^\varepsilon)).$$ The result then  follows as $\Has(L_F^\perp)=+1$ ($L_F^\perp$ is 1-dimensional).  
  \item If not, assume that $\chi(L^\flat)=\varepsilon$. Write $H_{m,F}^\varepsilon\cong L^\flat_F \obot W$ and $\mathbb{V}_m^\varepsilon\cong L_F^\flat \obot W'$, where $W$ and $W'$ are quadratic spaces over $F$ of dimension 2. By Lemma \ref{lem:directsum} (\ref{item:l3}), we have $$\Has(H_{m,F}^\varepsilon)=\Has(L_F^\flat)\Has(W)(\det (L_F^\flat), \det (W))_F,\quad \Has(\mathbb{V}_m^\varepsilon)=\Has(L_F^\flat)\Has(W')(\det (L_F^\flat), \det (W'))_F.$$ Since $\chi(H_{m,F}^\varepsilon)=\chi(\mathbb{V}_m^\varepsilon)=\varepsilon$, and $\chi(L^\flat)=\varepsilon$, we know that $\det(W)=\det(W')$ and $\chi(W)=\chi(W')=+1$ by Lemma \ref{lem:directsum} (\ref{item:l1}) (\ref{item:l2}). Since $W$ and $W'$ are of dimension 2, we know that $\Has(W)=\Has (W')=+1$. It follows that $\Has(H_{m,F}^\varepsilon)=\Has(\mathbb{V}_m^\varepsilon)$, a contradiction.     \qedhere
  \end{altenumerate}
\end{proof}

\subsection{Horizontal lattices} Take $m=n+1$ and $\varepsilon\in\{\pm1\}$ in this subsection.

\begin{definition}\label{def:horizontallattice}
Let $L^\flat\subseteq \mathbb{V}_m^\varepsilon$ be an $O_F$-lattice of rank $n-1$.
\begin{enumerate}
\item Say $L^\flat$ is \emph{co-isotropic} (in $\mathbb{V}_m^\varepsilon$), if the 2-dimensional quadratic space $(L^\flat_F)^\perp\subseteq \mathbb{V}_m^\varepsilon$ is isotropic (i.e., $\chi((L^\flat_F)^\perp)=+1$); and \emph{co-anisotropic} otherwise. By Lemma \ref{lem:directsum} (\ref{item:l2}), we know that $L^\flat$ is co-isotropic if and only if $\chi(L^\flat)=\varepsilon$.
\item Say $L^\flat$ is \emph{horizontal} (in $\mathbb{V}_m^\varepsilon$), if $L^\flat$ is integral and one of the following is satisfied:
  \begin{enumerate}
  \item $t(L^\flat)\le1$,
  \item $t(L^\flat)=2$ and $\varepsilon \sgn_{n-1}(L^\flat)=+1$.
  \end{enumerate}
\end{enumerate}
\end{definition}

\begin{definition}\label{def:horLflat}
Let $L^\flat\subseteq\mathbb{V}_m^\varepsilon$ be an $O_F$-lattice of rank $n-1$. Denote by $\Hore(L^\flat)$  the set of horizontal lattices $M^\flat\subseteq L^\flat_F$ such that $L^\flat\subseteq M^\flat$. 
\end{definition}

Such horizontal lattices parametrize the horizontal parts of the special cycle $\mathcal{Z}(L^\flat)$, hence the name (see Theorem \ref{thm:horizontal}).

\begin{corollary}\label{cor:pDendiff}
 Take $m=n+1$ and $\varepsilon\in\{\pm1\}$. Let $L^\flat\subseteq \mathbb{V}_m^\varepsilon$ be a quadratic $O_F$-lattice of rank $n-1$ with fundamental invariants $(a_1,\cdots,a_{n-1})$.  
  \begin{altenumerate}
  \item\label{item:co-isoempty} If $L^\flat$ is co-isotropic, then $\Denfe(1, L^\flat)=0$ and $\Hore(L^\flat)=\varnothing$.
  \item Assume that $L^\flat$ is co-anisotropic. Let $L=L^\flat+\pair{x}$ and  $\wit L=L^\flat+\pair{\varpi^{-1} x}$ where $x\perp L^\flat$ with $\val(x)>a_{n-1}$. If $\chi(L^\flat)\ne0$, then $$\pDene(L)-\pDene(\wit L)=\Denfe(1, L^\flat)=\sum_{M^\flat\in \Hore(L^\flat)} q^{\lfloor \frac{\val(M^\flat)}{2}\rfloor}\cdot
     \begin{cases}
    1, & t(M^\flat)=0,\\
    (1+ q^{-1}), & t(M^\flat)=1, \\
    2(1+ q^{-1}), & t(M^\flat)=2.
  \end{cases}$$
  If $\chi(L^\flat)=0$, then $$\pDene(L)-\pDene(\wit L)=2\Denfe(1,L^\flat)=2\sum_{M^\flat\in \Hore(L^\flat)} q^{\lfloor \frac{\val(M^\flat)}{2}\rfloor}\cdot
     \begin{cases}
    1, & t(M^\flat)=1, \\
    2, & t(M^\flat)=2.
  \end{cases}$$
  \end{altenumerate}
\end{corollary}

\begin{proof}
  \begin{altenumerate}
  \item Since $\chi(L^\flat)=\varepsilon$, by Lemma \ref{lem:latticemebed} (\ref{item:embed2}) we know that $L^\flat$ cannot be embedded into $H_m^\varepsilon$ as a quadratic submodule. Hence $\Denfe(1, L^\flat)=0$ by Definition \ref{def:denfe}. By definition, if $t(M^\flat)=2$, then $1+\varepsilon \sgn_{n-1}(M^\flat)=2$ when $M^\flat$ is horizontal, and $0$ otherwise. Then the summation in Corollary \ref{cor:Denfe1L} can be written as over $M^\flat\in\Hore(L^\flat)$, and hence $\Denfe(1,L^\flat)=0$ implies that $\Hore(L^\flat)=\varnothing$.
  \item   Since $\chi(L^\flat)\ne\varepsilon$, we know that $$1-\varepsilon\chi(L^\flat)q^{-1}=
    \begin{cases}
      1+q^{-1}, & \chi(L^\flat)\ne0, \\
      1, &  \chi(L^\flat)=0.
    \end{cases}$$
Moreover, if $\chi(L^\flat)=0$, then $L_F^\flat$ does not admit self-dual lattices, so $t(M^\flat)\ne0$. By Lemma~\ref{lem:latticemebed} (\ref{item:embed1}), we have $w^\varepsilon(L)=-1$. The result then follows from Corollary \ref{cor:pDendiffgeneral} and Corollary \ref{cor:Denfe1L}.\qedhere
  \end{altenumerate}
\end{proof}

\begin{lemma}\label{lem:saturatedhorizontal}
Take $m=n+1$ and $\varepsilon\in\{\pm1\}$. Let $M^\flat$ be an $O_F$-lattice of rank $n-1$ which embeds simultaneously into $\mathbb{V}_m^\varepsilon$ and $H_m^\varepsilon$ as quadratic submodules. Assume that $M^\flat=M_F^\flat\cap H_{m}^\varepsilon$ (under the embedding of $M^\flat$ into $H_m^\varepsilon$). Then $M^\flat$ is horizontal (in $\mathbb{V}_m^\varepsilon$).
\end{lemma}

\begin{proof}
To simplify notation, write $M=M^\flat$ and $H=H_m^\varepsilon$ for short.  Since $M=M_F\cap H$, we know that $H/M$ is a free $O_F$-module of rank 2. We can choose $e_n, e_{n+1}\in H$ whose images in $H/M$ form an $O_F$-basis of $H/M$. Then $H=M+\langle e_n, e_{n+1}\rangle$. Choose an orthogonal basis $\{e_1,\ldots,e_{n-1}\}$ and let $u_i=(e_i,e_i)\in F^\times$. The fundamental matrix of the $O_F$-basis $\{e_1,\ldots, e_{n+1}\}$ of $H$ has the form $$T=
  \begin{pmatrix}
    u_1 & 0& 0& (e_1,e_n) & (e_1, e_{n+1})\\
    0& u_2 & 0& (e_2,e_n) &  (e_2, e_{n+1})\\
    0&0 & \ddots & \vdots & \vdots \\
    (e_n,e_1) & (e_n,e_2) &\cdots & (e_n,e_n) &  (e_n, e_{n+1})\\
    (e_{n+1},e_1) & (e_{n+1},e_2) &\cdots & (e_{n+1},e_n) &  (e_{n+1}, e_{n+1})
  \end{pmatrix}.$$ If $t(M)\ge3$, then at least three of $u_i$'s have strictly positive valuation, and hence the rank of $T$ mod $\varpi$ is at most $n$, contradicting that $H$ is self-dual. Hence $t(M)\le2$.

  Now assume $t(M)=2$. We would like to show that $\varepsilon \sgn_{n-1}(M)=1$. Let $M_{n-3}=\langle e_1,\ldots,e_{n-3}\rangle$, $M'=\langle e_{n-2},e_{n-1}\rangle$ and $H'=M'+\langle e_n, e_{n+1}\rangle$. Then we may choose the basis $\{e_1,\ldots,e_{n+1}\}$ such that $\val(u_i)=0$ for $i=1,\ldots, n-3$, $\val(u_i)>0$ for $i=n-2,n-1$, and $M_{n-3}$ is orthogonal to $M'$ and $H'$. Then $\varepsilon \sgn_{n-1}(M)$ is equal to 
  \begin{align*}
    \chi(H)\sgn_{n-1}(M)&=\chi(H' \obot M_{n-3})\sgn_{n-1}(M'\obot M_{n-3})\\
    &=\chi(H')\sgn_{2}(M') \qquad (\text{as } M_{n-3} \text{ is self-dual})\\
    &=\chi(H') \qquad (\sgn_2(M')=1 \text{ as } M' \text{ has rank 2 and type 2}).
  \end{align*}
On the other hand, the fundamental matrix of $O_F$-basis $\{e_{n-2},\ldots,e_{n+1}\}$ of $H'$ has the form $$T'=\begin{pmatrix}
    u_{n-2} & 0&  (e_{n-2},e_n) & (e_{n-2}, e_{n+1})\\
    0& u_{n-1} &  (e_{n-1},e_n) &  (e_{n-1}, e_{n+1})\\
    (e_n,e_{n-2}) & (e_n,e_{n-1}) & (e_n,e_n) &  (e_n, e_{n+1})\\
    (e_{n+1},e_{n-2}) & (e_{n+1},e_{n-1}) & (e_{n+1},e_n) &  (e_{n+1}, e_{n+1})
  \end{pmatrix}.$$ Since $\val(u_i)>0$ for $i=n-2,n-1$, we know that $$\det (T')=(\det B)^2\pmod{\varpi}$$ is a square, where $$B=
  \begin{pmatrix}
    (e_{n-2},e_n) & (e_{n-2}, e_{n+1})\\
    (e_{n-1},e_n) &  (e_{n-1}, e_{n+1})\\
  \end{pmatrix}.$$ Hence $\chi(H')=\chi((-1)^{4 \choose 2}\det T')=+1$ as desired.
\end{proof}

\section{Special cycles on GSpin Rapoport--Zink spaces}\label{sec:kudla-rapop-cycl}

In this section we take $F=\mathbb{Q}_p$. From now on we fix $m=n+1\ge3$ and $\varepsilon\in\{\pm1\}$.  To simplify notation we will suppress the superscript $\varepsilon$ and the subscripts $m$ and $n$ when there is no confusion (see Convention \S\ref{sec:conventions}). Let $V=V_m^\varepsilon$ be a self-dual quadratic $O_F$-lattice of rank $m$ with $\chi(V)=\varepsilon$.{\footnote{So $V\simeq H_{m}^\varepsilon$ as quadratic $O_F$-lattices. We use the symbol $V$ to emphasize its role on the geometric side.}

\subsection{GSpin Rapoport--Zink spaces $\RRZ_G$} \label{sec:gspin-rapoport-zink}
 Associated to $V$ we have a \emph{local unramified Shimura-Hodge data} $(G, b, \mu, C)$ (in the sense of \cite[Definition 2.2.4]{Howard2015}) constructed in \cite[Proposition 4.2.6]{Howard2015}, where $G=\GSpin(V)$, $b\in G(\Fb)$ is a basic element, $\mu: \mathbb{G}_m\rightarrow G$ is a certain cocharacter, and $C=C(V)$ is the Clifford algebra of $V$ (which has rank $2^m$). See \cite[\S4.1]{Howard2015} or \cite[\S2.1]{Li2018} for a review on GSpin groups. Let $D=\Hom_{O_F}(C, O_F)$ be the linear dual of $C$. By \cite[Lemma 2.2.5]{Howard2015}, this local unramified Shimura-Hodge data gives rise to a (unique up to isomorphism) $p$-divisible group $\mathbb{X}$ over $\kb$ whose contravariant Dieudonn\'e module $\mathbb{D}(\mathbb{X})(\OFb)$ is given by $D_{\OFb}$ with Frobenius $\Fr=b\circ \sigma$. Moreover, the Hodge filtration $\Fil^1\mathbb{D}(\mathbb{X})(\kb)\subseteq D_\kb$ is induced by a conjugate of $\mu_\kb$.

Let $(s_\alpha)_{\alpha\in I}$ be a finite set of tensors $s_\alpha$ in the total tensor algebra $C^\otimes$ cutting out $G$ from $\GL(C)$. Then we obtain tensors $t_{\alpha,0}=s_\alpha \otimes 1\in (D^\otimes)_{\OFb}=\mathbb{D}(\mathbb{X})(\OFb)^{\otimes}$, which are $\Fr$-invariant elements of $\mathbb{D}(\mathbb{X})(\OFb)^{\otimes}[1/p]$. Each $t_{\alpha,0}$ induces a \emph{crystalline Tate tensor} $t_{\alpha,0}$  on $\mathbb{X}$. We recall that a crystalline Tate tensor on a $p$-divisible group $X$ over $\Spec R$, where $R\in \ansm$, is a morphism of crystals $t: \mathbf{1}:=\mathbb{D}(F/O_F)\rightarrow \mathbb{D}(X)^\otimes$  such that $t(R): \mathbf{1}(R)\rightarrow \mathbb{D}(X)(R)^{\otimes}$ is compatible with Hodge filtrations, and the induced morphisms of isocrystals $t: \mathbf{1}[1/p]\rightarrow \mathbb{D}(X)^{\otimes}[1/p]$ is $\Fr$-equivariant.

Associated to the local unramified Shimura-Hodge data, we have a \emph{GSpin Rapoport--Zink space} $\RRZ_G=\RRZ(G, b, \mu, C)$ of Hodge type (\cite[\S 4.3]{Howard2015}, see also \cite{Kim2018}) parametrizing $p$-divisible groups with crystalline Tate tensors deforming $(\mathbb{X}, (t_{\alpha,0})_{\alpha\in I}))$. More precisely, it is a formal scheme over $\Spf \OFb$ representing the functor sending $R\in \ansm$ to the set of isomorphism classes of tuples $(X, (t_\alpha)_{\alpha\in I}, \rho)$, where
\begin{itemize}
\item $X$ is a $p$-divisible group over $\Spec R$.
\item $(t_\alpha)_{\alpha\in I}$ is a collection of crystalline Tate tensors of $X$. 
\item $\rho: \mathbb{X} \otimes_\kb R/J \rightarrow X \otimes_{R} R/J$ is a framing, i.e., a quasi-isogeny such that each $t_\alpha$ pulls back to $t_{\alpha,0}$ under $\rho$, where $J$ is some ideal of definition of $R$ such that $p\in J$.
\end{itemize}
The tuple $(X, (t_\alpha), \rho)$ is required to satisfy additional assumptions \cite[Definition 2.3.3 (ii), (iii)]{Howard2015}.

The GSpin Rapoport--Zink space $\RRZ_G$ is formally locally of finite type and formally smooth of relative dimension $n-1$ over $\Spf \OFb$ (\cite[Theorem B]{Howard2015}). 

\subsection{Space of special quasi-endomorphisms $\mathbb{V}=\mathbb{V}_m^\varepsilon$}\label{sec:space-special-quasi}

The inclusion $V\subseteq C^\mathrm{op}$ (where $V$ acts on $C$ via right multiplication) realizes $V\subseteq \End_{O_F}(D)$ as special endomorphisms of $D$. Tensoring with $\Fb$ gives a subspace $$V_\Fb\subseteq \End_{\Fb}(D_\Fb).$$ Define the $\sigma$-linear operator $\Phi=\bar b\circ\sigma$ on $V_\Fb$, where $\bar b\in \SO(V)(\Fb)$ is the image of $b\in G(\Fb)$ under the natural quotient map (the standard representation) $G=\GSpin(V)\rightarrow \SO(V)$. Then $(V_\Fb,\Phi)$ is an isocrystal. 
The $\Phi$-fixed vectors form a distinguished $F$-vector subspace $$V_\Fb^\Phi\subseteq \Endo(\mathbb{X}):=\End(\mathbb{X})[1/p],$$ called \emph{special quasi-endomorphisms} of $\mathbb{X}$. The restriction of the quadratic form to $V_\Fb^\Phi$ satisfies $x\circ x=(x, x)\cdot \id_\mathbb{X}$ for $x\in V_\Fb^\Phi$, and we have an isomorphism of quadratic spaces over $F$ (\cite[\S 4.3.1]{Howard2015}) $$\mathbb{V}=\mathbb{V}_m^\varepsilon\xrightarrow{\sim}V_\Fb^\Phi.$$

\subsection{The projectors $\bpi_\crys$ and the crystal $\mathbf{V}_{\crys}$}\label{sec:proj-bpi-bpi_crys} The left action of $V$ on $C=C(V)$ gives a $G=\GSpin(V)$-equivariant embedding $V\hookrightarrow\End(C):=\End_{O_F}(C)$, where $G$ acts on $\End(C)$ via $g.f=gfg^{-1}$ for any $g\in G$ and $f\in \End(C)$. We identify $V\subseteq \End(C)$ via this embedding. Then $$f\circ f=(f,f)\cdot \id_C,\quad f\in V.$$ The non-degenerate symmetric bilinear form $(\ , \ )$ on $\End(C)$ defined by $$(f_1, f_2):=2^{-m}\tr(f_1\circ f_2),\quad f_1,f_2\in \End(C)$$ extends $(\ , \ )$ on $V$. Let $\{e_1,\ldots, e_m\}$ be an orthogonal $O_F$-basis of $V$.  Define $$\bpi: \End(C)\rightarrow \End(C),\quad \bpi(f):=\sum_{i=1}^m\frac{(f,e_i)}{(e_i,e_i)}e_i.$$ It is clear that $\bpi$ is an idempotent with image $V\subseteq \End(C)$ and $G$ stabilizes $\bpi$ (cf. \cite[Lemma 1.4]{MadapusiPera2016}). Thus we can choose the list of tensors $(s_\alpha)_{\alpha\in I}$ in $C^\otimes$ cutting out $G$ from $\GL(C)$ (\S\ref{sec:gspin-rapoport-zink})  to include the projector $\bpi$. In this way we obtain a crystalline Tate tensor $\bpi_{\crys,0}$ on the framing object $\mathbb{X}$. By construction, it induces a projector of crystals $$\bpi_{\crys,0}: \End(\mathbb{D}(\mathbb{X}))\rightarrow \End(\mathbb{D}(\mathbb{X})),$$ whose image $\mathbf{V}_{\crys,0}:=\im(\bpi_{\crys,0})$ is a crystal of $\mathcal{O}^\crys_{\kb/\OFb}$-modules of rank $m$. For any surjection $R\rightarrow \kb$ in $\alg$ whose kernel admits divided powers, we have $\mathbf{V}_{\crys,0}(R)$ a projective $R$-module of rank $m$. It is equipped with a non-degenerate symmetric $R$-bilinear form $(\ ,\  )$ induced from that on $\End(\mathbb{D}(\mathbb{X})(R))$. Moreover, the $\kb$-space $\mathbf{V}_{\crys,0}(\kb)$ is equipped with a Hodge filtration defined by $$\Fil^1\mathbf{V}_{\crys,0}(\kb):=\mathbf{V}_{\crys,0}(\kb)\cap\Fil^1\End(\mathbb{D}(\mathbb{X})(\kb)),$$ which is an isotropic $\kb$-line under $(\ ,\ )$. By construction, the F-isocrystal $\mathbf{V}_{\crys,0}(\OFb)[1/p]$ can be identified with $(V_{\Fb},\Phi)$.

 By \cite[Theorem 4.9.1]{Kim2018} we obtain from $\bpi$ a universal crystalline Tate tensor $\bpi_\crys$ on the universal $p$-divisible group $X^\mathrm{univ}$ over $\RRZ_G$, which induces a projector of crystals $$\bpi_\crys: \End(\mathbb{D}(X^\mathrm{univ}))\rightarrow \End(\mathbb{D}(X^\mathrm{univ}))$$ whose image $\mathbf{V}_\crys:=\im(\bpi_\crys)$ is a crystal of $\mathcal{O}_{\RRZ_G/\OFb}^\crys$-modules of rank $m$.

More generally, for any $S\in\alg$ and any $z\in \RRZ_G(S)$, we similarly have a projector of crystals $$\bpi_{\crys,z}: \End(\mathbb{D}(X_z))\rightarrow \End(\mathbb{D}(X_z)),$$ whose image $\mathbf{V}_{\crys,z}:=\im(\bpi_{\crys,z})$ is a crystal of $\mathcal{O}_{S/\OFb}^\crys$-modules of rank $m$. Here  $X_z$ denotes the $p$-divisible group over $S$ obtained by the base change of $X^\mathrm{univ}$ to $z$. Similarly, for any surjection $R\rightarrow S$ in $\alg$ whose kernel admits divided powers, we have $\mathbf{V}_{\crys,z}(R)$ a projective $R$-module of rank $m$. It is equipped with a non-degenerate symmetric $R$-bilinear form $(\ ,\  )$. The projective $S$-module  $\mathbf{V}_{\crys,z}(S)$ is equipped with a Hodge filtration $\Fil^1\mathbf{V}_{\crys,z}(S)$, which is an isotropic $S$-line. 

\subsection{The group $J$}\label{sec:group-j}

Let $J=J_b:=\GSpin(\mathbb{V})$. It is a reductive group over $F$ and an inner form of $G=\GSpin(V)$ (as $b$ is basic). Then $$J(F)=\{g\in G(\Fb): gb=b\sigma(g)\}$$ is the $\sigma$-centralizer of $b$, which acts on the framing object $(\mathbb{X}, (t_{\alpha,0}))$ via $J(F)\subseteq \Endo(\mathbb{X})^\times$ (\cite[\S 4.3.4]{Howard2015}), and hence acts on $\RRZ_G$.

\subsection{Connected GSpin Rapoport--Zink spaces $\mathcal{N}=\mathcal{N}_n^\varepsilon$}\label{sec:conn-gspin-rapop}

We may choose one of the tensor $(s_\alpha)$ to be a non-degenerate symplectic form $\psi$ on $C$ such that $G$ is a subgroup of $\mathrm{GSp}(C, \psi)$. Then $\psi$ induces a principal polarization $\lambda_0: \mathbb{X}\xrightarrow{\sim}\mathbb{X}^\vee$, and we have an associated symplectic Rapoport--Zink space $\RRZ(\mathbb{X},\lambda_0)$ parameterizing deformations of $(\mathbb{X},\lambda_0)$ as considered in \cite{RZ96}. Denote by $(X^\mathrm{univ}, \lambda^\mathrm{univ})$ the universal object over $\RRZ(\mathbb{X},\lambda_0)$. Then we have a closed immersion $\RRZ_G\hookrightarrow \RRZ(\mathbb{X},\lambda_0)$, and we still denote by $(X^\mathrm{univ}, \lambda^\mathrm{univ})$ the restriction of the universal object to $\RRZ_G$. The universal quasi-isogeny $\rho^\mathrm{univ}$ respects the polarizations $\lambda_0$ and $\lambda^\mathrm{univ}$ up to a scalar, and hence Zariski locally on $\RRZ_G$ we have  $\rho^{\mathrm{univ},*}(\lambda^\mathrm{univ})=c(\rho^{\mathrm{univ}})^{-1}\lambda_0$ for some $c(\rho^\mathrm{univ})\in F^\times$. We have a decomposition of $\RRZ_G$ into connected components (\cite[Theorem D (i)]{Howard2015}   $$\RRZ_G=\bigsqcup_{\ell\in \mathbb{Z}}\RRZ_G^{(\ell)},$$ where $\val(c(\rho^\mathrm{univ}))=\ell$ on $\RRZ_G^{(\ell)}$. Each $g\in J(F)$ restricts to an isomorphism $\RRZ_G^{(\ell)}\xrightarrow{\sim}\RRZ_G^{(\ell+\val(\eta(g)))}$, where $\eta: J\rightarrow \mathbb{G}_m$ is the spinor norm. The surjectivity of $\eta: J(F)\rightarrow F^\times$ implies that $\RRZ_G^{(\ell)}$ ($\ell\in \mathbb{Z}$) are all (non-canonically) isomorphic to each other.

\begin{definition}\label{def:conn-gspin-rapop-1}
Define the \emph{connected GSpin Rapoport--Zink space} $\mathcal{N}=\mathcal{N}_{n}^{\varepsilon}:=\RRZ_G^{(0)}$, a connected component of $\RRZ_G$, which has (total) dimension $m-1=n$ (hence the notation).  
\end{definition}

The space $\mathcal{N}$ is related to the Rapoport--Zink space of abelian type associated to the group $\SO(V)$ (see for example \cite[\S8.1]{She20}). We still denote by $(X^\mathrm{univ}, \lambda^\mathrm{univ})$ the restriction of the universal object to $\mathcal{N}$. In particular, on $\mathcal{N}$ we have $\rho^{\mathrm{univ},*}(\lambda^\mathrm{univ})=c(\rho^{\mathrm{univ}})^{-1}\lambda_0$ for some $c(\rho^{\mathrm{univ}})\in O_F^\times$.

\begin{remark}
 When $m$ is odd, scaling the quadratic form gives an isomorphism $\GSpin(V_m^+)\simeq \GSpin(V_m^-)$, which induces an isomorphism  $\mathcal{N}_{n}^+\simeq\mathcal{N}_{n}^-$. In other words, $\mathcal{N}_{n}^\varepsilon$ is independent of $\varepsilon\in\{\pm1\}$ when $m$ is odd, and we sometimes simply write $\mathcal{N}_n$ in this case.
\end{remark}

\begin{example}\label{exa:smallm} When $m$ is small, the exceptional isomorphisms between $\GSpin(V)$ and other classical groups induces isomorphisms between the space $\mathcal{N}$ and several classical formal moduli spaces (by for example \cite[p. 1215]{BP20}):
  \begin{altenumerate}
  \item When $m=n+1=3$, we have an exceptional isomorphism $G=\GSpin(V_3^\varepsilon)\simeq\GL_2$ as reductive groups over $O_F$, and $\mathcal{N}_2\simeq\Spf \OFb[[t]]$ is isomorphic to the Lubin--Tate deformation space of the formal group $\mathbb{E}$ of dimension 1 and height 2 over $\kb$. 
  \item When $m=n+1=4$  and $\varepsilon=+1$, we have an exceptional isomorphism $G=\GSpin(V_4^+)\simeq \GL_2\times_{\mathbb{G}_m} \GL_2$ as reductive groups over $O_F$, and $\mathcal{N}_3^+\simeq \Spf \OFb[[t_1,t_2]]$ is isomorphic to the product of two copies of Lubin--Tate deformation spaces $\mathcal{N}_2$ over $\Spf \OFb$.
  \item When $m=n+1=4$ and $\varepsilon=-1$, we have an exceptional isomorphism $\GSpin(V_4^-)\simeq \GL_{2,O_E}^{\det\in O_F^\times}$ (where $E/F$ is the unramified quadratic extension), and $\mathcal{N}_3^-$ is isomorphic to the formal moduli space $\mathcal{M}^\mathrm{HB}$ of principally polarized supersingular $p$-divisible groups of dimension 2 and height 4 with a special $O_E$-action defined in \cite[\S2]{Terstiege2011}. The space $\mathcal{N}_3^-$ appears in the $p$-adic uniformization of the supersingular locus of Hilbert modular surfaces at a good inert place $p$.
  \item When $m=n+1=5$, we have an exceptional isomorphism $\GSpin(V_5^\varepsilon)\simeq \GSp_4$ as reductive groups over $O_F$, and $\mathcal{N}_4$ is isomorphic to the formal moduli space of principally polarized supersingular $p$-divisible groups of dimension 2 and height 4. The space $\mathcal{N}_4$ appears in the $p$-adic uniformization of the supersingular locus of Siegel modular threefolds at a good place $p$.
  \end{altenumerate}
\end{example}

\subsection{Special cycles $\mathcal{Z}(L)$}\label{sec:kudla-rapop-cycl-1}

\begin{definition}\label{def:kudla-rapop-cycl-2}
For any subset $L\subseteq \mathbb{V}$, define the \emph{special cycle} $\mathcal{Z}(L)\subseteq \mathcal{N}$ to be the closed formal subscheme cut out by the condition $$\rho^\mathrm{univ}\circ x \circ (\rho^\mathrm{univ})^{-1}\subseteq \End(X^\mathrm{univ})$$ for all $x\in L$.  Notice that $\mathcal{Z}(L)$ only depends on the $O_F$-linear span of $L$ in $\mathbb{V}$, and is nonempty only when this span is an integral $O_F$-lattice (of arbitrary rank) in $\mathbb{V}$. Notice that a similar definition of special cycles applies to the Rapoport--Zink space $\RRZ_G$ (instead of $\mathcal{N}$).
\end{definition}

Let $R\rightarrow S$ be a surjection in $\alg$ whose kernel admits divided powers.  By definition,  for any $z\in \mathcal{Z}(L)(S)$ and any $x\in L$,  the crystalline realization $x_{\crys,z}(R)\in \End(\mathbb{D}(X_z)(R))$ of $x\in \End(X_z)$ lies in the image of $\bpi_{\crys,z}(R)$, and hence induces an element $x_{\crys,z}(R)\in \mathbf{V}_{\crys,z}(R)$. We denote by $L_{\crys,z}(R)\subseteq \mathbf{V}_{\crys,z}(R)$ the $R$-submodule spanned by $x_{\crys,z}(R)$ for all $x\in L$.

\begin{lemma}\label{lem:GM} Let $R\rightarrow S$ be a surjection in $\alg$ whose kernel admits nilpotent divided powers.   
  \begin{altenumerate} 
  \item Let $z_0\in \mathcal{N}(\kb)$ and $\wh{\mathcal{N}}_{z_0}$ be the completion of $\mathcal{N}$ at $z_0$. Let $z\in \mathcal{N}(S)$. Then there is a natural bijection $$\{\text{Lifts } \tilde z\in \wh{\mathcal{N}}_{z_0}(R)\text{ of }z\}\isoarrow\left\{\text{isotropic $R$-lines } \Fil^1\mathbf{V}_{\crys,z}(R) \text{ lifting } \Fil^1\mathbf{V}_{\crys,z}(S)\right\}.$$
  \item Let $L\subseteq \mathbb{V}$ be an $O_F$-lattice of rank $r\ge1$. Let $z_0\in \mathcal{Z}(L)(\kb)$ and $\wh{\mathcal{Z}(L)}_{z_0}$ be the completion of $\mathcal{Z}(L)$ at $z_0$. Let $z\in \mathcal{Z}(L)(S)$.  Then  there is a natural bijection $$\{\text{Lifts } \tilde z\in \wh{\mathcal{Z}(L)}_{z_0}(R)\text{ of }z\}\isoarrow\left\{\begin{array}{c}
      \text{isotropic $R$-lines } \Fil^1\mathbf{V}_{\crys,z}(R) \text{ lifting } \Fil^1\mathbf{V}_{\crys,z}(S)\\
      \text{ and orthogonal to } L_{\crys,z}(R)\subseteq \mathbf{V}_{\crys,z}(R)
\end{array}
\right\}. $$ 
  \end{altenumerate}
\end{lemma}

\begin{proof}
This is a consequence of Grothendieck--Messing theory. The same proof of \cite[Proposition 5.16]{MadapusiPera2016} (see also \cite[Proposition 4.3.2]{Andreatta2018}) works.
\end{proof}

\subsection{Generalized Deligne--Lusztig varieties $Y_W$}\label{sec:gener-deligne-luszt}

Let $W$ be the unique (up to isomorphism) non-split (non-degenerate) quadratic space over $\kappa=\mathbb{F}_q$ of even dimension $2(d+1)$ ($d\ge0$). The closed $\mathbb{F}_{q^2}$-subvariety of the orthogonal Grassmannian $\OGr_{d+1}(W)$ parameterizing totally isotropic subspaces $U\subseteq W$ of dimension $d+1$ such that $\dim_\kappa (U+ \sigma(U))=d+2$ has two isomorphic connected components. Define $Y_W$ to be one of the two connected components. It is a smooth projective variety of dimension $d$, and has a locally closed stratification $$Y_W=\bigsqcup_{i=0}^{d} X_{P_i}(w_i),$$ where each $X_{P_i}(w_i)$ is a generalized Deligne--Lusztig variety of dimension $i$ associated to a certain parabolic subgroup $P_i\subseteq\SO(W)$ (\cite[Proposition 3.8]{Howard2014}). The open stratum $Y_W^\circ\coloneqq X_{P_{d}}(w_{d})$ is a classical Deligne--Lusztig variety associated to a Borel subgroup $P_{d}\subseteq \SO(W)$ and a Coxeter element $w_{d}$. Each of the other strata $X_{P_i}(w_i)$ is also isomorphic to a parabolic induction of a classical Deligne--Lusztig variety of Coxeter type for a Levi subgroup of $\SO(W)$ (\cite[Proposition 2.5.1]{He2019}). For example, when $d=1$, the Deligne--Lusztig curve $Y_W\cong \mathbb{P}^1$ and its open stratum is given by $Y_W^\circ=\mathbb{P}^1-\mathbb{P}^1(\mathbb{F}_{q^2})$ (\cite[\S3.3]{Howard2014}).

\subsection{Minuscule special cycles $\mathcal{V}(\Lambda)$}\label{sec:minusc-kudla-rapop}

Let $\Lambda\subseteq\mathbb{V}=\mathbb{V}_m^\varepsilon$ be a vertex lattice. Then $W_\Lambda\coloneqq \Lambda^\vee/\Lambda$ is a $\kappa$-vector space of dimension $t(\Lambda)$, equipped with a non-degenerate quadratic form induced from $\mathbb{V}$.  By \cite[5.1.2]{Howard2015}, the type $t(\Lambda)$ of a vertex lattice $\Lambda\subseteq \mathbb{V}$ is always an even integer such that $2\le t(\Lambda) \le t_\mathrm{max}$, where
\begin{equation}
  \label{eq:tmax}t_\mathrm{max}=
  \begin{cases}
  m-2, & \text{if }m\text{ is even and } \varepsilon=+1, \\
  m-1, & \text{if }m\text{ is odd}, \\
  m, & \text{if }m\text{ is even and } \varepsilon=-1.
\end{cases}
\end{equation}
Thus we have the associated generalized Deligne--Lusztig variety $Y_{W_\Lambda}$ of dimension $t(\Lambda)/2-1$. The reduced subscheme of the minuscule special cycle $\mathcal{V}(\Lambda)\coloneqq \mathcal{Z}(\Lambda)^\mathrm{red}$ is isomorphic to $Y_{W_\Lambda,\kb}$\footnote{Notice that $\RRZ_\Lambda$ in \cite{Howard2015} and \cite{Li2018} is the same as our $\mathcal{V}(\Lambda^\vee)$.}. In fact $\mathcal{Z}(\Lambda)$ itself is already reduced (\cite[Theorem B]{Li2018}), so $\mathcal{V}(\Lambda)=\mathcal{Z}(\Lambda)$.

\subsection{The Bruhat--Tits stratification on $\mathcal{N}^\mathrm{red}$}\label{sec:bruh-tits-strat}

The reduced subscheme of $\mathcal{N}$ satisfies $\mathcal{N}^\mathrm{red}=\bigcup_{\Lambda} \mathcal{V}(\Lambda)$, where $\Lambda$ runs over all vertex lattices $\Lambda\subseteq \mathbb{V}$. For two vertex lattices $\Lambda, \Lambda'$, we have $\mathcal{V}(\Lambda)\subseteq \mathcal{V}(\Lambda')$ if and only if $\Lambda\supseteq \Lambda'$; and $\mathcal{V}(\Lambda)\cap \mathcal{V}(\Lambda')$ is nonempty if and only if $\Lambda+\Lambda'$ is also a vertex lattice, in which case it is equal to $\mathcal{V}(\Lambda+\Lambda')$. In this way we obtain a \emph{Bruhat--Tits stratification} of $\mathcal{N}^\mathrm{red}$ by locally closed subvarieties (\cite[\S 6.5]{Howard2015}), $$\mathcal{N}^\mathrm{red}=\bigsqcup_{\Lambda}\mathcal{V}(\Lambda)^\circ, \quad \mathcal{V}(\Lambda)^\circ\coloneqq \mathcal{V}(\Lambda)-\bigcup_{\Lambda\subsetneq \Lambda'}\mathcal{V}(\Lambda').$$ Each Bruhat--Tits stratum $\mathcal{V}(\Lambda)^\circ\simeq Y_{W_\Lambda,\kb}^\circ$ is a classical Deligne--Lusztig of Coxeter type associated to $\SO(W_\Lambda)$, which has dimension $t(\Lambda)/2-1$. It follows that the irreducible components of $\mathcal{N}^\mathrm{red}$ are exactly the projective varieties $\mathcal{V}(\Lambda)$, where $\Lambda$ runs over all vertex lattices of maximal type (\cite[Theorem D (iii)]{Howard2015}).

By \cite[Proposition 4.2]{Soylu2017}, the reduced subscheme $\mathcal{Z}(L)^\mathrm{red}$ of a special cycle $\mathcal{Z}(L)$ is a union of Bruhat--Tits strata,
\begin{equation}
  \label{eq:KRstrat}
\mathcal{Z}(L)^\mathrm{red}=\bigcup_{L\subseteq\Lambda}\mathcal{V}(\Lambda).  
\end{equation}

\subsection{Special divisors $\mathcal{Z}(x)$}

\begin{proposition}\label{prop:relativecartier}
  Let $x\in \mathbb{V}$ be nonzero and integral. Then $\mathcal{Z}(x)\subseteq \mathcal{N}$ is a Cartier divisor (i.e., defined locally by one nonzero equation) and moreover is flat over $\Spf \OFb$ (i.e., locally the equation is not divisible by $\varpi$). 
\end{proposition}

\begin{proof}
  Since $x$ is integral, we know that $\mathcal{Z}(x)$ is nonempty by (\ref{eq:KRstrat}). Let $z\in \mathcal{Z}(x)(\bar k)$. Let $\hat O_z$ be the complete local ring of $\mathcal{N}$ at $z$ with maximal ideal $\mathfrak{m}$. Let $J\subseteq \hat O_z$ be the ideal defining the completion of $\mathcal{Z}(x)$ at $z$. Let $R=\hat O_z/\mathfrak{m}J$ and $I=J/\mathfrak{m}J$. Then $R\in\alg$ and $I^2=0$ (hence $I$ admits trivial nilpotent divided powers). By Nakayama's lemma, to show that $J$ is principal it suffices to show that $I$ is principal. It remains to show that the condition that a lifting of $\tilde z\in\mathcal{N}(R)$ of $z\in \mathcal{Z}(x)(S)$ lies in $\mathcal{Z}(x)(R)$ is given by the vanishing of one nonzero element in $I$.

  By Lemma \ref{lem:GM}, the lift $\tilde z$ corresponds to an isotropic $R$-line $\Fil^1\mathbf{V}_{\crys,z}(R)$ lifting $\Fil^1\mathbf{V}_{\crys,z}(S)$. Since $z\in \mathcal{Z}(x)(S)$, we know that the pairing $(\Fil^1\mathbf{V}_{\crys,z}(R),x_{\crys,z}(R))$ vanishes modulo $I$. The condition that $\tilde z\in \mathcal{Z}(x)(R)$ is equivalent to the vanishing of the pairing $(\Fil^1\mathbf{V}_{\crys,z}(R),x_{\crys,z}(R))$, which amounts to the vanishing of one element of $I$ after choosing a basis of $\Fil^1\mathbf{V}_{\crys,z}(R)$.

It remains to show that this element is nonzero and is not divisible by $\varpi$. Since $\mathcal{N}$ is connected, it suffices to show that $\mathcal{Z}(x)_\kb$ is not the whole special fiber $\mathcal{N}_\kb$.  If not, then $\mathcal{Z}(x)_\kb=\mathcal{N}_\kb$ and hence $\mathcal{Z}(x)^\mathrm{red}=\mathcal{N}^\mathrm{red}$. By (\ref{eq:KRstrat}), we know that $\mathcal{Z}(x)^\mathrm{red}=\mathcal{N}^\mathrm{red}$ if and only if $x\in \Lambda$ for any vertex lattice $\Lambda$  of maximal type $t(\Lambda)=t_\mathrm{max}$.  By the proof of \cite[Proposition 5.1.2]{Howard2015}, such a $\Lambda$ admits a decomposition $$\Lambda=\langle e_1,f_1\ldots, e_{d}, f_{d}\rangle \obot Z,$$ where $d=t_\mathrm{max}/2-1$, $(e_i,e_j)=(f_i,f_j)=0$, $(e_i,f_j)=\varpi \delta_{i,j}$ and $Z$ is anisotropic of rank $m-2d$. If $d>0$, then we may choose $\Lambda$ such that $x$ lies in the subspace $\langle e_1, f_1,\ldots, e_d, f_d\rangle_F$. For any $k\in  \mathbb{Z}$, the lattice $$\Lambda(k):=\langle \varpi^k e_1,\varpi^{-k}f_1\ldots, \varpi^ke_{d}, \varpi^{-k}f_{d}\rangle \obot Z $$ is also a vertex lattice of maximal type, but clearly $x\not\in \Lambda(k)$ for some $|k|\gg0$, a contradiction. Hence  if $d>0$ then $\mathcal{Z}(x)_\kb\ne\mathcal{N}_\kb$.

  If $d=0$, then by (\ref{eq:tmax}) either $m=3$ or $m=4$ and $\varepsilon=+1$, and $\mathcal{N}$ is isomorphic to the Lubin--Tate deformation space of the formal group of dimension 1 and height 2 over $\kb$, or the product of two copies of the Lubin--Tate deformation space (cf. Example \ref{exa:smallm}). In both cases the special endomorphism $x$ ensures that the universal $p$-divisible group over $\mathcal{Z}(x)_\kb$ is non-ordinary, while it is well-known that the universal $p$-divisible group over the generic point of $\mathcal{N}_\kb$ is ordinary (e.g., by the $p$-adic uniformization theorem and the fact that ordinary points are dense on a modular curve or on a product of two modular curves). Therefore if $d=0$ then $\mathcal{Z}(x)_\kb\ne\mathcal{N}_\kb$. 
\end{proof}

\subsection{Derived special cycles $\LZ(L)$}\label{sec:deriv-kudla-rapop}

Using the flatness of $\mathcal{Z}(x_i)$ (Proposition \ref{prop:relativecartier}), the following linear independence of derived intersection (Lemma \ref{lem:two div}, Corollary \ref{cor:ind L}) is proved in the the same way as \cite[Lemma 4.1, Proposition 4.2]{Terstiege2011}\footnote{In \cite[Lemma 2.8.1]{LZ2019} we also provided an alternative proof (without using globalization) of the linear independence  in the unitary case. That alternative proof would still work for the orthogonal case in the current paper, as long as the orthogonal analogue of the main result of \cite{Terstiege2013b} (the regularity of difference divisors) is available.}.

\begin{lemma}\label{lem:two div}
Let $x,y\in\BV$ be linearly independent. Then the tor sheaves $\Tor_i^{\CO_{\CN}}(\CO_{\CZ(x)}, \CO_{\CZ(y)})$ vanish for all $i\geq 1$. In particular,
$$
\mathcal{O}_{\mathcal{Z}(x)} \otimes^\mathbb{L} \mathcal{O}_{\mathcal{Z}(y)}=\mathcal{O}_{\mathcal{Z}(x)} \otimes \mathcal{O}_{\mathcal{Z}(y)}.
$$
Here $\mathcal{O}_{\mathcal{Z}(x_i)}$ denotes the structure sheaf of the special divisor $\mathcal{Z}(x_i)$ and $\otimes^\mathbb{L}$ denotes the derived tensor product of coherent sheaves on $\mathcal{N}$.
\end{lemma}

\begin{corollary}\label{cor:ind L}Let $L\subseteq \mathbb{V}$ be an $O_F$-lattice of rank $r\ge1$. Let $x_1,\ldots, x_{r}$ be an $O_F$-basis of $L$. Then $\mathcal{O}_{\mathcal{Z}(x_1)} \otimes^\mathbb{L}\cdots \otimes^\mathbb{L}\mathcal{O}_{\mathcal{Z}(x_r)}\in K_0^{\mathcal{Z}(L)}(\mathcal{N})$ is independent of the choice of the basis.
\end{corollary}

Let $L\subseteq \mathbb{V}$ be an $O_F$-lattice of rank $r\ge1$. Let $x_1,\ldots, x_{r}$ be an $O_F$-basis of $L$. Since each $\mathcal{Z}(x_i)$ is a Cartier divisor on $\mathcal{N}$ (Proposition \ref{prop:relativecartier}), we know that $\mathcal{O}_{\mathcal{Z}(x_i)}\in \mathrm{F}^1K_0^{\mathcal{Z}(x_i)}(\mathcal{N})$ (see \S\ref{sec:notat-form-schem}), and hence by \cite[(B.3)]{Zhang2019} we obtain $$\mathcal{O}_{\mathcal{Z}(x_1)} \otimes^\mathbb{L}\cdots \otimes^\mathbb{L}\mathcal{O}_{\mathcal{Z}(x_r)}\in \mathrm{F}^{r}K_0^{\mathcal{Z}(L)}(\mathcal{N}).$$ By Corollary \ref{cor:ind L}, this is independent of the choice of the basis $x_1,\ldots, x_n$ and hence is a well-defined invariant of $L$ itself.

\begin{definition}
Define the \emph{derived special cycle} 
$^\BL\CZ(L)$ to be the image of $\mathcal{O}_{\mathcal{Z}(x_1)} \otimes^\mathbb{L}\cdots \otimes^\mathbb{L}\mathcal{O}_{\mathcal{Z}(x_r)}$ in the $r$-th graded piece $\Gr^{r}K_0^{\CZ(L)}(\mathcal{N})$. 
\end{definition}

\begin{definition}\label{def:Int(L)}
  When the $O_F$-lattice $L\subseteq \mathbb{V}$ has rank $r=n$, define the \emph{arithmetic intersection number} \begin{align}\label{eq:def Int}
\Int(L)\coloneqq \chi\bigl(\mathcal{N},{}^\BL\CZ(L)),
\end{align}  where  $\chi$ denotes the Euler--Poincar\'e characteristic (\S\ref{sec:notat-form-schem}). Notice that if $L$ is not integral then $\mathcal{Z}(L)$ is empty and hence $\Int(L)=0$.
\end{definition}

\begin{remark}\label{rem:Intisometry}
  If $L, L'\subseteq \mathbb{V}$ are isometric $O_F$-lattices of rank $n$, then we may find $g\in \SO(\mathbb{V})(F)$ such that $L=gL'$. We may further lift $g$ to $\tilde g\in \GSpin(\mathbb{V})(F)=J(F)$ such that $\eta(g)=1$. Then the automorphism $\tilde g$ of $\mathcal{N}$ (see \S\ref{sec:group-j}, \ref{sec:conn-gspin-rapop}) carries $\LZ(L')$ to $\LZ(L)$. In particular, $\Int(L)$ only depends on the isometry class of $L$.
\end{remark}

\subsection{Horizontal and vertical parts of $\mathcal{Z}(L)$}\label{sec:horiz-vert-parts}

\begin{definition}
  A formal scheme $Z$ over $\Spf \OFb$ is called \emph{vertical} (resp. \emph{horizontal}) if $\varpi$ is locally nilpotent on $Z$ (resp. flat over $\Spf \OFb$). Clearly the formal scheme-theoretic union of two vertical (resp. horizontal) formal subschemes of a formal scheme is also vertical (resp. horizontal).

  We define the \emph{horizontal part} $Z_\sH\subseteq Z$ to be the closed formal subscheme defined by the ideal sheaf $\mathcal{O}_{Z}[\varpi^\infty]\subseteq \mathcal{O}_Z$. Then $Z_\sH$ is the maximal horizontal closed formal subscheme of $Z$. 

  When $Z$ is noetherian, there exists $N\gg0$ such that $\varpi^N \mathcal{O}_{Z}[\varpi^\infty]=0$, and we define the \emph{vertical part} $Z_\sV\subseteq Z$ to be the closed formal subscheme defined by the ideal sheaf $\varpi^N\mathcal{O}_Z$. Since $\mathcal{O}_{Z}[\varpi^\infty]\cap \varpi^N\mathcal{O}_Z=0$, we have a decomposition  $$Z=Z_\sH\cup Z_\sV,$$ as a union of horizontal and vertical formal subschemes. Notice that the horizontal part $Z_\sH$ is canonically defined, while the vertical part $Z_\sV$ depends on the choice of $N$.

  When $Z_\sH=\varnothing$, we simply define the \emph{vertical part} $Z_\sV$ to be the entire $Z$.
\end{definition}

\begin{lemma}\label{lem:ZLnoetherian}
  Let $L\subseteq \mathbb{V}$ be an $O_F$-lattice of rank $r$.
  \begin{altenumerate}
  \item\label{item:noetherian} If $r=n$, or $r=n-1$ and $L$ is co-anisotropic, then $\mathcal{Z}(L)$ is noetherian.
  \item\label{item:horempty} If $r=n$, or $r=n-1$ and $L$ is co-isotropic, then $\mathcal{Z}(L)_\sH=\varnothing$.
  \end{altenumerate}

\end{lemma}

\begin{proof}
  \begin{altenumerate}
  \item As a closed formal subscheme of the locally noetherian formal scheme $\mathcal{N}$, we know that $\mathcal{Z}(L)$ is locally noetherian. By the assumption on $L$, the space $L_F^\perp$ is anisotropic and thus its integral cone $(L_F^\perp)^{\circ}$ is an $O_F$-lattice. Take $L_1=L\obot \varpi^{k} (L_F^\perp)^{\circ}$ for $k\gg0$ depending only on $L$. Then any vertex lattice $\Lambda\subseteq \mathbb{V}$ containing $L$ satisfies $L_1\subseteq \Lambda\subseteq L_1^\vee$. Hence the number of vertex lattices $\Lambda$ containing $L$ is finite. By (\ref{eq:KRstrat}), we know that $\mathcal{Z}(L)^\mathrm{red}$ is a closed subset in finitely many irreducible components of $\mathcal{N}^\mathrm{red}$. Since each irreducible component of $\mathcal{N}^\mathrm{red}$ is quasi-compact, we know that $\mathcal{Z}(L)$ is quasi-compact, hence noetherian.
  \item This follows immediately from Lemma \ref{lem:latticemebed} (\ref{item:embed1}) (when $r=n$), Lemma \ref{lem:latticemebed} (\ref{item:embed2}) (when $r=n-1$) and Corollary \ref{cor:ZMhoritonal} below. \qedhere
  \end{altenumerate}
\end{proof}

By Lemma \ref{lem:ZLnoetherian}, for $L\subseteq \mathbb{V}$ an $O_F$-lattice of rank $r\ge n-1$, we obtain a decomposition of the special cycle  into \emph{horizontal} and \emph{vertical} parts $$\mathcal{Z}(L)=\mathcal{Z}(L)_\sH\cup \mathcal{Z}(L)_\sV.$$ Again when $\mathcal{Z}(L)_{\sH}\ne\varnothing$, the vertical part $\mathcal{Z}(L)_\sV$ depends on the choice of an integer $N\gg 0$. Since the choice of $N$ is not important for our purpose we suppress it from the notation (cf. \S\ref{sec:horiz-vert-parts-1}).

\subsection{Finiteness of $\Int(L)$}

\begin{lemma}\label{lem:properscheme}
Let $L\subseteq \mathbb{V}$ be an $O_F$-lattice of rank $n$. Then the formal scheme $\mathcal{Z}(L)$ is a proper scheme over $\Spf \OFb$. In particular,   $\Int(L)$ is finite.
\end{lemma}

\begin{proof}
  The vertical part $\mathcal{Z}(L)_\sV$ is a scheme by Lemma \ref{prop:supp Z(L) V} below. By Lemma \ref{lem:ZLnoetherian} (\ref{item:horempty}), the horizontal part $\mathcal{Z}(L)_\sH$ is empty, and so $\mathcal{Z}(L)$ is a scheme. By Lemma \ref{lem:ZLnoetherian} (\ref{item:noetherian}), we know that $\mathcal{Z}(L)^\mathrm{red}$ is contained in finitely many irreducible components of $\mathcal{N}^\mathrm{red}$. Since the scheme $\mathcal{Z}(L)$ is a closed formal subscheme of $\mathcal{N}$ and each irreducible component of $\mathcal{N}^\mathrm{red}$ is proper over $\Spec\bar\kappa$, it follows that the scheme $\mathcal{Z}(L)$ is proper over $\Spf \OFb$.
\end{proof}

\subsection{A cancellation law for $\Int(L)$}\label{sec:cancelation-law-intl}

Let $ M\subset \mathbb{V}=\mathbb{V}_m^\varepsilon$ be a self-dual lattice of arbitrary rank $r$. Then we have an orthogonal decomposition $$\BV=\BV_m^\varepsilon=M_F\obot\BV_{m-r}^{\varepsilon'}$$ for a unique $\varepsilon'\in\{\pm1\}$. By the same proof of \cite[Lemma 3.2.2]{Li2018}, we have a natural embedding 
 \begin{align}\label{eq:inc M} 
 \delta_M\colon 
 \xymatrix{\CN_{n-r}^{\varepsilon'}\ar[r]& \CN_n^\varepsilon},
 \end{align}
 which gives an identification
 \begin{align}
   \label{eq:ZMisomo}
   \CN_{n-r}^{\varepsilon'}\xrightarrow{\delta_M\atop \sim}\CZ(M)\subseteq \mathcal{N}_n^\varepsilon.
 \end{align}
   For $u\in\BV_m^\varepsilon$, denote by $u^\flat$ the projection to $\BV_{m-r}^{\varepsilon'}$.  If $u^\flat\neq 0$, then the special divisor $\CZ(u)\subseteq \mathcal{N}_n^\varepsilon$ intersects transversely with $\CN_{n-r}^{\varepsilon'}$ and its pull-back to $\CN_{n-r}^{\varepsilon'}$ is the special divisor $\CZ(u^\flat)\subseteq \CN_{n-r}^{\varepsilon'}$. For later reference, we write this fact as follows:
 \begin{align}\label{eq:ind Z(u)}
 \CN_{n-r}^{\varepsilon'}\jiao \CZ(u)=\CZ(u^\flat).
\end{align}
By \eqref{eq:ind Z(u)} and Definition \ref{def:Int(L)}, we have the following cancellation law for $\Int$ (analogous to (\ref{eq:cancel den}) for $\pDen$),
\begin{equation}
  \label{eq:cancelInt}
\Int^{\varepsilon}(L^\flat\obot M)=\Int^{\varepsilon'}(L^\flat).
\end{equation}
 
\subsection{Local arithmetic Siegel--Weil formula}\label{sec:local-kudla-rapoport}

Now we can state the main theorem of this article, which proves a local arithmetic Siegel--Weil formula on the identity between arithmetic intersection numbers of special cycles and central derivatives of local densities.

\begin{theorem}[local arithmetic Siegel--Weil formula]\label{thm: main}
Let $L\subseteq \mathbb{V}$ be an $O_F$-lattice of rank $n$. Then $$\Int(L)=\pDen(L).$$
\end{theorem}
This will be proved in \S\ref{ss:proof}.

\begin{remark}\label{rem:lowrank}
When $m=n+1=4$, Theorem \ref{thm: main} was previously known by explicitly computing both sides, see \cite[\S2.16]{Wedhorn2007} (based on Gross--Keating \cite{Gross1993}) for $\varepsilon=+1$ and Terstiege \cite[Theorem 6.1]{Terstiege2011} for $\varepsilon=-1$.
\end{remark}

\begin{example}
  Consider $m=n+1=4$ and $L$ has fundamental invariants $(1,1,1)$. Then $t(L)=3$. There is exactly one integral lattice $L'\supseteq L$ with $t(L')=3$, i.e., $L'=L$, and $\sgn_4(L')=1$ ($L'$ has rank 3 and type 3). The number of  integral lattices $L'\supseteq L$ with $t(L')=1$ is equal to the number of isotropic lines in the 3-dimensional quadratic space $L^\vee/L$ over $\kappa$, which is $q+1$. Moreover we have $\sgn_4(L')=1$ (since $L'$ has rank 3 and type 1, the non-degenerate part of $L'_\kappa$ is a 2-dimensional quadratic space containing an isotropic line, hence must be split). It follows from Theorem \ref{thm: Den(X)} that
  \begin{align*}
    \Dene(X,L)&=(1+\varepsilon q X)(1-X^2)+(q+1)X^2(1+\varepsilon X) \\
    &=1+\varepsilon qX+qX^2+\varepsilon X^3.
  \end{align*}
  This agrees with \cite[p. 211]{Terstiege2008}: in the notation there we have $$\Dene(X, L)=\tilde F_q(T; \varepsilon X).$$ In fact, since $\beta_1=\beta_2=\beta_3=1$, $\tilde \xi=\pm1$, $\sigma=2$, $\eta=+1$ ($L$ is isotropic since $\beta_i$ are all odd), we compute that $$\tilde F_q(T; X)=1+\eta X^3+\tilde\xi^2q(X+X^2)=1+qX+qX^2+X^3.$$

Now consider $\varepsilon=-1$. Then $\Den^-(X,L)$ has sign of functional equation $w^-(L)=-1$. In this case $$ \pDen^-(L)=-(-q+2q-3)=3-q.$$ It agrees with Corollary \ref{cor: pDen}: the terms with $t(L')=3$ has weight factor $2(1-q)$ and the terms with $t(L')=1$ has weight factor $1$, so in total we obtain $$\pDen^-(L)=2(1-q)\cdot 1+(q+1)\cdot 1=3-q.$$ By Example \ref{exa:smallm}, the space $\mathcal{N}=\mathcal{N}_3^-$ is isomorphic to $\mathcal{M}^\mathrm{HB}$ defined in \cite[\S2]{Terstiege2011}, and $\pDen^-(L)$ also agrees with $\Int^-(L)=3-q$ computed in \cite[Proposition 5.5 (i), (ii)]{Terstiege2011}.
\end{example}

\section{Horizontal parts of special cycles}

We continue with the setup in \S\ref{sec:kudla-rapop-cycl}. Let $L^\flat\subseteq\mathbb{V}$ be an $O_F$-lattice of rank $n-1$. Our next goal is to determine the horizontal part $\mathcal{Z}(L^\flat)_\sH$ of the special cycle $\mathcal{Z}(L^\flat)$ (\S\ref{sec:kudla-rapop-cycl-1}) in terms of primitive horizontal cycles $\mathcal{Z}(M^\flat)^\circ$ parametrized by $M^\flat\in \Hor(L^\flat)$ (Definition \ref{def:horLflat}).

\subsection{Quasi-canonical lifting cycles}\label{sec:quasi-canon-lift}

In this subsection we consider $m=n+1=3$. Then $\mathcal{N}_2\simeq\Spf \OFb[[t]]$ is isomorphic the Lubin--Tate deformation space of the formal group $\mathbb{E}$ of dimension 1 and height 2 over $\kb$ (Example \ref{exa:smallm}). As defined in  \cite{Gross1986a},  for $s\ge0$ and a quadratic extension $K/F$, a \emph{quasi-canonical lifting $E_{K,s}$} (and \emph{canonical lifting} when $s=0$) is a lift of $\mathbb{E}$ whose endomorphism ring is $$O_{K,s}:=O_F+\varpi^s \cdot O_K.$$ Let $\Kb$ be the completion of the maximal unramified extension of $K$. Let $\Kb_s$ be ring class field of $\Kb$ corresponding to $O_{K,s}^\times$ under local class field theory, with ring of integers $O_{\Kb,s}$. By \cite[Proposition 5.3]{Gross1986a},  the ring of definition of $E_{K,s}$ is $O_{\Kb,s}$, and the universal quasi-canonical lifting defines a horizontal divisor $$\mathcal{Z}_{K,s}\simeq\Spf O_{\Kb,s}\subseteq \mathcal{N}_2$$ with
\begin{equation}
  \label{eq:quasicanonical}
\deg_{\OFb}\mathcal{Z}_{K,s}=[O_{\Kb,s}: O_\Fb]=
\begin{cases}
  1, & s=0,\ K/F \text{ is unramified},\\
  q^{s}(1+q^{-1}), & s\ge1,\ K/F \text{ is unramified},\\
  2, & s=0,\ K/F \text{ is ramified},\\
  2q^s, & s\ge1,\ K/F \text{ is ramified}.
\end{cases}
\end{equation}

Let $M^\flat\subseteq \mathbb{V}_3^\varepsilon$ be an $O_F$-lattice of rank 1. Let $K(M^\flat)=F\left(\sqrt{\disc(M^\flat)}\right)$. Since $\mathbb{V}_3^\varepsilon$ is anisotropic, we know that $\chi(M^\flat)\ne+1$, and so $K(M^\flat)/F$ is a quadratic extension, which is unramified (resp. ramified) if $\chi(M^\flat)=-1$ (resp. $\chi(M^\flat)=0$). By \cite[(5.10)]{Gross1993} (see also \cite[\S3, p.147]{Rapoport2007}), we have a decomposition as (Cartier) divisors on $\mathcal{N}_2$, $$\mathcal{Z}(M^\flat)=\sum_{s=0}^{\lfloor \frac{\val(M^\flat)}{2}\rfloor}\mathcal{Z}_{K(M^\flat), s}.$$  We define the \emph{primitive part}  $\mathcal{Z}(M^\flat)^\circ$ of $\mathcal{Z}(M^\flat)$ to be the closed formal subscheme given by the unique irreducible component of $\mathcal{Z}(M^\flat)$ such that  $\mathcal{Z}(M^\flat)^\circ\not\subseteq\mathcal{Z}(M^{\flat'})$, for any $O_F$-lattice $M^{\flat'}\subseteq M^\flat_F$ such that $M^{\flat'}\supsetneq M^\flat$. Equivalently, $\mathcal{Z}(M^\flat)^\circ=\mathcal{Z}_{K(M^\flat),s}\subseteq \mathcal{Z}(M^\flat)$ for the maximal $s=\lfloor \frac{\val(M^\flat)}{2}\rfloor$.

\subsection{Gross--Keating cycles}\label{sec:gross-keating-cycles}

In this subsection we consider $m=n+1=4$ and $\varepsilon=+1$. In this case $\mathcal{N}_3^+\simeq \Spf \OFb[[t_1,t_2]]$ is isomorphic to the product of two copies of Lubin--Tate deformation spaces $\mathcal{N}_2$ over $\Spf \OFb$ (Example \ref{exa:smallm}). The intersection problem of special divisors on $\mathcal{N}_3^+$ are studied in detail by Gross--Keating \cite{Gross1993}. Let $M^\flat\subseteq \mathbb{V}_4^+$ be an $O_F$-lattice of rank 2. By \cite[p.239]{Gross1993} (see also \cite[\S3, p.147]{Rapoport2007}), we know that $\mathcal{Z}(M^\flat)$ is a horizontal 1-dimensional affine formal subscheme of $\mathcal{N}_3^+$, and each irreducible component of $\mathcal{Z}(M^\flat)$ is isomorphic to a quasi-canonical lifting cycle $\Spf O_{\Kb,s}$ for some $s$. We define the \emph{primitive part}  $\mathcal{Z}(M^\flat)^\circ$ of $\mathcal{Z}(M^\flat)$ to be the closed formal subscheme given by the union of irreducible components of $\mathcal{Z}(M^\flat)$ such that  $\mathcal{Z}(M^\flat)^\circ\not\subseteq\mathcal{Z}(M^{\flat'})$, for any $O_F$-lattice $M^{\flat'}\subseteq M^\flat_F$ such that $M^{\flat'}\supsetneq M^\flat$.

\subsection{Horizontal cycles} \label{sec:horizontal-cycles} Now consider general $m=n+1\ge3$ and $\varepsilon\in\{\pm1\}$. Let $M^\flat\subseteq \mathbb{V}$ be an $O_F$-lattice of rank $n-1$. Assume that $M^\flat$ is horizontal (Definition \ref{def:horizontallattice}). We have two cases.

If $t(M^\flat)\le 1$, then there exists a self-dual $O_F$-lattice $M_{n-2}$ of rank $n-2$ and an $O_F$-lattice $M_1$ of rank 1 such that $M^\flat=M_{n-2}\obot M_1$. By \eqref{eq:ZMisomo}, we have an isomorphism $$\mathcal{Z}(M_{n-2})\simeq \mathcal{N}_2.$$ Under this isomorphism, we can identify the cycle $\mathcal{Z}(M^\flat)\subseteq \mathcal{Z}(M_{n-2})$ with the cycle $\mathcal{Z}(M_1)\subseteq \mathcal{N}_2$, which is a union of quasi-canonical lifting cycles as in \S\ref{sec:quasi-canon-lift}. We define the \emph{primitive part} $\mathcal{Z}(M^\flat)^\circ$ of $\mathcal{Z}(M^\flat)$ to be the closed formal subscheme given by the primitive part $\mathcal{Z}(M_1)^\circ\subseteq \mathcal{Z}(M_1)$ under this identification. Notice that $\mathcal{Z}(M^\flat)^\circ$ is independent of the choice of the self-dual lattice $M_{n-2}$.

If $t(M^\flat)=2$, then there exists a self-dual $O_F$-lattice $M_{n-3}$ of rank $n-3$ and an $O_F$-lattice $M_2$ of rank 2 such that $M^\flat=M_{n-3}\obot M_2$. Since $M^\flat$ is horizontal, by Lemma \ref{lem:directsum} (\ref{item:l2}) we know that $$\chi(M_{n-3,F}^\perp)=\chi(\mathbb{V})\chi(M_{n-3,F})=\varepsilon \sgn_{n-1}(M^\flat)=+1,$$ and thus $M_{n-3,F}^\perp\simeq \mathbb{V}_4^+$. By \eqref{eq:ZMisomo}, we have an isomorphism $$\mathcal{Z}(M_{n-3})\simeq\mathcal{N}_3^+.$$ Under this isomorphism, we can identify the cycle $\mathcal{Z}(M^\flat)\subseteq \mathcal{Z}(M_{n-2})$ with the cycle $\mathcal{Z}(M_2)\subseteq \mathcal{N}_3^+$, which is a Gross--Keating cycle, also a union of quasi-canonical lifting cycles as in \S\ref{sec:gross-keating-cycles}. Similarly, we define the \emph{primitive part} $\mathcal{Z}(M^\flat)^\circ$ of $\mathcal{Z}(M^\flat)$ to be the closed formal subscheme given by the primitive part $\mathcal{Z}(M_2)^\circ\subseteq \mathcal{Z}(M_2)$ under this identification. Again $\mathcal{Z}(M^\flat)^\circ$ is independent of the choice of the self-dual lattice $M_{n-3}$.

Notice that the above two cases be combined: using \eqref{eq:ZMisomo} again we may identify  $\mathcal{N}_2$ as a special divisor on $\mathcal{N}_3^+$ associated to a self-dual lattice of rank 1. So when $M^\flat$ is horizontal, we may always identify $\mathcal{Z}(M^\flat)$ as a Gross--Keating cycle on $\mathcal{N}_3^+$.

\begin{theorem}\label{thm:horizontal}
Let $L^\flat\subseteq\mathbb{V}$ be an $O_F$-lattice of rank $n-1$. Then 
\begin{equation}
  \label{eq:horizontal}
  \mathcal{Z}(L^\flat)_\sH=\bigcup_{M^\flat \in \Hor(L^\flat)}\mathcal{Z}(M^\flat)^\circ.
\end{equation}
Moreover, the identity
\begin{equation}
  \label{eq:horizontalK}
\mathcal{O}_{\mathcal{Z}(L^\flat)_\sH}=\sum_{M^\flat \in \Hor(L^\flat)}\mathcal{O}_{\mathcal{Z}(M^\flat)^\circ}
\end{equation}
 holds in $\Gr^{n-1}K_0^{\mathcal{Z}(L^\flat)_\sH}(\mathcal{N})$.  
\end{theorem}

\begin{lemma}\label{lem:alldistinct}
The primitive cycles $\mathcal{Z}(M^\flat)^\circ$ on the right-hand-side of (\ref{eq:horizontal}) do not share any common irreducible component. 
\end{lemma}

\begin{proof}
Suppose $\mathcal{Z}(M_1^\flat)^\circ$ and $\mathcal{Z}(M_2^\flat)^\circ$ share a common irreducible component for $M_1^\flat, M_2^\flat\in \Hor(L^\flat)$. Let $M^\flat=M_1^\flat+M_2^\flat$. Then $M^\flat\in \Hor(L^\flat)$. By definition $\mathcal{Z}(M_1^\flat)\cap \mathcal{Z}(M_2^\flat)=\mathcal{Z}(M^\flat)$. Hence $\mathcal{Z}(M_1^\flat)^\circ$ and $\mathcal{Z}(M^\flat)$ share a common irreducible component. But $\mathcal{Z}(M_1^\flat)^\circ$ is primitive and $M_1^\flat\subseteq M^\flat$, by definition we have $M_1^\flat=M^\flat$. Similarly, we know that $M_2^\flat=M^\flat$. Hence $M_1^\flat=M_2^\flat$ as desired.
\end{proof}

Theorem \ref{thm:horizontal} will be proved in \S\ref{sec:proof-theor-refthm:h}. By Lemma \ref{lem:alldistinct}, we know that (\ref{eq:horizontal}) implies (\ref{eq:horizontalK}). It is clear from construction that in (\ref{eq:horizontal}) the right-hand-side is contained in the left-hand-side. To show the reverse inclusion, we use the theory of Tate modules, to be explained in next two sections.

\subsection{Tate modules and the projector $\bpi_\et$}\label{sec:tate-modules}

Let $Z$ be a formal scheme that is formally smooth and locally formally of finite type over $\Spf\OFb$. Let $X$ be a $p$-divisible group over $Z$. We denote by $T(X)$ the lisse $O_F$-sheaf of Tate modules over the rigid generic fiber $Z^\mathrm{rig}$, given by the projective system of \etale sheaves $\{X^\mathrm{rig}[\varpi^n]\}$. By definition, an \emph{integral \etale Tate tensor} on $X^\mathrm{rig}$ is a morphism $t: \mathbf{1}:=O_F\rightarrow T(X)^\otimes$ of lisse $O_F$-sheaves. By \cite[Theorem 7.1.6]{Kim2018}, for any crystalline Tate tensor $t_\crys: \mathbf{1}\rightarrow \mathbb{D}(X)^\otimes$ on $X$, there exists a unique integral \etale Tate tensor $t_{\et}:\mathbf{1}\rightarrow T(X)^\otimes$ on $X^\mathrm{rig}$ such that for any classical point $z\in Z^\mathrm{rig}$,   $t_{\et, \bar z}\in T(X_{\bar z})^\otimes$ matches with $t_{\crys,z}: \mathbf{1}\rightarrow D(X_z)^\otimes$ under the classical crystalline comparison isomorphism
\begin{equation}
  \label{eq:cryscomparison}
  B_\crys(\OFb) \otimes_\OFb \mathbb{D}(X_{z,\kb})(\OFb)\xrightarrow{\sim} B_\crys(\OFb) \otimes_F T(X)^*_{\bar z}[1/p]. 
\end{equation}
Here $\bar z$ is any geometric point supported at $z$, and $B_\crys$ is Fontaine's crystalline period ring. 

Applying this construction to $Z=\mathcal{N}$, $X=X^\mathrm{univ}$ and $t_\crys=\bpi_\crys$, we obtain an integral \etale Tate tensor $\bpi_\et$ on $X^\mathrm{univ,rig}$, which is a projector $$\bpi_\et: \End(T(X^\mathrm{univ}))\rightarrow\End(T(X^\mathrm{univ}))$$  whose image $\mathbf{V}_\et:=\im(\bpi_\et)$ is a lisse $O_F$-sheaf of rank $m$.

Now let $K/\Fb$ be a finite extension. Let $z\in \mathcal{N}(O_K)$. Let $X/O_K$ be the $p$-divisible group corresponding to $z$. Let $T_pX$ be the integral $p$-adic Tate module of $X$, a free $O_F$-module of rank $2^m$ with an $O_F$-linear action of $\Gamma_K:=\Gal(\overline{K}/K)$. The projector $\bpi_{\et}$ induces a projector of $O_F[\Gamma_K]$-modules $$\bpi_{\et, \bar z}: \End(T_pX)\rightarrow \End(T_pX).$$ Its image $\mathbf{V}_{\et, \bar z}$ is a free $O_F$-module of rank $m$ with an action of $\Gamma_K$. The endomorphism ring $\End(T_pX)$ has a natural quadratic module structure over $O_F$ given by $(f_1, f_2)=2^{-m}\tr(f_1\circ f_2)$, which induces a quadratic module structure on $\mathbf{V}_{\et,\bar z}$ satisfying $f\circ f=(f,f)\cdot\id_{T_pX}$ for $f\in \mathbf{V}_{\et,\bar z}$. This makes $\mathbf{V}_{\et,\bar z}$ a \emph{self-dual} $O_F$-lattice of rank $m$ isomorphic to $V$. In this way we view $V\simeq\mathbf{V}_{\et,\bar z}\subseteq\End(T_pX)$.

\subsection{Special endomorphisms} Let $X/O_K$ be the $p$-divisible group corresponding to $z\in \mathcal{N}(O_K)$ as in \S\ref{sec:tate-modules}.  We have a natural injection of $O_F$-modules $$i_K: \End(X)\hookrightarrow \End(T_pX).$$  On the other hand, the reduction map induces injection of $O_F$-modules $$i_\kb: \End(X)\hookrightarrow \End(X_\kb).$$ Tensoring with $F$ we obtain two injections (still denoted by the same notation) of vector spaces over $F$, $$\xymatrix{&\Endo(X) \ar[ld]_{i_K} \ar[rd]^{i_\kb} &  \\ \Endo(T_pX) & & \Endo(X_\kb).}$$

Recall that we view $\mathbb{V}\subseteq \Endo(X_\kb)$ via the projector $\bpi_\crys$ (\S\ref{sec:proj-bpi-bpi_crys}) and $V\subseteq \End(T_pX)$ via the projector $\bpi_\et$ (\S\ref{sec:tate-modules}). By the compatibility of $\bpi_\crys$ and $\bpi_\et$ under the crystalline comparison isomorphism (\ref{eq:cryscomparison}), we know that for $f\in \Endo(X)$, $\bpi_\crys(f)=f$ if and only if $\bpi_\et(f)=f$. Hence $i_K(f)\in V_F$ if and only if $i_\kb(f)\in \mathbb{V}$.

\begin{definition}
Define the space of \emph{special quasi-endomorphisms} of $X$ to be subspace $$\SEndo(X):=\{f\in \Endo(X): i_K(f)\in V_F\}=\{f\in \Endo(X): i_\kb(f)\in \mathbb{V}\}\subseteq \Endo(X),$$ and define  the space of \emph{special endomorphisms} of $X$ to be $\SEnd(X):=\SEndo(X)\cap \End(X)$.  
\end{definition}
 We have a natural quadratic $O_F$-module structure on $\SEnd(X)$ satisfying $f\circ f=(f,f)\cdot\id_X$ for $f\in\SEnd(X)$, which is compatible with the quadratic form on $V$ via $i_K$ (resp. on $\mathbb{V}$ via $i_\kb$).  In this way we obtain two injections of quadratic spaces over $F$, $$\xymatrix{&\SEndo(X) \ar[ld]_{i_K} \ar[rd]^{i_\kb} &  \\ V_F & & \mathbb{V}.}$$ 

\begin{lemma} \label{lem:tate}
The following identity holds: $$\SEnd(X)=i_K^{-1}(V).$$
\end{lemma}

\begin{proof}
By definition $V=V_F\cap \End(T_pX)$, so we have $i_K(\SEnd(X))\subseteq V_F\cap \End(T_pX)=V$. Conversely, suppose $f\in \SEndo(X)$ such that $i_K(f)\in V$. To show that $f\in\SEnd(X)$ it remains to show that $f\in \End(X)$. By \cite[Theorem 4, Corollary 1]{Tate1967}, the map $i_K$ induces an isomorphism $$i_K:\End(X)\cong \End_{O_F[\Gamma_K]}(T_pX),$$ where $\Gamma_K=\Gal(\ov K/K)$, and so an isomorphism $$i_K:\Endo(X)\cong \Endo_{O_F[\Gamma_K]}(T_pX).$$ Hence $$i_K(f)\in \Endo_{O_F[\Gamma_K]}(T_pX)\cap V\subseteq \Endo_{O_F[\Gamma_K]}(T_pX)\cap \End(T_pX)=\End_{O_F[\Gamma_K]}(T_pX).$$ It follows that $f\in \End(X)$ as desired. 
\end{proof}

\begin{corollary}\label{cor:ZMhoritonal}
  Let $M\subseteq \mathbb{V}$ be an $O_F$-lattice (of arbitrary rank). Then $z\in\mathcal{Z}(M)(O_K)$ if and only if
  \begin{align}\label{eq:SW}
    M \subseteq i_\kb( i_K^{-1}(V)).
  \end{align} In particular, when $z\in \mathcal{Z}(M)(O_K)$, there exists an embedding of quadratic $O_F$-modules $i_K\circ i_{\kb}^{-1}: M\hookrightarrow V$.
\end{corollary}

\begin{proof}
By definition we have $z\in \mathcal{Z}(M)(O_K)$ if and only if $M \subseteq i_\kb(\SEnd(X))$. The result then follows from Lemma \ref{lem:tate}.  
\end{proof}

\subsection{Proof of Theorem \ref{thm:horizontal}}\label{sec:proof-theor-refthm:h}

Let $z\in \mathcal{Z}(L^\flat)(O_K)$. By Corollary \ref{cor:ZMhoritonal}, we know that $$L^\flat\subseteq i_{\kb}( i_K^{-1} (V)).$$ Define $M^\flat\coloneqq L^\flat_F \cap i_{\kb}( i_K^{-1} (V))$. By Lemma \ref{lem:saturatedhorizontal}, we know that $M^\flat\in \Hor(L^\flat)$. By construction $M^\flat$ is the largest lattice in $L_F^\flat$ contained in $i_{\kb}( i_K^{-1}(V))$, thus we obtain that $z\in \mathcal{Z}(M^\flat)^\circ(O_K)$ by Corollary \ref{cor:ZMhoritonal} again. Therefore the $O_K$-points of both sides of \eqref{eq:horizontal} are equal.

To finish the proof of Theorem \ref{thm:horizontal}, by the flatness of both sides of \eqref{eq:horizontal} it remains to check that the $O_K[t]$-points of both sides are equal (where $t^2=0$). Namely, we would like to show that for each $z\in \mathcal{Z}(L^\flat)(O_K)$, there is a unique lift of $z$ in $\mathcal{Z}(L^\flat)(O_K[t])$. Let $\mathbf{V}_{\crys,z}$ be the crystal of rank $m$ associated to $z$ (\S\ref{sec:proj-bpi-bpi_crys}). Since the kernel of $O_K[t]\rightarrow O_K$ admits trivial nilpotent divided powers, by Lemma \ref{lem:GM}, a lift $\tilde z\in \mathcal{Z}(L^\flat)(O_K[t])$ of $z$ corresponds to an isotropic $O_K[t]$-line $\Fil^1\mathbf{V}_{\crys}(O_K[t])$ in $\mathbf{V}_{\crys}(O_K[t])$  lifting $\Fil^1\mathbf{V}_{\crys}(O_K)$  and orthogonal to the $O_K[t]$-submodule $L^\flat_{\crys,z}(O_K[t])\subseteq \mathbf{V}_{\crys,z}(O_K[t])$. By Breuil's theorem (as recalled in \cite[\S4.3]{LZ2019}), the $S$-submodule (where $S$ is Breuil's ring) $L^\flat_{\crys,z}(S)\subseteq \mathbf{V}_\crys(S)$ has rank $n-1$, and base changing from $S$ to $O_K$ we know that the $O_K$-module $L^\flat_{\crys,z}(O_K)\subseteq \mathbf{V}_\crys(O_K)$ also has rank $n-1$. Hence we know that there is a unique choice of such isotropic line $\Fil^1\mathbf{V}_{\crys}(O_K[t])$ orthogonal to the $O_K[t]$-submodule $L^\flat_{\crys,z}(O_K[t])\subseteq \mathbf{V}_{\crys,z}(O_K[t])$ of rank $n-1$. Hence the lift $\tilde z$ is unique as desired.

\subsection{Degree of primitive cycles}

\begin{lemma}\label{lem:degreeprimitive}
Let $M^\flat\subseteq \mathbb{V}$ be an $O_F$-lattice of rank $n-1$. Assume that $M^\flat$ is horizontal. If $\chi(M^\flat)\ne0$, then  $$\deg_{\OFb} \mathcal{Z}(M^\flat)^\circ=q^{\lfloor \frac{\val(M^\flat)}{2}\rfloor}\cdot 
  \begin{cases}
    1, & t(M^\flat)=0,\\
    (1+q^{-1}), & t(M^\flat)=1,\\
    2(1+q^{-1}), & t(M^\flat)=2.
  \end{cases}
  $$
If $\chi(M^\flat)=0$, then  $$\deg_{\OFb}\mathcal{Z}(M^\flat)^\circ=2q^{\lfloor \frac{\val(M^\flat)}{2}\rfloor}\cdot 
  \begin{cases}
    1, & t(M^\flat)=1,\\
    2, & t(M^\flat)=2.
  \end{cases}
$$
\end{lemma}

\begin{proof}
Since $M^\flat$ is horizontal,  as in \S\ref{sec:horizontal-cycles} we are reduced to the Gross--Keating case $m=4$, $\varepsilon=+1$. In this case, we prove the degree formula by induction on $\val(M^\flat)$. By induction hypothesis, the degree formula is true for all horizontal lattices $N^\flat\subseteq \mathbb{V}_4^+$ with $\val(N^\flat)<\val(M^\flat)$. 

  By Theorem \ref{thm:horizontal}, we know that $$\deg_{\OFb}\mathcal{Z}(M^\flat)^\circ=\deg_{\OFb}\mathcal{Z}(M^\flat) -\sum_{M^\flat \subsetneq N^\flat\in\Hor(M^\flat)}\deg_{\OFb}\mathcal{Z}(N^\flat)^\circ.$$ Choose $x\in (M^\flat_F)^\perp$ such that $M=M^\flat\obot \langle x\rangle$ and $\wit M=M^\flat\obot \langle\varpi^{-1}x\rangle$ are integral. By the proof of \cite[Theorem 6.8]{CY}, we know that $$\deg_{\OFb}\mathcal{Z}(M^\flat)=\pDen^+(M)-\pDen^+(\wit M).$$ Hence $$\deg_{\OFb}\mathcal{Z}(M^\flat)^\circ= (\pDen^+(M)-\pDen^+(\wit M)) -\sum_{M^\flat \subsetneq N^\flat\in\Hor(M^\flat)}\deg_{\OFb}\mathcal{Z}(N^\flat)^\circ.$$ The desired degree formula for $\mathcal{Z}(M^\flat)^\circ$ then follows from Corollary \ref{cor:pDendiff} and the induction hypothesis.
\end{proof}

\subsection{Relation with local densities}

Notice that $\deg_\OFb(\mathcal{Z}(L^\flat)_\sH)$ is equal to the degree of the 0-cycle $\mathcal{Z}(L^\flat)_{\Fb}$ in the generic fiber $\mathcal{N}_{\Fb}$ of the Rapoport--Zink space, which may be interpreted as a \emph{geometric} intersection number on the generic fiber. We have the following identity between this geometric intersection number and a local density.

\begin{corollary}\label{cor:genericfiber} Let $L^\flat\subseteq\mathbb{V}$ be an $O_F$-lattice of rank $n-1$.
 Then $$\deg_{\OFb}(\mathcal{Z}(L^\flat)_\sH)=
  \begin{cases}
    \Denf(1, L^\flat), & \chi(L^\flat)\ne0\\
    2\Denf(1, L^\flat), & \chi(L^\flat)=0.
  \end{cases}
$$
\end{corollary}

\begin{proof}
  It follows immediately from Theorem \ref{thm:horizontal}, Lemma \ref{lem:degreeprimitive} and Corollary \ref{cor:pDendiff}.
\end{proof}

\begin{remark}
  Using the $p$-adic uniformization theorem (\S\ref{sec:p-adic-unif}) and the flatness of the horizontal part of the global special cycles, one may deduce from Corollary \ref{cor:genericfiber} an identity between the geometric intersection number (i.e. the degree) of a special 0-cycle on a compact Shimura variety associated to $\GSpin(n-1,2)$ and the value of a Fourier coefficient of a \emph{coherent} Siegel Eisenstein series on $\Sp(n,n)$ at the near central point $s=1/2$. This should give a different proof of a theorem of Kudla \cite[Theorem 10.6]{Kudla1997}.
\end{remark}

\section{Vertical parts of special cycles}\label{sec:vert-comp-kudla}

We continue with the setup in \S\ref{sec:kudla-rapop-cycl}. Let $L^\flat\subseteq \mathbb{V}$ be an $O_F$-lattice of rank $n-1$.

\subsection{The support of the vertical part $\mathcal{Z}(L^\flat)_{\sV}$} \label{sec:supp-vert-part} 

Recall that $\mathcal{Z}(L^\flat)_\sV$ is the vertical part of the special cycle $\mathcal{Z}(L^\flat)\subseteq \mathcal{N}$ (\S\ref{sec:kudla-rapop-cycl-1}).

\begin{proposition}\label{prop:supp Z(L) V}
$\mathcal{Z}(L^\flat)_\sV$ is supported on $\mathcal{N}^\mathrm{red}$, i.e., $\mathcal{O}_{\mathcal{Z}(L^\flat)_\sV}$ is annihilated by a power of the ideal sheaf of $\mathcal{N}^\mathrm{red}\subseteq \mathcal{N}$.
\end{proposition}

\begin{proof}
  If not, we may find a affine formal curve $\mathcal{C}=\Spf R\subseteq \mathcal{Z}(L^\flat)_\sV$ such that $\mathcal{C}$ has a unique closed point $z\in \mathcal{N}(\kb)$. The universal $p$-divisible group $X^\mathrm{univ}$ over $\mathcal{N}$ pulls back to a $p$-divisible group $\mathcal{X}$ over $\Spec R$. Let $\eta$ be a geometric generic point of $\Spec R$, with algebraically closed residue field $\kb_\eta$. Let $\Fb_\eta$ be the fraction field of the Witt ring of $\kb_\eta$.  Let $\mathbf{V}_{\crys, \eta}$ (resp. $\mathbf{V}_{\crys, z}$) be the crystal associated to the point $\eta$ (resp. $z$) as in \S\ref{sec:proj-bpi-bpi_crys}. Denote the corresponding isocrystal by $(V_{\Fb_\eta}, \Phi_\eta=b_\eta\circ\sigma)$, where $b_\eta\in H(\Fb_\eta)$ (resp. $(V_\Fb, \Phi=\bar b\circ\sigma)$, where $\bar b\in H(\Fb)$ as in \S\ref{sec:space-special-quasi}), and $H=\SO(V_F)$. Recall that these isocrystals with $H$-structure are classified by the class of $b_\eta$ and $\bar b$ in the Kottwitz set $B(H)$ of $\sigma$-conjugacy classes of $H(\Fb)$, and $B(H)$ is independent of the algebraically closed residue fields $\kb_\eta$ and $\kb$ (\cite[Lemma 1.3]{Rapoport1996}).  By the specialization theorem of Rapoport--Richartz \cite[Theorem 3.6]{Rapoport1996}, we have $\bar b\prec b_\eta$, where $\prec$ is the partial ordering on $B(H)$ (\cite[\S2.3]{Rapoport1996}).

  To obtain a contradiction to that $z$ is the unique closed point of $\mathcal{C}$, it suffices to show that $b_\eta$ is a basic element. Since $\mathcal{C}\subseteq \mathcal{Z}(L^\flat)_\sV$, the lattice $L^\flat$ acts on $\mathcal{X}$ via special endomorphisms and hence $\Phi$ (resp. $\Phi_\eta$) stabilizes $L^\flat_\Fb$ (resp. $L^\flat_{\Fb_\eta}$).  Let $\mathbb{W}$ be the 2-dimensional orthogonal complement of $L^\flat_F$ in $\mathbb{V}$ and $T=\SO(\mathbb{W})$. Since $\Phi$ stabilizes $L^\flat_\Fb$, we may write $\bar b=(b^\flat, b^\perp)$ for basic elements $b^\flat\in B(\SO(L^\flat_F))$ and $b^\perp\in B(T)$. Similarly, we may write $b_\eta=(b^\flat, b_\eta^\perp)$ for some (possibly non-basic) element $b_\eta^\perp\in B(T)$. Now $b^\perp$ is a specialization of $b^\perp_\eta$ implies that $b^\perp\prec b^\perp_\eta$ for the partial ordering on $B(T)$. But $T$ is a torus, we know that $b^\perp\prec b^\perp_\eta$ implies that $b^\perp=b^\perp_\eta$ in $B(T)$ by \cite[\S1.7 and Theorem 3.6 (i)]{Rapoport1996} (namely, in the notation of \cite{Rapoport1996}, the map $\delta: \mathcal{N}(T)\rightarrow X_*(T)^\Gamma \otimes \mathbb{Q}$ is an isomorphism for any torus $T$, so $\delta(\nu(b))=\delta(\nu(b_\eta))$ implies $\nu(b)=\nu(b_\eta)$). Hence $b_\eta=\bar b\in B(H)$ is also basic, a contradiction as desired.
\end{proof}

\begin{remark}\label{sec:supp-vert-part-1}
  The key observation in the proof is that the quadratic space $(L^\flat_F)^\perp$ has dimension 2 and hence its isometry group is a torus. In the unitary case, the similar observation also holds (the hermitian space $(L^\flat_F)^\perp$ has dimension 1 and its isometry group is a torus) and provides an alternative (more group-theoretic) proof of \cite[Lemma 5.1.1]{LZ2019}.
\end{remark}

\subsection{Horizontal and vertical parts of $^\BL\CZ(L^\flat)$}\label{sec:horiz-vert-parts-1}

Since $\CZ(L^\flat)_\sH$ is either empty or 1-dimensional (Theorem \ref{thm:horizontal}), the intersection $\CZ(L^\flat)_\sH\cap \CZ(L^\flat)_\sV$ must be either empty or 0-dimensional. It follows that there is a decomposition of the $(n-1)$-th graded piece 
\begin{align}\label{eq:gradedK0}
 \Gr^{n-1}K_0^{\CZ(L^\flat)}(\mathcal{N})= \Gr^{n-1}K_0^{\CZ(L^\flat)_\sH}(\mathcal{N})\oplus \Gr^{n-1}K_0^{\CZ(L^\flat)_\sV}(\mathcal{N}).
\end{align}

\begin{definition}
The decomposition (\ref{eq:gradedK0})  induces a decomposition of the derived special cycle into \emph{horizontal} and \emph{vertical} parts
\begin{align*}
^\BL\CZ(L^\flat)=\,^\BL\CZ(L^\flat)_\sH+ ~ ^\BL\CZ(L^\flat)_\sV\in \Gr^{n-1}K_0^{\CZ(L^\flat)}(\mathcal{N}).
\end{align*}
From this decomposition, we see that even though the vertical part $\mathcal{Z}(L^\flat)_\sV$ may depend on the choice of an integer $N\gg0$ (\S\ref{sec:horiz-vert-parts}), the element $^\BL\CZ(L^\flat)_\sV\in \Gr^{n-1}K_0^{\CZ(L^\flat)}(\mathcal{N})$ is canonical and independent of the choice of $N$.

Since $\CZ(L^\flat)_\sH$ is either empty or has the expected dimension one, the first summand $^\BL\CZ(L^\flat)_\sH$ is represented by the structure sheaf of $\CZ(L^\flat)_\sH$ by \cite[Lemma B.2 (ii)]{Zhang2019}. Abusing notation we shall write the sum as
\begin{align}\label{eq:H+V}
^\BL\CZ(L^\flat)=\CZ(L^\flat)_\sH+ ~ ^\BL\CZ(L^\flat)_\sV.
\end{align}
\end{definition}

\subsection{The Tate conjecture for certain Deligne--Lusztig varieties} Consider the generalized Deligne--Lusztig variety $Y_{d}\coloneqq Y_W$ and the classical Deligne--Lusztig $Y_d^\circ\coloneqq Y_W^\circ$ as defined in \S \ref{sec:gener-deligne-luszt}, where $W$ is the unique non-split quadratic space over $\kappa$ of dimension $2(d+1)$. Recall that we have a stratification $$Y_d=\bigsqcup_{i=0}^d X_{P_i}(w_i).$$ Let $$X_i^\circ\coloneqq X_{P_i}(w_i), \quad X_i\coloneqq \overline{X_i^\circ}=\bigsqcup_{k=0}^i X_k^\circ.$$ Then $X_i^\circ$ is a disjoint union of isomorphic copies of the classical Deligne--Lusztig variety $Y_i^\circ$, and each irreducible component of $X_i$ is isomorphic to $Y_i$.

For any $\mathbb{F}_{q^2}$-variety $S$, we write  $H^{j}(S)(i)\coloneqq H^{j}(S_{\kb}, \overline{\mathbb{Q}_\ell}(i))$ ($\ell\ne p$ is a prime). Let $\Fr=\Fr_{q^2}$ be the $q^2$-Frobenius acting on $H^{j}(S)(i)$.

\begin{lemma} \label{lem:lusztig}For any $d, i\ge0$ and $s\ge1$, the action of $\Fr^s$ on the following cohomology groups are semisimple, and the space of $\Fr^s$-invariants is zero when $j\ge1$.
  \begin{altenumerate}
  \item\label{item:1} $H^{2j}(Y_d^\circ)(j)$.
  \item\label{item:2} $H^{2j}(X_i^\circ)(j)$.
  \item\label{item:3} $H^{2j}(Y_d-X_i)(j)$.
  \end{altenumerate}
\end{lemma}

\begin{proof}
  \begin{altenumerate}
  \item The assertion is clear when $d=\dim Y_d^\circ\le1$.  When $d\ge 2$, by \cite[7.3~Case~$^2D_{n}~(n\ge3)$]{Lusztig1976/77} (notice the adjoint group assumption is harmless due to \cite[1.18]{Lusztig1976/77}), we know that there are exactly $d+1$ eigenvalues $\{1, q^2,\ldots, q^{2d}\}$ of $\Fr$ acting on $H^{*}_c(Y_d^\circ)$, with $q^{2j}$ exactly appearing in degree $d+j$. By the Poincare duality, we have a perfect pairing  $$H^{2d-j}_c(Y_d^\circ) \times H^{j}(Y_d^\circ)(d)\rightarrow H^{2d}_c(Y_d^\circ)(d)\simeq\overline{\mathbb{Q}_\ell}.$$ Thus the eigenvalue of $\Fr$ on $H^{2j}(Y_d^\circ)(j)$ are given by $q^{2(d-j)}$ times the inverse of the eigenvalue in $H_c^{2(d-j)}(Y_d^\circ)$, which is equal to $q^{2j}$. Hence the eigenvalue of $\Fr^s$ is never equal to 1 when $j\ge1$.  The semisimplicity of the action of $\Fr^s$ follows from \cite[6.1]{Lusztig1976/77}.
  \item It follows from (\ref{item:1}) since $X_i^\circ$ is a disjoint union of $Y_i^\circ$.
  \item It follows from (\ref{item:2}) since $Y_d-X_i=\bigsqcup_{k=i+1}^d X_k^\circ$.\qedhere
  \end{altenumerate}
\end{proof}

\begin{theorem}\label{thm:tate}
  For any $0\le i \le d$ and any $s\ge1$, we have
  \begin{altenumerate}
  \item   The space of Tate classes $H^{2i}(Y_d)(i)^{\Fr^s=1}$ is spanned by the cycle classes of the irreducible components of $X_{d-i}$. In particular, the Tate conjecture (\cite[Conjecture 1]{Tate1965}, or \cite[Conjecture $T^i$]{Tate1994}) holds for $Y_d$.
  \item\label{item:eigen1}   Let $H^{2i}(Y_d)(i)_1\subseteq H^{2i}(Y_d)(i)$ be the the generalized eigenspace of $\Fr^s$ for the eigenvalue 1. Then $H^{2i}(Y_d)(i)_1=H^{2i}(Y_d)(i)^{\Fr^s=1}$.
  \end{altenumerate}
\end{theorem}

\begin{proof}
  The same proof of \cite[Theorem 5.2.2]{LZ2019} works verbatim using Lemma \ref{lem:lusztig} in place of \cite[Lemma 5.2.1]{LZ2019}.
\end{proof}

\section{Fourier transform: the geometric side}\label{sec:four-transf-geom}

We continue with the setup in \S\ref{sec:kudla-rapop-cycl}. Let $L^\flat\subseteq \mathbb{V}$ be an $O_F$-lattice of rank $n-1$.

\subsection{Decomposition of $\Int$ and $\pDen$}\label{sec:decomp-int-pden}
Recall from \eqref{eq:H+V} that there is a decomposition of the derived special cycle $^\BL\CZ(L^\flat)$ into a sum of vertical and horizontal parts
$$
^\BL\CZ(L^\flat)=\CZ(L^\flat)_\sH+ ~ ^\BL\CZ(L^\flat)_\sV.
$$

\begin{definition}
  Let $L^\flat\subseteq \mathbb{V}$ be an $O_F$-lattice of rank $n-1$. Denote by $\mathbb{W}:=(L^\flat_F)^\perp\subseteq \BV$, a 2-dimensional non-degenerate quadratic space over $F$. Define $$\Omega(L^\flat):=L^\flat_F\times \mathbb{W}^\mathrm{an},$$ an open dense subset of $\mathbb{V}$. Then we have $$\Omega(L^\flat)=\{x\in \mathbb{V}: L^\flat+\langle x\rangle\text{ is a non-degenerate } O_F\text{-lattice of rank } n\}.$$ 
\end{definition}

\begin{definition}
For $x\in \Omega(L^\flat)$, define the \emph{arithmetic intersection number} $$\Int_{L^\flat}(x)\coloneqq\Int(L^\flat +\langle x\rangle)= \chi(\mathcal{N}, \mathcal{Z}(x)\jiao \mathcal{Z}(L^\flat)).$$ Define its \emph{horizontal part} to be
  \begin{align}
    \Int_{L^\flat,\sH}(x)\coloneqq\chi(\CN,\CZ(x)\jiao
    \CZ(L^\flat)_\sH),\quad 
  \end{align} and its  \emph{vertical part} to be
\begin{align}\label{def: Int V}
\Int_{L^\flat,\sV}(x)\coloneqq\chi(\CN,\CZ(x)\jiao {}^\BL\CZ(L^\flat)_\sV).
\end{align}
\end{definition}
Then there is a decomposition
\begin{align}\label{eq:Intdecomp}
\Int_{L^\flat}(x)=\Int_{L^\flat,\sH}(x)+\Int_{L^\flat,\sV}(x).
\end{align}

\begin{definition}\label{def:horpDen}
  Analogously, for $x\in \Omega(L^\flat)$, define the \emph{derived local density} $$\pDen_{L^\flat}(x)\coloneqq \pDen(L^\flat+\langle x\rangle).$$ By Corollary \ref{cor: pDen}, we have
  \begin{align}
    \label{eq:pDenx}
    \pDen_{L^\flat}(x)=\sum_{ L^\flat\subset L'\subset
      L'^\vee}\fkm(t(L'),\sgn_{n+1}(L')){\bf 1}_{L'}(x).
  \end{align}
 Here $L'\subseteq \mathbb{V}$ are $O_F$-lattices of rank $n$. Define its \emph{horizontal part} to be
  \begin{align}\label{pDen H}
    \pDen_{L^\flat,\sH}(x)\coloneqq\sum_{ L^\flat\subset L'\subset
      L'^\vee \atop L'^\flat\in \Hor(L^\flat) }\fkm(t(L'),\sgn_{n+1}(L')){\bf 1}_{L'}(x), 
  \end{align}
  and its \emph{vertical part} to be \begin{align}\label{pDen V}
\pDen_{L^\flat,\sV}(x)\coloneqq \sum_{ L^\flat\subset L'\subset
      L'^\vee \atop L'^\flat\not\in \Hor(L^\flat) }\fkm(t(L'),\sgn_{n+1}(L')){\bf 1}_{L'}(x).
\end{align}
  Here we denote
  \begin{align}\label{eq:L' flat}
    L'^\flat\coloneqq L'\cap L_F^\flat\subset L_F^\flat.
  \end{align}
\end{definition}
Then there is a decomposition 
$$
\pDen_{L^\flat}(x)=\pDen_{L^\flat,\sH}(x)+\pDen_{L^\flat,\sV}(x).
$$

\subsection{The horizontal identity}

\begin{lemma}\label{lem:primitiveIntH}
Let $L^\flat\subseteq \mathbb{V}$ be an $O_F$-lattice of rank $n-1$. Assume that $L^\flat$ is horizontal. Then for $x\in \Omega(L^\flat)$, $$\chi(\mathcal{N}, \mathcal{Z}(x)\jiao \mathcal{Z}(L^\flat)^\circ)=\sum_{L^\flat\subset L'\subset L'^\vee \atop L'^\flat=L^\flat}\fkm(t(L'),\sgn_{n+1}(L')){\bf 1}_{L'}(x).$$
\end{lemma}

\begin{proof}
By Theorem \ref{thm:horizontal}, we have
\begin{align}\label{eq:IntH}
  \Int_{L^\flat,\sH}(x)=\chi(\mathcal{N}, \mathcal{Z}(x)\jiao \mathcal{Z}(L^\flat)^\circ)+ \sum_{L'^\flat \in\Hor(L^\flat)\atop L'^\flat\ne L^\flat} \chi(\mathcal{N}, \mathcal{Z}(x)\jiao \mathcal{Z}(L'^\flat)^\circ).
\end{align}
 Since $L^\flat$ is horizontal,  as in \S\ref{sec:horizontal-cycles} we are reduced to the Gross--Keating case $m=4$, $\varepsilon=+1$. In this case, we prove the desired formula by induction on $\val(L^\flat)$.  By induction hypothesis, the desired formula for each $\chi(\mathcal{N}, \mathcal{Z}(x)\jiao \mathcal{Z}(L'^\flat)^\circ)$ in the second summand is valid. Hence $$\Int_{L^\flat,\sH}(x)=\chi(\mathcal{N}, \mathcal{Z}(x)\jiao \mathcal{Z}(L^\flat)^\circ)+\sum_{L^\flat\subset L'\subset L'^\vee \atop  L'^\flat\ne L^\flat}\fkm(t(L'),\sgn_{n+1}(L')){\bf 1}_{L'}(x),$$ where we notice that $L'^\flat\in \Hor(L^\flat)$ always as $L^\flat\subseteq \mathbb{V}=\mathbb{V}_4^+$ has rank $n-1=2$. By \S\ref{sec:gross-keating-cycles} we have $\mathcal{Z}(L^\flat)=\mathcal{Z}(L^\flat)_\sH$ and hence $$\Int_{L^\flat,\sH}(x)=\Int_{L^\flat}(x).$$ By Remark~\ref{rem:lowrank} we have $$\Int_{L^\flat}(x)=\pDen_{L^\flat}(x).$$ It follows that $$\chi(\mathcal{N}, \mathcal{Z}(x)\jiao \mathcal{Z}(L^\flat)^\circ)=\pDen_{L^\flat}(x)-\sum_{L^\flat\subset L'\subset L'^\vee \atop  L'^\flat\ne L^\flat}\fkm(t(L'),\sgn_{n+1}(L')){\bf 1}_{L'}(x).$$ The desired formula then follows from (\ref{eq:pDenx}).
\end{proof}

\begin{theorem}\label{thm:Int H}Let $L^\flat\subseteq \mathbb{V}$ be an $O_F$-lattice of rank $n-1$. Then as functions on  $\Omega(L^\flat)$, 
$$\Int_{L^\flat,\sH}=\pDen_{L^\flat,\sH}.$$
\end{theorem}
\begin{proof}
This follows immediately from (\ref{eq:IntH}), Lemma \ref{lem:primitiveIntH} and (\ref{pDen H}).
\end{proof}

\subsection{Decomposition of $\Int_{L^\flat,\sV}$}

By Proposition \ref{prop:supp Z(L) V}, we have a change-of-support homomorphism 
$$\xymatrix{\Gr^{n-1}K_0^{\CZ(L^\flat)_\sV}(\mathcal{N})\ar[r]& \Gr^{n-1}K_0^{\mathcal{N}^\mathrm{red}}(\mathcal{N}).}$$ Abusing notation we will also denote the image of $\LZ(L^\flat)_\sV$ in the target by the same symbol.

\begin{lemma}\label{cor:curves}
There exist countably many (and finitely many, if  $L^\flat$ is co-anisotropic) curves $C_i\subseteq \mathcal{N}^\mathrm{red}$ and $\mult_{C_i}\in \mathbb{Q}$ such that each irreducible component of $\mathcal{N}^\mathrm{red}$ contains only finitely many $C_i$'s and $$\LZ(L^\flat)_\sV=\sum_i \mult_{C_i}[\mathcal{O}_{C_i}]\in \Gr^{n-1}K_0^{\mathcal{N}^\mathrm{red}}(\mathcal{N}).$$ 
\end{lemma}

\begin{proof}
  It follows immediately from Proposition \ref{prop:supp Z(L) V}, where the finiteness of such curves $C_i$ in the co-anisotropic case is due to the noetherianess of $\mathcal{Z}(L^\flat)$ by Lemma \ref{lem:ZLnoetherian} (\ref{item:noetherian}). 
\end{proof}

\begin{corollary} \label{cor:ZV}
  There exist (countably many) Deligne--Lusztig curves $C_i\subseteq \mathcal{N}_n^\mathrm{red}$ (i.e., $C_i=\mathcal{V}(\Lambda)\cong \mathbb{P}^1_\kb$ for a vertex lattice $\Lambda\in \Ver^4(\mathbb{V})$, see \S\ref{sec:bruh-tits-strat}) and $\mult_{C_i}\in \mathbb{Q}$ such that each irreducible component of $\mathcal{N}^\mathrm{red}$ contains only finitely many $C_i$'s and $$ \Int_{L^\flat,\sV}(x)=\sum_i \mult_{C_i}\cdot  \chi(\CN_n,\,C_i\jiao \mathcal{Z}(x))
  $$ as functions on $\Omega(L^\flat)$, where the sum on the right-hand-side is locally finite (and finite if $L^\flat$ is co-anisotropic).
\end{corollary}

\begin{proof}
  The same proof of \cite[Corollary 5.3.3]{LZ2019} works using Theorem  \ref{thm:tate} in place of \cite[Theorem 5.3.2]{LZ2019} and Lemma \ref{cor:curves} in place of \cite[Corollary 5.2.2]{LZ2019}, where the local finiteness of the sum is due to the noetherianess $\mathcal{Z}(L^\flat+\langle x\rangle)$ by Lemma \ref{lem:ZLnoetherian} (\ref{item:noetherian}). 
\end{proof}

\subsection{Computation of $\Int_{\mathcal{V}(\Lambda)}$} Let $\Lambda\subseteq \mathbb{V}$ be a vertex lattice. Let $\CV(\Lambda)$ be the Deligne--Lusztig variety in the Bruhat--Tits stratification of $\mathcal{N}^\mathrm{red}$ (\S\ref{sec:bruh-tits-strat}). Define 
\begin{equation}\label{eq:Int CV}\Int_{\CV(\Lambda)}(x)\coloneqq \chi\bigl(\CN,\mathcal{V}(\Lambda) \jiao \mathcal{Z}(x)\bigr),\quad x\in \mathbb{V}\setminus \{0\}.
\end{equation}

Next we explicitly compute $\Int_{\CV(\Lambda)}$  for $\Lambda\in \Ver^4(\mathbb{V})$, i.e., for $\mathcal{V}(\Lambda)\cong \mathbb{P}^1_\kb$ a Deligne--Lusztig curve.

\begin{theorem}\label{thm:int Lam}
Let $\Lambda\in \Ver^4(\mathbb{V})$.
Then
$$
\Int_{\CV(\Lambda)}(x)=\begin{cases}(1-q) ,& x\in\Lambda ,\\
1,& x\in \Lambda^\vee\setminus \Lambda, \text{and }\val(x)\geq 0,\\
0,&\text{otherwise}. 
\end{cases}
$$ In particular, the function $\Int_{\CV(\Lambda)}$ on $\mathbb{V}\setminus\{0\}$ extends to a (necessarily unique) function in $\Ss(\mathbb{V})$, which we still denote by the same symbol.
\end{theorem}

\begin{proof}
  Since $\Lambda\in\Ver^4(\mathbb{V})$, it admits an orthogonal decomposition $\Lambda=\Lambda^\flat\obot M$, where $\Lambda^\flat$ is a rank $4$ vertex lattice of type $4$, and $M$ is a self-dual lattice of rank $m-4$. Moreover, by \eqref{eq:tmax} we know that $\chi(\Lambda^\flat_F)=-1$. Thus by the same proof of \cite[Lemma 6.2.1]{LZ2019} using the cancellation law \eqref{eq:ZMisomo}, we are reduced to the case $m=4$ and $\varepsilon=-1$. Namely, we may assume that $\mathbb{V}=\mathbb{V}_4^-$ and $\mathcal{N}=\mathcal{N}_3^-$. Let $x\in \mathbb{V}\setminus \{0\}$. Since $\mathcal{V}(\Lambda)=\mathcal{Z}(\Lambda)$ (\S\ref{sec:minusc-kudla-rapop}), we know that $\mathcal{V}(\Lambda)\cap \mathcal{Z}(x)=\mathcal{Z}(\Lambda+\langle x\rangle)$. If $x\not\in \Lambda^\vee$ or $\val(x)<0$, then $\Lambda+\langle x\rangle$ is not integral. Hence $\mathcal{V}(\Lambda)\cap \mathcal{Z}(x)=\varnothing$ and so $\Int_{\mathcal{V}(\Lambda)}=0$. If $x\in \Lambda^\vee\setminus \Lambda$, then $\Lambda+\langle x\rangle\in\Ver^2(\mathbb{V})$ and hence $\mathcal{V}(\Lambda)\cap \mathcal{Z}(x)=\mathcal{V}(\Lambda+\langle x\rangle)$ consists of a single $\kb$-point, and so $\Int_{\mathcal{V}(\Lambda)}(x)=1$.

It remains to show that $\Int_{\mathcal{V}(\Lambda)}(x)=1-q$ for nonzero $x\in \Lambda$. Since $\mathcal{Z}(\Lambda)\jiao \mathcal{Z}(x)$ is invariant when replacing $x$ by any nonzero element in $\Lambda$ by the proof of \cite[Proposition 4.2]{Terstiege2011}, we know that $\Int_{\mathcal{V}(\Lambda)}(x)$ is a constant for all nonzero $x\in \Lambda$. Hence without loss of generality, we may assume $\val(x)=1$.  As $\mathcal{V}(\Lambda) \jiao \mathcal{Z}(x)$ is supported on the special fiber $\mathcal{N}_\kb$ and $\mathcal{Z}(x)$ is flat over $\Spf\OFb$ (Proposition \ref{prop:relativecartier}) we know that $$\Int_{\mathcal{V}(\Lambda)}(x)=\chi(\mathcal{N}, \mathcal{V}(\Lambda)\jiao\mathcal{Z}(x))=\chi(\mathcal{N}_\kb, \mathcal{V}(\Lambda)\jiao_{\mathcal{N}_\kb} \mathcal{Z}(x)_\kb).$$ When $\val(x)=1$, via moduli interpretation (\cite[p. 216]{Terstiege2008}) we know that $\mathcal{Z}(x)\simeq \mathcal{M}$ is isomorphic to Drinfeld's moduli space $\mathcal{M}$ of special formal $O_B$-modules (where $B$ is the quaternion algebra over $F$). The special fiber $\mathcal{M}_\kb$ is a union of $\mathbb{P}^1_\kb$'s whose dual graph is the Bruhat--Tits tree of $\PGL_2(F)$. It follows that $\mathcal{V}(\Lambda)$ intersects with $\mathcal{Z}(x)_\kb$ exactly (apart from the self intersection) with $q+1$ of adjacent $\mathbb{P}^1_\kb$'s, and the intersection number is equal to 1 for each such $\mathbb{P}^1_\kb$. It follows that that $$\chi(\mathcal{N}_\kb, \mathcal{V}(\Lambda)\jiao_{\mathcal{N}_\kb} \mathcal{Z}(x)_\kb)=(q+1) +\chi(\mathcal{N}_\kb, \mathcal{V}(\Lambda)\jiao_{\mathcal{N}_\kb}\mathcal{V}(\Lambda)).$$ This is equal to $(q+1)+(-2q)=1-q$ by Lemma \ref{lem:excess} below, which completes the proof.
\end{proof}

\begin{lemma}\label{lem:excess}
  Assume that $\mathbb{V}=\mathbb{V}_4^-$ and $\Lambda\in \Ver^4(\mathbb{V})$. Let $\mathcal{I}$ be the ideal sheaf of the closed subscheme $\mathcal{V}(\Lambda)\subseteq \mathcal{N}_{\kb}$. Then
  \begin{altenumerate}
  \item there is an isomorphism of coherent sheaves on $\mathcal{V}(\Lambda)$, $$\mathcal{I}/\mathcal{I}^2\isoarrow \mathcal{O}(2q),$$ where $\mathcal{O}(k)$ is the line bundle of degree $k$ on $\mathcal{V}(\Lambda)\simeq \mathbb{P}^1_{\kb}$.
    \item the self intersection number of $\mathcal{V}(\Lambda)\subseteq \mathcal{N}_\kb$ equals  $$\chi(\mathcal{N}_\kb, \mathcal{V}(\Lambda)\jiao_{\mathcal{N}_\kb} \mathcal{V}(\Lambda))=-2q.$$
  \end{altenumerate}
\end{lemma}

\begin{proof}
  Since $\mathcal{V}(\Lambda)$ is an irreducible curve on the 2-dimensional regular formal scheme $\mathcal{N}_\kb$, it is a Cartier divisor on $\mathcal{N}_\kb$ and the closed immersion $\mathcal{V}(\Lambda)\rightarrow \mathcal{N}_\kb$ is regular. Hence there is a conormal exact sequence of locally free sheaves
  \begin{equation}
    \label{eq:conormal}
    0\rightarrow \mathcal{I}/\mathcal{I}^2\rightarrow \Omega_{\mathcal{N}_\kb}|_{\mathcal{V}(\Lambda)}\rightarrow \Omega_{\mathcal{V}(\Lambda)}\rightarrow 0.
  \end{equation}
 Under the identification $\mathcal{V}(\Lambda)\simeq \mathbb{P}^1_\kb$, we have the cotangent sheaf $\Omega_{\mathcal{V}(\Lambda)}\simeq \mathcal{O}(-2)$.

By Example \ref{exa:smallm}, the space  $\mathcal{N}$ is isomorphic to the formal moduli space  $\mathcal{M}^\mathrm{HB}$ defined in \cite[\S2]{Terstiege2011}. Let $X$ be the universal $p$-divisible group over $\mathcal{M}^\mathrm{HB}_\kb$ and $\mathbb{D}(X)$ be its Dieudonn\'e crystal (taken to be covariant for the purpose of this proof). Then we have a Hodge filtration $$0\rightarrow \omega_{X^\vee}\rightarrow \mathbb{D}(X)\rightarrow \Lie_X\rightarrow0.$$ Here $\omega_{X^\vee},\Lie_X$ are both locally free of rank 2. By Grothendieck--Messing theory we know that the tangent sheaf is given by $$T_{\mathcal{M}^\mathrm{HB}_\kb}\simeq \HOM_{O_E}(\omega_{X^\vee}, \Lie_X).$$ The special $O_E$-action on $X$ induces $\mathbb{Z}/2 \mathbb{Z}$-grading $$\mathbb{D}(X)=\mathbb{D}(X)_0 \oplus \mathbb{D}(X)_1,\quad \omega_{X^\vee}=\omega_{X^\vee,0} \oplus \omega_{X^\vee,1}, \quad \Lie_X=\Lie_{X,0}\oplus \Lie_{X,1}$$ compatible with the Hodge filtration. Each of $\omega_{X^\vee,}$  and $\Lie_{X,i}$ ($i\in \mathbb{Z}/2 \mathbb{Z}$) is locally free of rank 1. Hence there is an isomorphism
  \begin{equation}
    \label{eq:tangentM}
    T_{\mathcal{M}^\mathrm{HB}_\kb}\simeq\HOM(\omega_{X^\vee, 0},\Lie_{X,0}) \oplus \HOM(\omega_{X^\vee,1}, \Lie_{X,1}).
  \end{equation}
 Now recall the identification $\mathcal{V}(\Lambda)\simeq \mathbb{P}^1_\kb$ in \cite[\S2]{Terstiege2011} and \cite[\S4]{Kudla1999a}. There exists a unique critical index $i\in \mathbb{Z}/ 2 \mathbb{Z}$ for $\mathcal{V}(\Lambda)$. For $i$ critical, there exists an $O_E$-lattice $\Lambda_i$ of rank 2 such that for any $z\in \mathcal{V}(\Lambda)(\kb)$, we have $\mathbb{D}(X_z)_i(\OFb)=\Lambda_{i,\OFb}$ and $\mathcal{V}(\Lambda)\simeq\mathbb{P}^1(\Lambda_{i,\kb})$. The line corresponding to $z\in \mathbb{P}^1_\kb=\mathbb{P}^1(\Lambda_{i,\kb})$ is then given by the Hodge filtration for the (critical) $i$-th grading $$\omega_{X_z^\vee,i}\subseteq \mathbb{D}(X_z)_i(\kb)=\Lambda_{i,\kb}.$$ It follows that if $i$ is critical for $\mathcal{V}(\Lambda)$, then $\omega_{X^\vee,i}|_{\mathcal{V}(\Lambda)}$ is the tautological line bundle $\mathcal{O}(-1)$ on $\mathbb{P}^1_\kb$, and thus $$\HOM(\omega_{X^\vee,i},\Lie_{X,i})|_{\mathcal{V}(\Lambda)}\simeq\HOM(\mathcal{O}(-1), \mathcal{O}(1))\simeq \mathcal{O}(2).$$ Moreover $\omega_{X^\vee,i+1}|_{\mathcal{V}(\Lambda)}$ is the Frobenius twist of $\Lie_{X,i}|_{\mathcal{V}(\Lambda)}$ by \cite[Lemma 4.3 (a)]{Kudla1999a}, and thus $$\HOM(\omega_{X^\vee,i+1},\Lie_{X,i+1})|_{\mathcal{V}(\Lambda)}\simeq\HOM(\mathcal{O}(q), \mathcal{O}(-q))\simeq \mathcal{O}(-2q).$$ It follows from Equation (\ref{eq:tangentM}) that the tangent sheaf $$T_{\mathcal{N}_{\kb}}|_{\mathcal{V}(\Lambda)}\simeq T_{\mathcal{M}^\mathrm{HB}_{\kb}}|_{\mathcal{V}(\Lambda)}\simeq \mathcal{O}(2) \oplus \mathcal{O}(-2q),$$ and so the cotangent sheaf $$\Omega_{\mathcal{N}_\kb}|_{\mathcal{V}(\Lambda)}\simeq \mathcal{O}(-2) \oplus \mathcal{O}(2q).$$ By Equation (\ref{eq:conormal}), we obtain $$\mathcal{I}/\mathcal{I}^2\simeq \mathcal{O}(2q).$$ This completes the proof of the first assertion.

To show the second assertion,  we use the short exact sequence $$0\rightarrow \mathcal{I}\rightarrow \mathcal{O}_{\mathcal{N}_\kb}\rightarrow \mathcal{O}_{\mathcal{V}(\Lambda)}\rightarrow 0.$$ Applying $\otimes_{\mathcal{O}_{\mathcal{N}_\kb}}\mathcal{O}_{\mathcal{V}(\Lambda)}$ gives an exact sequence $$0\rightarrow \Tor_1^{\mathcal{O}_{\mathcal{N}_\kb}}(\mathcal{O}_{\mathcal{V}(\Lambda)},\mathcal{O}_{\mathcal{V}(\Lambda)})\rightarrow \mathcal{I}/\mathcal{I}^2\rightarrow \mathcal{O}_{\mathcal{V}(\Lambda)}\rightarrow \mathcal{O}_{\mathcal{V}(\Lambda)}\rightarrow 0.$$ Hence we obtain an isomorphism $$\Tor_1^{\mathcal{O}_{\mathcal{N}_\kb}}(\mathcal{O}_{\mathcal{V}(\Lambda)},\mathcal{O}_{\mathcal{V}(\Lambda)})\simeq \mathcal{I}/\mathcal{I}^2.$$ Since $\mathcal{V}(\Lambda)$ is a Cartier divisor on $\mathcal{N}_\kb$, we know that $\Tor_i^{\mathcal{O}_{\mathcal{N}_\kb}}(\mathcal{O}_{\mathcal{V}(\Lambda)},\mathcal{O}_{\mathcal{V}(\Lambda)})=0$ for all $i\ge2$. Thus the derived intersection is represented by a complex $$\mathcal{V}(\Lambda)\jiao_{\mathcal{N}_\kb} \mathcal{V}(\Lambda)\simeq[\mathcal{O}_{\mathcal{V}(\Lambda)}\rightarrow \mathcal{I}/\mathcal{I}^2],$$ where $\mathcal{O}_{\mathcal{V}(\Lambda)}$ is in degree 0 and $\mathcal{I}/\mathcal{I}^2$ is in degree 1. It follows from definition that 
  $$\chi(\mathcal{N}_\kb, \mathcal{V}(\Lambda)\jiao_{\mathcal{N}_\kb} \mathcal{V}(\Lambda))=\chi(\mathcal{V}(\Lambda), \mathcal{O}_{\mathcal{V}(\Lambda)})-\chi(\mathcal{V}(\Lambda), \mathcal{I}/\mathcal{I}^2).$$ Under the identifications $\mathcal{V}(\Lambda)\simeq\mathbb{P}^1_\kb$ and $\mathcal{I}/\mathcal{I}^2\simeq\mathcal{O}(2q)$, this is equal to $$\chi(\mathbb{P}^1_\kb, \mathcal{O})-\chi(\mathbb{P}^1_\kb,\mathcal{O}(2q))=\deg\mathcal{O}-\deg \mathcal{O}(2q)=-2q$$ by the Riemann--Roch theorem.
\end{proof}

\begin{remark}\label{sec:comp-int_m}
  A proof similar to that of Theorem \ref{thm:int Lam} works for the unitary case and provides an alternative way (independent of the displays computations in \cite{Terstiege2013}) of computing $\Int_{\mathcal{V}(\Lambda)}(x)=1-q^2$ for nonzero $x\in \Lambda$ in  \cite[Lemma 6.2.1]{LZ2019}. In fact, in the unitary case the Deligne--Lusztig curve $\mathcal{V}(\Lambda)$ is isomorphic to the Fermat curve of degree $q+1$ in $\mathbb{P}^2_\kb$ and its conormal sheaf $\mathcal{I}/\mathcal{I}^2$ is isomorphic to the pullback to $\mathcal{V}(\Lambda)$ of the line bundle $\mathcal{O}(2q-1)$ on $\mathbb{P}^2_\kb$. Thus the self intersection number of $\mathcal{V}(\Lambda)$ is $(1-2q)(q+1)$. Moreover, when $\val(x)=1$ the special fiber $\mathcal{Z}(x)_\kb$ is a union of Fermat curves whose dual graph is the Bruhat--Tits tree of the quasi-split unitary group in 2 variables, and hence we obtain $\Int_{\mathcal{V}(\Lambda)}(x)=q(q+1) +(1-2q)(q+1)=1-q^2$.
\end{remark}

\begin{remark}
  The number $-2q$ also agrees with the global result on the self intersection number of irreducible components of the supersingular locus of Hilbert modular surfaces at a good inert prime (\cite{Tian2019}).
\end{remark}

\begin{corollary}\label{int Lam}
Let $\Lambda\in \Ver^4(\mathbb{V})$.
Then $$\Int_{\CV(\Lambda)}=-q(1+q) {\bf 1}_{\Lambda}+\sum_{\Lambda\subset\Lambda'\in \Ver^2(\mathbb{V})}{\bf 1}_{\Lambda'}.$$
\end{corollary}

\begin{proof}
  We compute the value of the right-hand-side at $x\in\BV$ according to three cases.
  \begin{altenumerate}
  \item If $x\in \Lambda$, then there are exactly $q^2+1$ type 2 lattices $\Lambda'$ containing $\Lambda$, and the value is $$-q(1+q)+(q^2+1)=1-q.$$
  \item If $x\in \Lambda^\vee\setminus \Lambda$ and $\val(x)\geq 0$, then $\Lambda + \langle x\rangle\in \Ver^2(\mathbb{V})$. Then there are exactly one lattice $\Lambda'=\Lambda+\langle x\rangle$  appearing in the sum, and the value is $0+1=1$.
  \item If $x\not\in \Lambda^\vee$ or $\val(x)\le0$, then $\Lambda+\langle x\rangle$ is not integral, and the value is clearly 0.
  \end{altenumerate}
  The result then follows from Theorem \ref{thm:int Lam}.
\end{proof}

\begin{corollary}\label{cor:LC int}
  The function $\Int_{L^\flat,\sV}$ is in $\Lloc(\Omega(L^\flat))$ (see Notations \S\ref{sec:notations-functions}) and extends (uniquely) to a distribution in $\Dd(\mathbb{V})$ (and a function in $\Ss(\mathbb{V})$ if $L^\flat$ is co-anisotropic), which we still denote by the same symbol. \end{corollary}

\begin{proof}
  It follows from Corollary \ref{cor:ZV} and  Corollary \ref{int Lam} that $\Int_{L^\flat,\sV}$ is a locally finite (and finite if $L^\flat$ is co-anisotropic) linear combination of functions in $\Ss(\mathbb{V})$ and hence locally integrable on $\Omega(L^\flat)$ (and in $\Ss(\mathbb{V})$ if $L^\flat$ is co-anisotropic). \end{proof}

\subsection{Fourier transform: the geometric side; ``Local modularity''}
We compute the Fourier transform of $\Int_{\CV(\Lambda)}$ as a function in $\Ss(\BV)$.

\begin{lemma}\label{lem: FT Int V}
Let $\Lambda\in \Ver^4(\mathbb{V})$. Then $\Int_{\CV(\Lambda)}\in\Ss(\BV)$ satisfies
$$
\wh{\Int_{\CV(\Lambda)}}=\gamma_\BV\Int_{\CV(\Lambda)}.
$$
Here $\gamma_\BV=-1$ is the Weil constant.
\end{lemma}
\begin{proof}By Corollary \ref{int Lam} and (\ref{eq:latticefourier}), we obtain
  \begin{align*}
    \wh{\Int_{\CV(\Lambda)}}&=-\vol(\Lambda)\cdot q(1+q)\cdot 1_{\Lambda^\vee}+\sum_{\Lambda\subset\Lambda'\in \Ver^2(\mathbb{V})}\vol(\Lambda')\cdot 1_{\Lambda'^\vee}\\
&=-(1+q^{-1})\cdot 1_{\Lambda^\vee}+\sum_{\Lambda\subset\Lambda'\in \Ver^2(\mathbb{V})}q^{-1}\cdot1_{\Lambda'^\vee}.
  \end{align*}

Now we compute its value at $u\in\BV$ according to four cases.
\begin{altenumerate}
\item If $u\in \Lambda$, there are exactly $q^2+1$ type 2 lattices $\Lambda'$ containing $\Lambda$, and the value is 
$$-(1+q^{-1})+q^{-1} (q^2+1)=q-1.$$
\item If $u\in \Lambda_2\setminus \Lambda$ for some $\Lambda_2\in \Ver^2(\mathbb{V})$, i.e., the image of $\bar u$ of $u$ in $\Lambda^\vee/\Lambda$ is an isotropic vector. Notice that $u\in \Lambda'^\vee $ if and only if $\ov u$ is orthogonal to the line given by the image of $(\Lambda')^\vee$ in $\Lambda^\vee/\Lambda$. So there is exactly one such $ \Lambda'\in \Ver^2(\mathbb{V})$, i.e., $\Lambda'=\Lambda_2$, and we obtain the value
 $$-(1+q^{-1})+q^{-1} =-1.$$
 \item If $u\in \Lambda^\vee\setminus\Lambda$ but $u\not\in \Lambda_2\setminus \Lambda$ for any $\Lambda_2\in\Ver^2(\mathbb{V})$. Then $\ov u$ is anisotropic in $\Lambda^\vee/\Lambda$. Notice that $\langle \ov u\rangle^\perp$ is a non-degenerate quadratic space of dimension 3, and $\Lambda'$ corresponds to an isotropic line in $\langle\ov u\rangle^\perp$. So there are exactly  $q+1$ of such $ \Lambda'\in \Ver^2(\mathbb{V})$, and we obtain the value
 $$ -(1+q^{-1})+q^{-1}(q+1)=0.$$
\item If $u\not\in \Lambda^\vee$, then the value at $u$ is clearly 0.
\end{altenumerate}
The result then follows from Theorem \ref{thm:int Lam}.
\end{proof}

\begin{remark}
It follows from Lemma \ref{lem: FT Int V} that $\Int_{\CV(\Lambda)}$ is  $\SL_2(O_{F})$-invariant under the Weil representation. This invariance may be viewed  as a  ``local modularity'', an analog of the global modularity of arithmetic generating series of special divisors (such as in \cite{HMP20}).
\end{remark}

\begin{corollary}\label{cor:FT int}
The following identity in $\Dd(\mathbb{V})$ (and in $\Ss(\mathbb{V})$ if $L^\flat$ is co-anisotropic) holds,
 $$\wh{\Int_{L^\flat,\sV}}=\gamma_\BV\Int_{L^\flat,\sV}.
 $$
\end{corollary}

\begin{proof}
This follows from Corollary \ref{cor:ZV} and Lemma \ref{lem: FT Int V}.
\end{proof}

\subsection{Fourier transform: the geometric side; ``Higher local modularity''} In this subsection we generalize Lemmas \ref{int Lam} and \ref{lem: FT Int V} on the function $\Int_{\CV(\Lambda)}$ for vertex lattices $\Lambda$ of type 4 to vertex lattices $\Lambda$ of arbitrary type $t(\Lambda)=2d+2\ge4$ (i.e., $d\ge1$).

Let $$\ch: K_0(\mathcal{V}(\Lambda))_\mathbb{Q}\rightarrow \bigoplus_{i=0}^d\Ch^i(\mathcal{V}(\Lambda))_\mathbb{Q}$$ be the Chern character from the Grothendieck ring to the Chow ring of $\mathcal{V}(\Lambda)$, which is an isomorphism of graded rings.  In particular, it induces an isomorphism $$\ch_i: \Gr^iK_0(\mathcal{V}(\Lambda))_\mathbb{Q}\isoarrow \Ch^i(\mathcal{V}(\Lambda))_\mathbb{Q},$$ for $0\le i\le d$. Let $$\cl_i: \Ch^i(\mathcal{V}(\Lambda))_\mathbb{Q}\rightarrow H^{2i}(\mathcal{V}(\Lambda), \mathbb{Q}_\ell)(i)$$ be the $\ell$-adic cycle class map and let $$\cl=\bigoplus_{i=0}^d\cl_i: \bigoplus_{i=0}^d\Ch^i(\mathcal{V}(\Lambda))_\mathbb{Q}\rightarrow \bigoplus_{i=0}^d H^{2i}(\mathcal{V}(\Lambda), \mathbb{Q}_\ell)(i).$$ Then $\cl$ intertwines the intersection product on the Chow ring and the cup product on the cohomology ring, namely the following diagram commutes,
\begin{equation}
  \label{eq:chcl}
\begin{gathered}
  \xymatrix@C=0.5em{\Gr^iK_0(\mathcal{V}(\Lambda))_\mathbb{Q} \ar[d]^{\ch_i}_{\rotatebox{90}{$\sim$}} &\times & \Gr^jK_0(\mathcal{V}(\Lambda))_\mathbb{Q} \ar[d]^{\ch_j}_{\rotatebox{90}{$\sim$}} \ar[rr]^-{\cdot} && \Gr^{i+j}K_0(\mathcal{V}(\Lambda))_\mathbb{Q} \ar[d]^{\ch_{i+j}}_{\rotatebox{90}{$\sim$}}\\\Ch^i(\mathcal{V}(\Lambda))_\mathbb{Q} \ar[d]^{\cl_i} &\times & \Ch^j(\mathcal{V}(\Lambda))_\mathbb{Q} \ar[d]^{\cl_j} \ar[rr]^-{\cdot} && \Ch^{i+j}(\mathcal{V}(\Lambda))_\mathbb{Q} \ar[d]^{\cl_{i+j}} \\ H^{2i}(\mathcal{V}(\Lambda)),\mathbb{Q}_\ell)(i) &\times & H^{2j}(\mathcal{V}(\Lambda), \mathbb{Q}_\ell)(j) \ar[rr]^-{\cup} &&  H^{2(i+j)}(\mathcal{V}(\Lambda), \mathbb{Q}_\ell)(i+j).}
\end{gathered}
\end{equation}

Denote by $\Tate^{2i}_\ell(\mathcal{V}(\Lambda))\subseteq H^{2i}(\mathcal{V}(\Lambda), \mathbb{Q}_\ell)(i)$ the subspace of Tate classes, i.e., the elements fixed by $\Fr^s$ for some power $s\ge1$. Then by Theorem \ref{thm:tate}, we have the identity $$\im(\cl_i)_{\mathbb{Q}_\ell}=\Tate^{2i}_\ell(\mathcal{V}(\Lambda)),$$ and moreover $\Tate^{2i}_\ell(\mathcal{V}(\Lambda))$ is spanned by the cycle classes of $\mathcal{V}(\Lambda')\subseteq \mathcal{V}(\Lambda)$, where $\Lambda'\supseteq \Lambda$ runs over vertex lattices of type $2(d-i)+2$. Denote by
\begin{equation}
  \label{eq:barK}
  \barK:=K_0(\mathcal{V}(\Lambda))_\mathbb{Q}/\ker({\cl\circ\ch}),\quad \barCh^i(\mathcal{V}(\Lambda)):=\Ch^i(\mathcal{V}(\Lambda))_\mathbb{Q}/\ker \cl_i.
\end{equation}
 Then $\ch$ and $\cl$ induce isomorphisms
 \begin{equation}
   \label{eq:Tateclasses}
   \ch: \barK\isoarrow \bigoplus_{i=0}^d \barCh^i(\mathcal{V}(\Lambda)), \quad \cl: \bigoplus_{i=0}^d \barCh^i(\mathcal{V}(\Lambda))_{\mathbb{Q}_\ell}\isoarrow \bigoplus_{i=0}^d \Tate^{2i}_\ell(\mathcal{V}(\Lambda)).
 \end{equation}
  By Theorem \ref{thm:tate}~(\ref{item:eigen1}) and that the cup product 
  is $\Fr$-equivariant, the Poincar\'e duality induces a perfect pairing
  \begin{equation}
    \label{eq:tatepairing}
    \cup: \Tate_\ell^{2i}(\mathcal{V}(\Lambda)) \times \Tate_\ell^{2d-2i}(\mathcal{V}(\Lambda))\rightarrow \mathbb{Q}_\ell.
  \end{equation}

  \begin{definition}\label{def:K CV}
  For $x\in \mathbb{V}\setminus\{0\}$, define $\IntK(x)\in\barK$ to be the image of $$\mathcal{V}(\Lambda)\jiao \mathcal{Z}(x)\in K^{\mathcal{V}(\Lambda)}_0(\mathcal{N})\isoarrow K_0(\mathcal{V}(\Lambda))$$ under (\ref{eq:barK}).  
\end{definition}
\begin{remark}
Our main result in this subsection (Theorem \ref{thm:highermodularity}) shows that the function $\IntK$ satisfies the local modularity analogous to Lemma \ref{lem: FT Int V}.
  \end{remark}
 Since $\mathcal{Z}(x)$ is a Cartier divisor on $\mathcal{N}$,  we know that $\mathcal{V}(\Lambda)\jiao \mathcal{Z}(x)$ is explicitly represented by the two-term complex of line bundles on $\mathcal{V}(\Lambda)$, $$[\mathcal{O}_{\mathcal{N}}(-\mathcal{Z}(x))|_{\mathcal{V}(\Lambda)}\rightarrow \mathcal{O}_{\mathcal{N}}|_{\mathcal{V}(\Lambda)}]\in \mathrm{F}^1K_0(\mathcal{V}(\Lambda)).$$ Thus we have the Chern character
 \begin{align*}
   \ch(\mathcal{V}(\Lambda)\jiao \mathcal{Z}(x))&=\ch(\mathcal{O}_{\mathcal{N}}(-\mathcal{Z}(x))|_{\mathcal{V}(\Lambda)})-\ch(\mathcal{O}_{\mathcal{V}(\Lambda)})\\
   &=\exp(c_1(\mathcal{O}_{\CN}(-\mathcal{Z}(x))|_{\mathcal{V}(\Lambda)}))-\exp(c_1(\mathcal{O}_{\mathcal{V}(\Lambda)}))\\
   &=\sum_{i=1}^d\frac{ c_1(\mathcal{O}_{\CN}(-\mathcal{Z}(x))|_{\mathcal{V}(\Lambda)})^i}{i!}.  
 \end{align*}

\begin{definition}
For $x\in \mathbb{V}\setminus\{0\}$,  define $$\Intch(x):=c_1(\mathcal{O}_{\CN}(-\mathcal{Z}(x))|_{\mathcal{V}(\Lambda)})\in \Ch^1(\mathcal{V}(\Lambda))_\mathbb{Q},$$ and $$\Intc(x):=c_1(\mathcal{O}_{\CN}(-\mathcal{Z}(x))|_{\mathcal{V}(\Lambda)})^d\in \Ch^d(\mathcal{V}(\Lambda))_\mathbb{Q}\xrightarrow{\stackrel{\deg}{\sim}} \mathbb{Q}.$$
\end{definition}

\begin{lemma}\label{lem:Lambdainvariant}
The function $\Intch$ (resp. $\Intc$) is $\Lambda$-invariant, under the translation by $\Lambda$. In particular, the function $\Intch$ (resp. $\Intc$) extends uniquely to an $\Lambda$-invariant function on $\mathbb{V}$, or equivalently, a function on $\mathbb{V}/\Lambda$ (still denoted by the same symbol).
\end{lemma}

\begin{proof}
  The same proof of \cite[Lemma 6.4.4]{LZ2019} works verbatim.
\end{proof}

\begin{lemma}\label{lem:firstchern}
  Let $\Lambda\in \Ver^{2d+2}(\mathbb{V})$. Then for any $x\in \Lambda$, we have $$\Intch(x)=-\frac{1}{1+q^{d+1}}\sum_{y\in W_\Lambda\setminus \{0\}\atop (y,y)=0}\Intch(y)\in \Ch^1(\mathcal{V}(\Lambda))_\mathbb{Q},$$
  where $W_\Lambda=\Lambda^\vee/\Lambda$ is a quadratic space over $\kappa$ of dimension $2d+2$ (see \S\ref{sec:minusc-kudla-rapop}).
\end{lemma}

\begin{proof}
  Since the cycle class map for divisors $\cl_1: \Ch^1(\mathcal{V}(\Lambda))_\mathbb{Q}\rightarrow H^2(\mathcal{V}(\Lambda), \mathbb{Q}_\ell)(1)$  is injective, we know from (\ref{eq:Tateclasses}) that $$\Ch^1(\mathcal{V}(\Lambda))_{\mathbb{Q}_\ell}\isoarrow \Tate^2_\ell(\mathcal{V}(\Lambda)).$$ It follows from the perfect pairing (\ref{eq:tatepairing}) that to show the desired identity it suffices to show that for any Deligne--Lusztig curve $\mathcal{V}(\Lambda')\subseteq \mathcal{V}(\Lambda)$ ($t(\Lambda')=4$), the following identity
  \begin{equation}
    \label{eq:cupwithtypethree}
    \Intch(x)\cdot [\mathcal{V}(\Lambda')]=-\frac{1}{1+q^{d+1}}\sum_{y\in W_\Lambda\setminus \{0\}\atop (y,y)=0}\Intch(y)\cdot [\mathcal{V}(\Lambda')]
  \end{equation}
  holds in $\Ch^d(\mathcal{V}(\Lambda))_\mathbb{Q}\isoarrow \mathbb{Q}$. By the projection formula, $$\Intch(x)\cdot [\mathcal{V}(\Lambda')]=c_{1, \mathcal{V}(\Lambda')}(x),\quad \Intch(y)\cdot [\mathcal{V}(\Lambda')]=c_{1, \mathcal{V}(\Lambda')}(y).$$ Since $t(\Lambda')=4$, we know from Theorem \ref{thm:int Lam} that $$c_{1, \mathcal{V}(\Lambda')}(x)=\Int_{\mathcal{V}(\Lambda')}(x)=(1-q),\quad c_{1, \mathcal{V}(\Lambda')}(y)=\Int_{\mathcal{V}(\Lambda')}(y)=
  \begin{cases}
    (1-q), & y\in \Lambda'/\Lambda,\\
    1, & y\in \Lambda'^\vee/\Lambda, y\not\in\Lambda'/\Lambda, \\
    0, & y\not\in \Lambda'^\vee/\Lambda.
  \end{cases}
$$ Since $\Lambda'/\Lambda\subseteq W_\Lambda$ is totally isotropic, the number of nonzero isotropic vectors $y\in \Lambda'/\Lambda$ equals $\#(\Lambda'/\Lambda)-1=q^{d-1}-1$. The number of isotropic vectors $y\in \Lambda'^\vee/\Lambda$, $y\not\in \Lambda'/\Lambda$ equals $\#(\Lambda'/\Lambda)$ times the number of nonzero isotropic vectors in $\Lambda'^\vee/\Lambda'$, which evaluates to $q^{d-1}\cdot (q^2+1)(q-1)$. It follows that $$\sum_{y\in W_\Lambda\setminus \{0\}\atop (y,y)=0}\Intch(y)\cdot [\mathcal{V}(\Lambda')]=(q^{d-1}-1)\cdot (1-q)+q^{d-1}(q^2+1)(q-1)=-(1-q)(1+q^{d+1}).$$ Hence $$-\frac{1}{1+q^{d+1}}\sum_{y\in W_\Lambda\setminus \{0\}\atop (y,y)=0}\Intch(y)\cdot [\mathcal{V}(\Lambda')]=(1-q)=\Intch(x)\cdot [\mathcal{V}(\Lambda')],$$ and the desired identity (\ref{eq:cupwithtypethree}) holds.
\end{proof}

\begin{lemma}\label{lem:Intcformula}
  Let $\Lambda\in \Ver^{2d+2}(\mathbb{V})$. Then $$\Intc(x)=c(d)\cdot
  \begin{cases}
    (1-q^{d}), & x\in \Lambda, \\
    1, & x\in \Lambda^\vee\setminus\Lambda, \val(x)\ge0,\\
    0, & \text{otherwise}.
  \end{cases}$$
  Here $c(1)=1$ and $c(d)=\prod_{i=1}^{d-1}(1-q^{i})$.
\end{lemma}

\begin{proof}
  We induct on $d$. The base case $d=1$ follows from Theorem \ref{thm:int Lam}. By the same proof as in Theorem \ref{thm:int Lam}, we know that $\Intc(x)=0$ unless $x\in \Lambda^\vee$ and $\val(x)\ge0$. By the $\Lambda$-invariance of $\Intc$ in Lemma \ref{lem:Lambdainvariant}, to show the result it remains to show that
  \begin{equation}
    \label{eq:Intc}
    \Intc(0)=c(d)(1-q^{d}),\quad \Intc(x)=c(d)
  \end{equation}
  for any $x\in W_\Lambda\setminus\{0\}$ with $(x,x)=0$.

  By Lemma \ref{lem:firstchern}, we have $$\Intc(0)=\Intch(0)^{d-1}\Intch(0)=\Intch(0)^{d-1}\left( -\frac{1}{1+q^{d+1}}\sum_{y\in W_\Lambda\setminus \{0\}\atop (y,y)=0}\Intch(y) \right).$$ By the projection formula, we have $$\Intch(0)^{d-1}\Intch(y)=c_{1, \mathcal{V}(\Lambda+\langle y\rangle)}(0)^{d-1},$$ which by induction equals $$c_{\mathcal{V}(\Lambda+\langle y\rangle)}(0)=c(d-1)(1-q^{d-1})$$ since $t(\Lambda+\langle y\rangle)=2d$. The number of nonzero isotropic vectors $y\in W_\Lambda$ equals $(q^{d+1}+1)(q^d-1)$. 
  Hence $$\Intc(0)=(q^{d+1}+1)(q^{d}-1)\cdot \left( -\frac{1}{1+q^{d+1}} c(d-1)(1-q^{d-1})\right)=(1-q^{d})(1-q^{d-1})\cdot c(d-1).$$   On the other hand, for any $x\in W_\Lambda\setminus\{0\}$ with $(x,x)=0$. by the projection formula, we have $$\Intc(x)=\Intch(x)^{d-1}\Intch(x)=c_{1,\mathcal{V}(\Lambda+\langle x\rangle)}(x)^{d-1},$$ which by induction equals $$c_{\mathcal{V}(\Lambda+\langle x\rangle)}(x)=c(d-1)(1-q^{d-1})$$ since $t(\Lambda+\langle x\rangle)=2d$. The desired identity (\ref{eq:Intc}) then follows as $c(d-1)(1-q^{d-1})=c(d)$.
\end{proof}

\begin{lemma}\label{lem:Intc}
  Let $\Lambda\in \Ver^{2d+2}(\mathbb{V})$. Then $$\Intc=\frac{c(d)}{c'(d)}\sum_{\Lambda' \in \Ver^4(\mathbb{V})\atop \Lambda'\supseteq \Lambda}\Int_{\mathcal{V}(\Lambda')}.$$ Here $c'(1)=1$ and $c'(d)=\prod_{i=2}^{d}(1+q^{i+1})$ when $d\geq 2$.
\end{lemma}

\begin{proof}
We distinguish three cases.
\begin{altenumerate}
\item For $x\in \Lambda$, we have $\Int_{\mathcal{V}(\Lambda')}(x)=1-q$ for any $\Lambda'$ in the sum by Theorem \ref{thm:int Lam}. The number of such $\Lambda'$ is the number of $(d-1)$-dimensional totally isotropic subspaces in $W_\Lambda$, which equals $S_{d-1}(W_\Lambda)$ (in the notation of Lemma \ref{lem:Smb}). Since $\dim_\kappa W_\Lambda=2d+2$ and $\chi(W_\Lambda)=-1$, the right-hand-side evaluates to $$\frac{c(d)}{c'(d)}S_{d-1}(W_\Lambda)(1-q)=\frac{c(d)}{\prod_{i=2}^d(1+q^{i+1})}\frac{(q^{d+1}+1)(q^2-1)\prod_{i=3}^{d}(q^{2i}-1)}{\prod_{i=1}^{d-1}(q^i-1)}(1-q)=c(d)(1-q^{d}),$$ which equals $\Intc(x)$ by Lemma \ref{lem:Intcformula}.
\item For $x\in \Lambda^\vee\setminus \Lambda$ with $\val(x)\ge0$, we have $$\Int_{\mathcal{V}(\Lambda')}(x)=  \begin{cases}
    (1-q), & x\in \Lambda',\\
    1,  & x\in \Lambda'^\vee\setminus \Lambda'.
  \end{cases}
$$ The number of $\Lambda'$ such that $x\in \Lambda'$ is the number of $(d-2)$-dimensional totally isotropic subspaces in $W_{\Lambda+\langle x\rangle}$, which equals $S_{d-2}(W_{\Lambda+\langle x\rangle})$. The number of $\Lambda'$ such that $x\in \Lambda'^\vee\setminus \Lambda'$ is the number of $(d-1)$-dimensional totally isotropic subspaces  $W\subseteq W_\Lambda$ such that $x\not\in W$ but $x\in W^\perp$. In this case the map $W\mapsto W+\langle x\rangle/\langle x\rangle$ gives a surjection onto the set of $(d-1)$-dimensional totally isotropic subspaces in $\langle x\rangle^\perp/\langle x\rangle$, whose fiber has size equal to the number of $(d-1)$-dimensional subspaces of $W+\langle x\rangle$ not containing $\langle x\rangle$. Hence the number of such $W$ is equal to $S_{d-1}(W_{\Lambda+\langle x\rangle})\cdot q^{d-1}$. Since $\dim_\kappa W_{\Lambda+\langle x\rangle}=2d$ and $\chi(W_{\Lambda+\langle x\rangle})=-1$, the right-hand-side evaluates to
\begin{align*}
  &\quad  \frac{c(d)}{c'(d)}(S_{d-2}(W_{\Lambda+\langle x\rangle})(1-q)+S_{d-1}(W_{\Lambda+\langle x\rangle})\cdot q^{d-1})\\ &=\frac{c(d)}{\prod_{i=2}^d(1+q^{i+1})}\left(\frac{(q^{d}+1)(q^{2}-1)\prod_{i=3}^{d-1}(q^{2i}-1)}{\prod_{i=1}^{d-2}(q^{i}-1)}(1-q)+\frac{(q^{d}+1)(q-1)\prod_{i=2}^{d-1}(q^{2i}-1)}{\prod_{i=1}^{d-1}(q^{i}-1)}q^{d-1}\right)\\
  &=c(d),
\end{align*} which equals $\Intc(x)$ by Lemma \ref{lem:Intcformula}.
\item If $x\not\in \Lambda^\vee$ or $\val(x)<0$, then both sides are zero.\qedhere
\end{altenumerate}
\end{proof}

\begin{corollary}\label{pro:Intc}
   Let $\Lambda\in \Ver^{2d+2}(\mathbb{V})$. Then $\Intc\in \Ss(\mathbb{V})$ satisfies $$\wh{\Intc}=\gamma_\mathbb{V}\,\Intc.$$
\end{corollary}

\begin{proof}
  It follows immediately from Lemmas \ref{pro:Intc} and \ref{lem: FT Int V}.
\end{proof}

\begin{theorem}[$K$-theoretic local modularity]\label{thm:highermodularity}
Let $\Lambda\in \Ver^{2d+2}(\mathbb{V})$.  For any linear map $l: \barK\rightarrow \mathbb{Q}$, the function $l\circ \IntK$ extends to a (necessarily unique) function in $\Ss(\mathbb{V})$ and satisfies $$\wh{l\circ\IntK}=\gamma_\mathbb{V}\,  l\circ\IntK.$$Here, we refer to Definition \ref{def:K CV} for $\IntK$.
\end{theorem}

\begin{proof}
  The same proof of \cite[Theorem 6.4.9]{LZ2019} works using Corollary \ref{pro:Intc}.
\end{proof}

Now we return to the function $\Int_{\CV(\Lambda)}$ defined by \eqref{eq:Int CV}.

\begin{corollary}[Higher local modularity]\label{cor:highermodularity}
  Let $\Lambda\in \Ver^{2d+2}(\mathbb{V})$. Then $\Int_{\CV(\Lambda)}$ extends to a (necessarily unique) function in $\Ss(\mathbb{V})$ and satisfies
$$
\wh{\Int_{\CV(\Lambda)}}=\gamma_\BV\Int_{\CV(\Lambda)}.
$$
\end{corollary}

\begin{proof}
   The same proof of \cite[Corollary 6.4.10]{LZ2019} works using Theorem \ref{thm:highermodularity}.
\end{proof}

\begin{remark}\label{rem:avoid B3}
  Corollary \ref{cor:highermodularity} allows us to give an alternative proof of Corollary \ref{cor:FT int} without a priori knowing that only the $(n-1)$-th graded piece of the derived special cycle contributes to $\Int_{L^\flat, \sV}(x)$ in the decomposition (\ref{eq:Intdecomp}), in particular, without using \cite[(B.3)]{Zhang2019} for a formal scheme.
\end{remark}

\section{Fourier transform: the analytic side}
\label{s:FT ana}

We continue with the steup of \S\ref{sec:kudla-rapop-cycl}, but we also allow $F$ to be any non-archimedean local field of odd residue characteristic in this subsection. Fix an $O_F$-lattice $L^\flat\subset\BV$ of rank $n-1$, and denote by $\mathbb{W}=(L^\flat_F)^\perp\subseteq \BV$.

\subsection{The partial Fourier transform $\pDenp$}

\begin{definition}
For $x\in \mathbb{W}^\mathrm{an}$, define the \emph{partial Fourier transform} of $\pDen_{L^\flat,\sV}$ by $$\pDenp(x):=\int_{L^\flat_F}\pDen_{L^\flat, \sV}(y+x)\rd y.$$
\end{definition}

Our main goal is this section is to prove the following recurrence  relations for the partial Fourier transform $\pDenp$.

\begin{proposition}\label{prop:pDenp}\quad
   \begin{altenumerate}
  \item\label{item:pDenp1}If $L^\flat$  is co-anisotropic, then $$\pDenp(x)=\pDenp(\varpi^{-1}x),$$ for all  $x\in \mathbb{W}^{\circ\circ}\cap \mathbb{W}^\mathrm{an}$.
  \item\label{item:pDenp2} If  $L^\flat$  is co-isotropic, then $$\pDenp(x)-\pDenp(\varpi^{-1}x)$$ is a constant  for all $x\in \mathbb{W}^{\circ\circ}\cap \mathbb{W}^\mathrm{an}$.
\end{altenumerate}
\end{proposition}

\begin{proof}
  By Lemma \ref{lem:latticemebed} (\ref{item:embed1}) and Corollary \ref{cor: pDen}, we obtain $$\pDenp(x)=\int_{L^\flat_F}\sum_{L^\flat\subseteq L'\subseteq L'^\vee\atop L'^\flat\not\in \Hor(L^\flat)}\wt(t(L'), \sgn_{n+1}(L'))\mathbf{1}_{L'}(y+x)\rd y.$$ Here $L'$ runs over $O_F$-lattices of rank $n$ in the non-degenerate quadratic space $L^\flat_F+\langle x\rangle_F$.

  Let $\pi_\flat: \mathbb{V}\rightarrow L_F^\flat$ be the orthogonal projection sending $x=x^\flat+x^\perp$ to $x^\flat$. We may break the sum according to $L'\cap L^\flat_F$ and $\pi_\flat(L')$. Define $$\mathrm{Lat}(L'^\flat, \wit L'^\flat):=\{L': L'\cap L^\flat_F=L'^\flat, \pi_\flat(L')=\wit L'^\flat\}.$$ 
Then $$\pDenp(x)=\sum_{L^\flat\subseteq L'^\flat\atop L'^\flat\not\in \Hor(L^\flat)} \sum_{L'^\flat\subseteq \wit L'^\flat \atop\wit L'^\flat/L'^\flat \text{ cyclic}}\sum_{L'\in \mathrm{Lat}(L'^\flat, \wit L'^\flat)}\wt(t(L'),\sgn_{n+1}(L'))\int_{L^\flat_F}\mathbf{1}_{L'}(y+x)\rd y.$$ Notice that the sum is absolutely convergent which allows us to change the order of integration and summation. 

Fix $u^\flat\in L^\flat_F$ such that $\wit L'^\flat=L'^\flat+\langle u^\flat\rangle$ (i.e., $u^\flat$ is a generator of the cyclic $O_F$-module $\wit L'^\flat/L'^\flat$). For $L'$ in the sum, we may write $L'=L'^\flat+\langle u\rangle$ for some $u=u^\flat+u^\perp$. Write $x=\lambda u^\perp$ for some $\lambda \in F^\times$. Then $y+x\in L'$ if and only if $\val(\lambda)\ge0$ and $y-\lambda u^\flat\in L'^\flat$. It follows that $$\int_{L^\flat_F}\mathbf{1}_{L'}(y+x)- \mathbf{1}_{L'}(y+\varpi^{-1}x)\rd y=
 \begin{cases}
   \vol(L'^\flat), & \langle u^\perp\rangle=\langle  x\rangle ,\\
   0, & \text{otherwise}.
 \end{cases}$$ Hence the difference $$\pDenp(x)-\pDenp(\varpi^{-1}x)=\sum_{L^\flat\subseteq L'^\flat\atop L'^\flat\not\in \Hor(L^\flat)} \vol(L'^\flat)\sum_{L'^\flat\subseteq \wit L'^\flat \atop\wit L'^\flat/L'^\flat \text{ cyclic}}\sum_{u^\perp \in\langle x\rangle \text{ generator}\atop L'=L'^\flat+\langle u^\flat+u^\perp\rangle}\wt(t(L'),\sgn_{n+1}(L')).$$ Since $x\in \mathbb{W}^{\circ\circ}$, the integrality of $L'$ implies that $u^\flat$ is also integral, and $u^\flat\in (L'^\flat)^\vee$. Hence the two inner sums become $$\sum_{u^\flat\in (L'^\flat)^\vee/L'^\flat\atop \val(u^\flat)\ge0}\sum_{L'=L'^\flat+\langle u^\flat+x\rangle}\wt(t(L'),\sgn_{n+1}(L')).$$ It remains to show this sum is a constant independent of $x\in \mathbb{W}^{\circ\circ}\cap \mathbb{W}^\mathrm{an}$ and is zero when $L^\flat$ is co-anisotropic.

  Choose an orthogonal basis $\{e_1,\ldots,e_{n-1}\}$ of $L'^\flat$ such that $\val(e_1),\ldots,\val(e_t)>0$, where $t=t(L'^\flat)$. Let $L_0:=\langle e_{t+1},\ldots, e_{n-1}\rangle$, a self-dual lattice. Let $L_1:=\langle e_1,\ldots, e_t\rangle$. Then $L'^\flat=L_0 \obot L_1$. Let $u_1$ be the orthogonal projection of $u^\flat$ to $L_1$.  Since $x\in \mathbb{W}^{\circ\circ}$, looking at the fundamental matrix of $L'$ for the basis $\{e_1,\ldots, e_{n-1}, u^\flat+x\}$ we obtain
  \begin{equation}
    \label{eq:threecases}
    t(L')=
 \begin{cases}
   t(L'^\flat)+1, & \val(u_1)>0, \text{ and }\val((e_i, u_1))>0 \text{ for all }i=1,\ldots, t, \\
   t(L'^\flat), & \val(u_1)=0, \text{ and }\val((e_i, u_1))>0 \text{ for all }i=1,\ldots, t,\\
   t(L'^\flat)-1, & \val(u_1)\ge0, \text{ and }\val((e_i,u_1))=0\text{ for some } i=1,\ldots, t.
 \end{cases}
\end{equation}
Notice that the three cases in (\ref{eq:threecases}) exactly correspond to the three conditions $$u_1\in (\varpi L_1^\vee)^{\circ\circ}/L_1, \quad u_1\in (\varpi L_1^\vee)^{\circ}\setminus(\varpi L_1^\vee)^{\circ\circ}/L_1,\quad u_1\in (L_1^\vee)^{\circ}\setminus(\varpi L_1^\vee)^{\circ}/L_1.$$ Thus to show the desired constancy we need to show that the weight factor $\wt(t(L'),\sgn_{n+1}(L'))$ depends only on $L'^\flat$ and $u_1$, but not on $x$.

    We have two cases:
    \begin{altenumerate}
    \item     If $t(L'^\flat)$ is odd, we compute that $\wt(t(L'),\sgn_{n+1}(L'))$ equals 
        $$2\prod_{1\le i<(t(L'^\flat)-2)/2}(1-q^{2i})\cdot
      \begin{cases}
        (1-q^{t(L'^\flat)-1}), & u_1\in (\varpi L_1^\vee)^{\circ\circ}/L_1, \\
        1+\varepsilon \sgn_{n+1}(L')q^{(t(L'^\flat)-1)/2}, & u_1\in (\varpi L_1^\vee)^{\circ}\setminus(\varpi L_1^\vee)^{\circ\circ}/L_1,\\
        1, & u_1\in (L_1^\vee)^{\circ}\setminus(\varpi L_1^\vee)^{\circ}/L_1.
      \end{cases}$$
   The only possible dependence on $x$ is when $u_1\in (\varpi L_1^\vee)^{\circ}\setminus(\varpi L_1^\vee)^{\circ\circ}/L_1$. In this case, the rank of the self-dual part $L_0\obot\langle u_1\rangle$ of $L'$ has  the same parity as $n+1$, and $$\sgn_{n+1}(L')=\chi(L_0\obot\langle u_1\rangle),$$ which depends only on $L_1$ and $u_1$ as desired.  Moreover,
 when $\chi(L^\flat)\ne0$ and $L^\flat$ is co-anisotropic, we have $\chi(L^\flat)=-\varepsilon$. In this case let $\langle u_1\rangle_F^\perp$  be the orthogonal complement of $\langle u_1\rangle_F$ in $L_{1,F}$, which has even dimension $t(L'^\flat)-1$. Then $\chi(L_0\obot\langle u_1\rangle)=\chi(\langle u_1\rangle^\perp_F)\chi(L^\flat)$, and hence $$\varepsilon\sgn_{n+1}(L')=-\chi(\langle u_1\rangle^\perp_F).$$  Thus we obtain the desired vanishing result when $L^\flat$ is co-anisotropic by Proposition \ref{prop:countingodd} below applied to the lattice $L_1$.
    \item If $t(L'^\flat)$ is even, we compute that $\wt(t(L'),\sgn_{n+1}(L'))$ equals
      $$2\prod_{1\le i<(t(L'^\flat)-2)/2}(1-q^{2i})\cdot\begin{cases}
        (1+\varepsilon\sgn_{n+1}(L')q^{t(L'^\flat)/2})(1-q^{t(L'^\flat)-2}), &  u_1\in (\varpi L_1^\vee)^{\circ\circ}/L_1, \\
        1-q^{t(L'^\flat)-2}, & u_1\in (\varpi L_1^\vee)^{\circ}\setminus(\varpi L_1^\vee)^{\circ\circ}/L_1, \\
        1+\varepsilon\sgn_{n+1}(L')q^{t(L'^\flat)/2-1}, & u_1\in (L_1^\vee)^{\circ}\setminus(\varpi L_1^\vee)^{\circ}/L_1.
      \end{cases}$$
      When $u_1\in (\varpi L_1^\vee)^{\circ\circ}/L_1$,  we have $$\sgn_{n+1}(L')=\chi(L_0).$$ When $u_1\in (L_1^\vee)^{\circ}\setminus(\varpi L_1^\vee)^{\circ}/L_1$, we have $$\sgn_{n+1}(L')=\chi(L_0)\chi(\langle e_1, u_1\rangle)=\chi(L_0).$$ Hence the independence on $x$ follows. Moreover,   when $\chi(L^\flat)\ne0$ and $L^\flat$ is co-anisotropic, we have $\chi(L'^\flat_F)=-\varepsilon$ and so $$\varepsilon \chi(L_0)=-\chi(L_1).$$
    \end{altenumerate}
Thus we obtain the desired vanishing result when $L^\flat$ is co-anisotropic by Proposition \ref{prop:countingeven} below applied to the lattice $L_1$. 
\end{proof}

\subsection{Weighted counting identities} In this subsection we proved the weighted counting identities needed in the proof of Proposition \ref{prop:pDenp}. 

\begin{definition}
  Assume that $t>1$. Let $L$ be a quadratic $O_F$-lattice of rank $t$ and type $t$. Set $$\mu^+(L):=\#(\varpi L^{\vee})^{\circ\circ}/L, \quad \mu^0(L):=\#((\varpi L^\vee)^{\circ}\setminus(\varpi L^\vee)^{\circ\circ})/L,  \quad \mu^{-}(L)=\#((L^\vee)^{\circ}\setminus (\varpi L^\vee)^\circ)/L.$$ If $\chi(L)=0$, for any $s\in\{\pm1\}$ set $$\mu^{0,s}(L):=\#\{x \in (\varpi L^\vee)^{\circ}\setminus(\varpi L^\vee)^{\circ\circ}:\chi(\langle x\rangle_F)=s\}/L.$$ If $\chi(L)\ne0$, for any $s\in\{\pm1\}$ set $$\mu^{0,s}(L):=\#\{x \in (\varpi L^\vee)^{\circ}\setminus(\varpi L^\vee)^{\circ\circ}:\chi(\langle x\rangle^\perp_F)=s\}/L,$$ where $\langle x\rangle_F^\perp$ is the orthogonal complement of $\langle x\rangle_F$ in $L_F$.
\end{definition}

\begin{definition}
    Assume that $t>1$. Let $L$ and $M$ be quadratic $O_F$-lattices of rank $t$ and type $t$ such that $L\subseteq M\subseteq \varpi^{-1}L$. Set $$\mu^{?}(L,M):=\mu^{?}(L)-[L:M]\mu^{?}(M),$$ where $?\in \{+,0,-,\{0,+\}, \{0,-\}\}$.
\end{definition}

\begin{lemma}\label{lem:countingodd1}
  Assume that $t>1$. Let $L$ and $M$ be quadratic $O_F$-lattices of rank $t$ and type $t$ such that $L\subseteq M\subseteq \varpi^{-1}L$. Then $$\mu^+(L,M)+\mu^0(L,M)+\mu^-(L,M)=q^{t-1}\cdot \mu^+(L,M).$$
\end{lemma}

\begin{proof}
  Notice that $(\varpi L^\vee)^\circ\subseteq (M^\vee)^\circ$ as $M\subseteq \varpi^{-1}L$, to prove the desired identity it remains to show that the multiplication-by-$\varpi$ map: $$((L^\vee)^\circ\setminus (M^\vee)^\circ)/L\rightarrow ((\varpi L^\vee)^{\circ\circ}\setminus (\varpi M^\vee)^{\circ\circ})/L$$ is surjective with every fiber of size $q^{t-1}$. Take $x\in \varpi L^\vee$ in the target. The fiber at $x$ is given by the set $$\{\varpi^{-1}(y+x)\in (L^\vee)^\circ: y\in L/\varpi L\}.$$ Notice that $x\in \varpi L^\vee$, the condition $\varpi^{-1}(y+x)\in (L^\vee)^\circ$ is equivalent to $$y\in (L\cap \varpi L^\vee)/\varpi L=L/\varpi L,\quad (y+x, y+x)\equiv 0\pmod{\varpi^2}.$$   Let $\{e_1,\ldots, e_t\}$ be a standard orthogonal basis of $L$ with $(e_i,e_i)=\epsilon_i\varpi^{a_i}$, $\epsilon_i\in O_F^\times$. Then $$\varpi L^\vee=\bigoplus_i \langle \varpi^{-a_i+1}e_i\rangle.$$ Write $$x=\sum_{i} \lambda_i\varpi^{-a_i+1}e_i\in \varpi L^\vee,\quad \lambda_i\in O_F,\quad y=\sum_{i}\mu_ie_i\in L,\quad \mu_i\in O_F.$$ Then $$(y, y)\equiv \sum_{a_i=1}(\mu_ie_i, \mu_ie_i)=\sum_{a_i=1}\mu_i^2\epsilon_i \varpi\pmod{\varpi^2},$$ $$(y,x)=\sum_{i}(\mu_ie_i,\lambda_i\varpi^{-a_i+1}e_i)=\sum_{i}\mu_i\lambda_i\epsilon_i\varpi.$$ Since $x\in (\varpi L^\vee)^{\circ\circ}$, we know that the condition $(y+x,y+x)\equiv0\pmod{\varpi^2}$ is equivalent to the equation $$\varpi^{-1}(x,x)+\sum_{a_i=1}\left(\mu_i^2\epsilon_i+2\mu_i\lambda_i\epsilon_i\right)+\sum_{a_i>1}2\mu_i\lambda_i\epsilon_i\equiv0\pmod{\varpi}.$$ Since $x\not\in \varpi M^\vee$, we know that $\lambda_i\not\equiv0\pmod{\varpi}$ for some $i$ such that $a_i>1$. Thus we may choose arbitrary $\mu_i\in O_F/\varpi O_F$ for any $i$ such that $a_i=1$ and solve for the equation for $\mu_i\in O_F/\varpi O_F$'s with $a_i>1$, $$\sum_{a_i>1}2\mu_i\lambda_i\epsilon_i+b\equiv0\pmod{\varpi},$$ where $$b=\varpi^{-1}(x,x)+\sum_{a_i=1}\left(\mu_i^2\epsilon_i+2\mu_i\lambda_i\epsilon_i\right).$$ This is a nontrivial linear equation in $t$ variables over $\mathbb{F}_q=O_F/\varpi O_F$, so the total number of solutions is exactly $q^{t-1}$. This finishes the proof.
\end{proof}

Now the discussion will depend on the parity of the type. First we consider the case of odd type.

\begin{lemma}\label{lem:countingodd2}
  Assume that $t>1$ is odd. Let $L$ be a  quadratic $O_F$-lattice of rank $t$ and type $t$. Assume that $\chi(L)\ne0$. Then there exists a quadratic $O_F$-lattice $M$ of rank $t$ and type $t$ such that $L\subseteq^1 M$ and $$\mu^{0,+}(L,M)=\mu^{0,-}(L,M).$$
\end{lemma}

\begin{proof}
We need to prove that $$\#\{x\in (\varpi L^\vee)^{\circ}\setminus (\varpi M^\vee)^{\circ}: \chi(\langle x\rangle_F^\perp)=+1\}/L=\#\{x\in (\varpi L^\vee)^{\circ}\setminus (\varpi M^\vee)^{\circ}: \chi(\langle x\rangle_F^\perp)=-1\}/L.$$ 
Let $\{e_1,\ldots, e_t\}$ be a standard orthogonal basis of $L$ with $(e_i,e_i)=\epsilon_i\varpi^{a_i}$, $\epsilon_i\in O_F^\times$. We distinguish two cases. 
    \begin{enumerate}[label=(\roman*), wide, labelindent=0pt]
    \item If $a_t\ge3$, then we may choose $M=\langle e_1,\ldots, e_{t-1},\varpi^{-1}e_t\rangle$ with fundamental invariants $(a_1,\ldots, a_{t-1},\allowbreak a_t-2)$. In this case $$\varpi L^\vee=\bigoplus_i \langle \varpi^{-a_i+1}e_i\rangle,\quad \varpi M^\vee=\left(\bigoplus_{i<t} \langle \varpi^{-a_i+1}e_i\rangle \right)\oplus \langle \varpi^{-a_t+2}e_t\rangle.$$ We fix an $$x_0=\sum_{i<t} \lambda_i\varpi^{-a_i+1}e_i,\quad \lambda_i\in O_F,$$ and consider the sets for $s\in \{\pm1\}$, $$S^s:=\{x\in (\varpi L^\vee)^{\circ}\setminus (\varpi M^\vee)^{\circ}: x=x_0+\lambda_t \varpi^{-a_t+1} e_t, \lambda_t\in O_F, \chi(\langle x\rangle_F^\perp)=s\}/L.$$ It suffices to show that $\#S^+=\# S^-$. Notice that $x\not\in \varpi M^\vee$ if and only if $\lambda_t\in O_F^\times$. We compute $$(x,x)=(x_0,x_0)+\lambda_t^2\epsilon_t\varpi^{-a_t+2}.$$ Hence $x\in S^s$ if and only if $$(x_0, x_0)+\lambda_t^2\epsilon_t\varpi^{-a_t+2}\in O_F^\times,$$ and $$\chi((x_0, x_0)+\lambda_t^2\epsilon_t\varpi^{-a_t+2})=s \chi(L_F).$$

      Write $$\lambda_t=\sum_{i=0}^\infty b_i \varpi^i,\quad -\varpi^{a_t-2}\epsilon_t^{-1}(x_0,x_0)=\sum_{i=0}^\infty c_i \varpi^{i},$$ the $\varpi$-adic expansions with $b_i,c_i\in O_F/\varpi O_F=\mathbb{F}_q$. Then $x\in S^s$ is equivalent the following equations
      \begin{align*}
        c_0 & =b_0^2\\
        c_1 & =2b_0b_1 \\
        c_2 & = 2b_0b_2+b_1^2\\
        & \cdots\\
        c_{a_t-3} &= b_0b_{a_t-3}+ b_1b_{a_t-4}\cdots +b_{a_t-3}b_0\\
        s\chi(L_F) &= \chi(c_{a_t-2}-b_0b_{a_t-2}+ b_1b_{a_t-3}\cdots +b_{a_t-2}b_0).
      \end{align*} It is clear that the number of solutions of $b_0,\ldots, b_{a_t-2}$ is independent of $s\in\{\pm1\}$, and thus $\#S^+=\# S^-$ as desired.
    \item If $a_t=a_{t-1}=2$, then we may choose $M$ with fundamental invariants $(a_1,\ldots, a_{t-2}, a_{t-1}-1, a_t-1)$. We may choose $\{f_1,\ldots,f_t\}$ be an orthogonal basis of $\varpi M^\vee$ such that $$\val(f_i)=-a_i+2,\ i=1,\ldots, t-2,\quad  \val(f_{t-1})=\val(f_{t})=1.$$ and  $$\varpi L^\vee=\langle f_1,\ldots, f_{t-1}, \varpi^{-1}(f_{t-1}+f_t)\rangle,\quad \varpi M^\vee=\langle f_1,\ldots, f_t\rangle.$$ We fix an $$x_0=\sum_{i<t} \lambda_if_i,\quad \lambda_i\in O_F,$$ and consider the sets for $s\in \{\pm1\}$, $$S^s:=\{x\in (\varpi L^\vee)^{\circ}\setminus (\varpi M^\vee)^{\circ}: x=x_0+\lambda_t \varpi^{-1}(f_{t-1}+f_t), \lambda_t\in O_F, \chi(\langle x\rangle_F^\perp)=s\}/L.$$ It suffices to show that $\#S^+=\# S^-$. Notice that $x\not\in \varpi M^\vee$ if and only if $\lambda_t\in O_F^\times$. We compute $$(x,x)=(x_0,x_0)+2\lambda_{t-1}\lambda_t\varpi^{-1}(f_{t-1},f_{t-1})+\lambda_t^2\varpi^{-2}((f_{t-1},f_{t-1})+(f_t,f_t)).$$ Similarly write down the equation for the $\varpi$-adic expansion of $\lambda_t$ required for $x\in S^{s}$, we see the number of solutions is independent of $s\in \{\pm1\}$ as desired.   
      \qedhere
    \end{enumerate}
\end{proof}

\begin{proposition}\label{prop:countingodd}
  Assume that $t>1$ is odd. Let $L$ be a quadratic $O_F$-lattice of rank $t$ and type $t$.   Then $$(1-q^{t-1})\mu^+(L)+(1- s q^{(t-1)/2})\mu^{0,+}(L)+(1+sq^{(t-1)/2})\mu^{0,-}(L)+\mu^-(L)=0$$ when $\chi(L)=0$ and $s\in\{\pm1\}$,  or $\chi(L)\ne0$ and $s=+1$.
\end{proposition}

\begin{proof}
 We induct on $\val(L)$. Since $L$ has type $t$, we know that $\val(L)\ge t$. We have two base cases depending on the parity of $\val(L)$. Let $\{e_1,\ldots, e_t\}$ be a standard orthogonal basis of $L$ with $(e_i,e_i)=\epsilon_i\varpi^{a_i}$, $\epsilon_i\in O_F^\times$.
  \begin{enumerate}[wide,labelindent=0pt]
  \item if $\val(L)=t$, we know that $L$ is a vertex lattice of type $t$ and $\chi(L)=0$. So $L^\vee/L$ is a nondegenerate $\mathbb{F}_q$-quadratic space of odd dimension $t$. By definition we have $(L^\vee)^{\circ}/L$ is the set of isotropic vectors in $L^\vee/L$ and $(\varpi L^\vee)^\circ/L=(\varpi L^\vee)^{\circ\circ}/L=\{0\}$. Hence $$\#(L^\vee)^\circ/L=q^{t-1},\quad \#(\varpi L^\vee)^\circ/L=1,\quad \#(\varpi L^\vee)^{\circ\circ}/L=1.$$ So $$\mu^-(L)=q^{t-1}-1, \quad\mu^{0}(L)=0,\quad \mu^+(L)=1$$ satisfies the desired identity for any $s\in\{\pm1\}$.
  \item if $\val(L)=t+1$, we know that $L$ has fundamental invariants $(1,\ldots,1,2)$ and $\chi(L)\ne0$. Then $L^\vee=\langle \varpi^{-1}e_1,\ldots,\varpi^{-1}e_{t-1},e_t\rangle$. Let $$x=\lambda_1\varpi^{-1}e_1+\cdots+\lambda_{t-1}\varpi^{-1}e_{t-1}+\lambda_t\varpi^{-2}e_{t}\in L^\vee,\quad \lambda_i\in O_F.$$ Then $x\in (L^\vee)^\circ$ if and only if $\varpi^2(x,x)\equiv0\pmod{\varpi^2}$, namely, $$\sum_{i=1}^{t-2}\lambda_i^2\varpi+\lambda_{t-1}^2\epsilon_{t-1}\varpi+\lambda_t^2\epsilon_t\equiv0\pmod{ \varpi^2},$$ or equivalently $$\sum_{i=1}^{t-2}\lambda_i^2+\lambda_{t-1}^2\epsilon_{t-1}\equiv0\pmod{\varpi},\quad \lambda_t=0\pmod{\varpi}.$$ It follows that $$\#(L^\vee)^\circ/L=((q^{(t-1)/2}-s)(q^{(t-3)/2}+s)+1)q.$$ Here $s=\chi(\langle e_1,\ldots,e_{t-1}\rangle)\in\{\pm1\}$.  Similarly let $$x=\lambda_1e_1+\cdots +\lambda_{t-1}e_{t-1}+\lambda_t\varpi^{-1}e_t\in \varpi L^\vee,\quad \lambda_i\in O_F.$$ Then $x\in (\varpi L^\vee)^{\circ}$ always, and $x\in (\varpi L^\vee)^{\circ\circ}$ if and only if $\lambda_t\equiv 0\pmod{\varpi}$. Hence $$\#(\varpi L^\vee)^{\circ}=q,\quad \#(\varpi L^\vee)^{\circ\circ}=1.$$ So $$\mu^-(L)=(q^{(t-1)/2}\pm 1)(q^{(t-3)/2}\mp1)q,\quad \mu^0(L)=q-1, \quad \mu^+(L)=1.$$ Notice that when $x\in (\varpi L^\vee)^\circ\backslash (\varpi L^\vee)^{\circ\circ}$, we have $\chi(\langle x\rangle_F^{\perp})=\chi(\langle e_1,\ldots, e_{t-1}\rangle)=s$, and thus $$\mu^{0,s}(L)=q-1,\quad \mu^{0,-s}(L)=0.$$ Hence the we have the desired identity $$ \left(q^{\frac{t-1}{2}}-s\right)\left(q^{\frac{t-3}{2}}+s\right)q +(q-1)\left(1-sq^{\frac{t-1}{2}}\right)+\left(1-q^{t-1}\right)=0.$$ 
  \end{enumerate}

  Now we run induction on $\val(L)$.  Let  $M$ be an $O_F$-lattice of rank $t$ and type $t$ such that $L\subseteq^1 M$ as in Lemma \ref{lem:countingodd2}. Then $\val(M)<\val(L)$. 
  By induction hypothesis it remains to prove the similar identity for the difference $$(1-q^{t-1})\cdot \mu^+(L,M)+(1-sq^{(t-1)/2})\cdot \mu^{0,+}(L,M)+(1+sq^{(t-1)/2})\cdot \mu^{0,-}(L,M)+\mu^{-}(L,M)=0,$$ for any $s\in \{\pm1\}$. This follows from the two stronger relations in Lemmas \ref{lem:countingodd1} and \ref{lem:countingodd2}. \qedhere.  
\end{proof}

Next we consider the case of even type.

\begin{lemma}\label{lem:countingeven2}
  Assume that $t>1$ is even. Let $L$ be a  quadratic $O_F$-lattice of rank $t$ and type $t$. Then there exists a quadratic $O_F$-lattice $M$ of rank $t$ and type $t$ such that $L\subseteq^1 M$ and $$\mu^+(L,M)+\mu^0(L,M)=q\cdot \mu^+(L,M).$$
\end{lemma}

\begin{proof}
We need to prove that $$q\cdot \#((\varpi L^\vee)^{\circ\circ}\setminus (\varpi M^\vee)^{\circ\circ})/L=\#((\varpi L^\vee)^{\circ}\setminus (\varpi M^\vee)^{\circ})/L.$$ Let $\{e_1,\ldots, e_t\}$ be a standard orthogonal basis of $L$ with $(e_i,e_i)=\epsilon_i\varpi^{a_i}$, $\epsilon_i\in O_F^\times$. We distinguish two cases. 
    \begin{enumerate}[label=(\roman*), wide, labelindent=0pt]
    \item If $a_t\ge3$, then we may  choose $M=\langle e_1,e_2,\dots, \varpi^{-1}e_t\rangle$ (with fundamental invariants $(a_1,\ldots,a_{t-1},\allowbreak a_t-2)$. Notice that $$\varpi L^\vee=\bigoplus_i \langle \varpi^{-a_i+1}e_i\rangle,\quad \varpi M^\vee=\left(\bigoplus_{i<t} \langle \varpi^{-a_i+1}e_i\rangle \right)\oplus \langle \varpi^{-a_t+2}e_t\rangle.$$ We fix an $$x_0=\sum_{i<t} \lambda_i\varpi^{-a_i+1}e_i,\quad \lambda_i\in O_F,$$ and consider the sets $$S^\circ:=\{x\in (\varpi L^\vee)^{\circ}\setminus (\varpi M^\vee)^{\circ}: x=x_0+\lambda_t \varpi^{-a_t+1} e_t, \lambda_t\in O_F\}/L.$$ $$S^{\circ\circ}:=\{x\in (\varpi L^\vee)^{\circ\circ}\setminus (\varpi M^\vee)^{\circ\circ}: x=x_0+\lambda_t \varpi^{-a_t+1} e_t, \lambda_t\in O_F\}/L.$$ It suffices to show that $q\cdot \#S^{\circ\circ}=\#S^{\circ}$. Notice that $x\not\in \varpi M^\vee$ if and only if $\lambda_t\in O_F^\times$. We compute $$(x,x)=(x_0,x_0)+\lambda_t^2\epsilon_t\varpi^{-a_t+2}.$$ Hence $x\in S^{\circ}$ if and only if $$\varpi^{a_t-2}(x_0,x_0)+\lambda_t^2\epsilon_t\equiv0\pmod{\varpi^{a_t-2}},\quad \lambda_t\in O_F^\times,$$ and $x\in S^{\circ\circ}$ if and only if $$\varpi^{a_t-2}(x_0,x_0)+\lambda_t^2\epsilon_t\equiv0\pmod{\varpi^{a_t-1}},\quad \lambda_t\in O_F^\times.$$ Write $\lambda_t=b_0+b_1\varpi+b_2\varpi^2+\cdots$ to be the $\varpi$-adic expansion of $\lambda_t$. Then $x\in S^{\circ\circ}$ if and only if $x\in S^{\circ}$ together with an additional equation $$b_0 b_{a_t-2}=c,$$ for $c\in \mathbb{F}_q$ determined by $(x_0, b_0,\ldots,b_{a_t-3})$, which has exactly one solution $b_{a_t-2}\in \mathbb{F}_q$. Hence $q\cdot \#S^{\circ\circ}=\#S^{\circ}$ as desired.
    \item If $a_t=a_{t-1}=2$, then we may choose $M$ with fundamental invariants $(a_1,\ldots, a_{t-2}, a_{t-1}-1, a_t-1)$. We may choose $\{f_1,\ldots,f_t\}$ be an orthogonal basis of $\varpi M^\vee$ such that $$\val(f_i)=-a_i+2,\ i=1,\ldots, t-2,\quad  \val(f_{t-1})=\val(f_{t})=1.$$ and  $$\varpi L^\vee=\langle f_1,\ldots, f_{t-1}, \varpi^{-1}(f_{t-1}+f_t)\rangle,\quad \varpi M^\vee=\langle f_1,\ldots, f_t\rangle.$$ We fix an $$x_0=\sum_{i<t} \lambda_if_i,\quad \lambda_i\in O_F,$$ and consider the sets $$S^\circ:=\{x\in (\varpi L^\vee)^{\circ}\setminus (\varpi M^\vee)^{\circ}: x=x_0+\lambda_t\varpi^{-1}(f_{t-1}+f_t), \lambda_t\in O_F\}/L.$$ $$S^{\circ\circ}:=\{x\in (\varpi L^\vee)^{\circ\circ}\setminus (\varpi M^\vee)^{\circ\circ}: x=x_0+\lambda_t\varpi^{-1}(f_{t-1}+f_t), \lambda_t\in O_F\}/L.$$ It suffices to show that $q\cdot \#S^{\circ\circ}=\#S^{\circ}$. Notice that $x\not\in \varpi M^\vee$ if and only if $\lambda_t\in O_F^\times$. Similarly write down the equation for the $\varpi$-adic expansion of $\lambda_t$ required for $x\in S^{\circ}$ (resp. $S^{\circ\circ}$), we see that $q\cdot \#S^{\circ\circ}=\#S^{\circ}$ as desired.\qedhere
    \end{enumerate}
  \end{proof}

\begin{proposition}\label{prop:countingeven}
Assume that $t>1$ is even. Let $L$ be a quadratic $O_F$-lattice of rank $t$ and type $t$. Then $$(1-sq^{t/2})(1+sq^{t/2-1})\mu^+(L)+(1+sq^{t/2-1})\mu^0(L)+\mu^-(L)=0$$ when $\chi(L)=0$ and $s\in\{\pm1\}$, or $\chi(L)=s\in\{\pm1\}$.
\end{proposition}

\begin{proof}
  We induct on $\val(L)$. Since $L$ has type $t$, we know that $\val(L)\ge t$. We have two base cases depending on the parity of $\val(L)$. Let $\{e_1,\ldots, e_t\}$ be a standard orthogonal basis of $L$ with $(e_i,e_i)=\epsilon_i\varpi^{a_i}$, $\epsilon_i\in O_F^\times$.
  \begin{enumerate}[wide,labelindent=0pt]
  \item if $\val(L)=t$, we know that $L$ is a vertex lattice of type $t$ and $\chi(L)\ne0$. So $L^\vee/L$ is a nondegenerate $\mathbb{F}_q$-quadratic space of even dimension $t$. By definition we have $(L^\vee)^{\circ}/L$ is the set of isotropic vectors in $L^\vee/L$ and $(\varpi L^\vee)^\circ/L=(\varpi L^\vee)^{\circ\circ}/L=\{0\}$.  Let $s=\sgn(L^\vee/L)$. Then $$\#(L^\vee)^\circ/L=(q^{t/2}-s)(q^{t/2-1}+s)+1,\quad \#(\varpi L^\vee)^\circ/L=1,\quad \#(\varpi L^\vee)^{\circ\circ}/L=1.$$ So  $$\mu^-(L)=(q^{t/2}-s)(q^{t/2-1}+s), \quad\mu^{0}(L)=0,\quad \mu^+(L)=1$$ satisfies the desired identity for $s=\chi(L)$.
  \item if $\val(L)=t+1$, we know that $L$ has fundamental invariants $(1,\ldots,1,2)$ and $\chi(L)=0$. Then $L^\vee=\langle \varpi^{-1}e_1,\ldots,\varpi^{-1}e_{t-1},e_t\rangle$. Let $$x=\lambda_1\varpi^{-1}e_1+\cdots+\lambda_{t-1}\varpi^{-1}e_{t-1}+\lambda_t\varpi^{-2}e_{t}\in L^\vee,\quad \lambda_i\in O_F.$$ Then $x\in (L^\vee)^\circ$ if and only if $\varpi^2(x,x)\equiv0\pmod{\varpi^2}$, namely, $$\sum_{i=1}^{t-2}\lambda_i^2\varpi+\lambda_{t-1}^2\epsilon_{t-1}\varpi+\lambda_t^2\epsilon_t\equiv0\pmod{ \varpi^2},$$ or equivalently $$\sum_{i=1}^{t-2}\lambda_i^2+\lambda_{t-1}^2\epsilon_{t-1}\equiv0\pmod{\varpi},\quad \lambda_t=0\pmod{\varpi}.$$ It follows that $$\#(L^\vee)^\circ/L=q^{t-2}\cdot q=q^{t-1}.$$ Similarly let $$x=\lambda_1e_1+\cdots +\lambda_{t-1}e_{t-1}+\lambda_t\varpi^{-1}e_t\in \varpi L^\vee,\quad \lambda_i\in O_F.$$ Then $x\in (\varpi L^\vee)^{\circ}$ always, and $x\in (\varpi L^\vee)^{\circ\circ}$ if and only if $\lambda_t\equiv 0\pmod{\varpi}$. Hence $$\#(\varpi L^\vee)^{\circ}=q,\quad \#(\varpi L^\vee)^{\circ\circ}=1.$$ So $$\mu^-(L)=q^t-q,\quad \mu^0(L)=q-1, \quad \mu^+(L)=1.$$ Hence the we have the desired identity $$(1-sq^{t/2})(1+sq^{t/2-1}) +(1+sq^{t/2-1})(q-1)+(q^{t-1}-q)=0$$ for any $s\in\{\pm1\}$.
  \end{enumerate}
  
  Now we run induction on $\val(L)$.  Let $M$ be an $O_F$-lattice of rank $t$ and type $t$ such that $L\subseteq^1 M\subseteq$ as in Lemma \ref{lem:countingeven2}. Then $\val(M)<\val(L)$. By induction hypothesis it remains to prove the similar identity for the difference $$(1+sq^{t/2})(1-sq^{t/2-1})\mu^+(L,M)+(1- sq^{t/2-1})\mu^0(L,M)+\mu^{-}(L,M)=0$$ for any $s\in \{\pm1\}$.  This follows from the two stronger relations in Lemmas \ref{lem:countingodd1} and \ref{lem:countingeven2}. 
  \end{proof}

\section{Invariant distributions and the proof of the main theorem}\label{sec:invar-distr-proof}
\subsection{Invariant distributions on isotropic 2-dimensional quadratic spaces}\label{sec:invar-distr-2}

In this section, we consider a general non-degenerate 2-dimensional quadratic space $\mathbb{W}$ over a non-archimedean local field $F$ with odd residue characteristic.  The general linear group $\GL(\mathbb{W})(F)$ naturally acts on $\Ss(\mathbb{W})$ and $\Dd(\mathbb{W})$ (see Notations \S\ref{sec:notations-functions}) via $$(h\varphi)(x)=\varphi(h^{-1}x),\quad (hT)(\varphi)=T(h^{-1}\varphi),\quad h\in \GL(\mathbb{W})(F),\ \varphi\in\Ss(\mathbb{W}),\ T\in \Dd(\mathbb{W}).$$

\begin{definition}\label{def:DD}
   Let $k\ge0$. Let $\Dd(\mathbb{W})_k$ be the space of distributions $T\in \Dd(\mathbb{W})$ such that
  \begin{altenumerate}
  \item\label{item:DD1} $T$ is $\O(\mathbb{W})(F)$-invariant.
  \item $T$ is $O_F^\times$-invariant.
  \item\label{item:DD3} $\supp(T)\subseteq \mathbb{W}^{\ge-k}$.
  \item\label{item:DD4} $\supp(\hat T)\subseteq \mathbb{W}^{\circ}$.
  \end{altenumerate}
\end{definition}

We will need the following lemma (cf. the analogous identity for $\GL_2$ in \cite[(4.26)]{Bump1997}), which should be well-known but we include a proof for the sake of completeness.

\begin{lemma}\label{lem:K0pk}
  Denote by
  \begin{altitemize}
  \item $B\subseteq \SL_2$ the Borel subgroup of upper triangular matrices.
  \item $A\subseteq \SL_2$  the torus of diagonal matrices.
  \item $N\subseteq \SL_2$ (resp. $N_-\subseteq \SL_2$) the subgroup of upper (resp. lower) triangular unipotent matrices.
  \item $K_0(\varpi^k):=\{\left(\begin{smallmatrix}a & b\\ c & d\end{smallmatrix}\right): c\in \varpi^kO_F\}\subseteq\SL_2(O_F)$ for $k\geq 0$.
  \end{altitemize}
  Then we have, for $k\ge1$, $$K_0(\varpi^k)=N_-(\varpi^kO_F)A(O_F)N(O_F).$$
\end{lemma}

\begin{proof}
  Clearly the right-hand-side is contained $K_0(\varpi^k)$. Let $\left(\begin{smallmatrix}a & b\\ c & d \end{smallmatrix}\right)\in K_0(\varpi^k)$. Since $k\ge1$, we know that $a\in O_F^\times$. Then the identity $$\left(\begin{matrix}a & b\\ c & d \end{matrix}\right)=\left(\begin{matrix}1 & 0\\ a^{-1}c & 1 \end{matrix}\right)\left(\begin{matrix}a & 0\\ 0 & a^{-1} \end{matrix}\right)\left(\begin{matrix}1 & a^{-1}b\\ 0 & 1 \end{matrix}\right)$$ shows that $\left(\begin{smallmatrix}a & b\\ c & d \end{smallmatrix}\right)$ is an element of the right-hand-side.
\end{proof}

\begin{proposition}\label{prop:distributionbasis} Assume that $\mathbb{W}$ is \emph{isotropic} (equivalently, $\chi(\mathbb{W})=+1$). 
  \begin{altenumerate}
  \item\label{item:invDD1} $\Dd(\mathbb{W})_k$ has dimension $2k$ for $k\ge1$.
  \item Define functions
  \begin{align*}
      \phi_0(x)&:=\sum_{i\ge0}\mathbf{1}_{\mathbb{W}^{\ge i}}(x)=
  \begin{cases}
    \val(x)+1, & x\in \mathbb{W}^{\circ},\\
    0, & \text{otherwise},
  \end{cases}\\
  \phi_1(x)&:=\mathbf{1}_{\mathbb{W}^{\ge -1}}(x)=\begin{cases}
    1, & x\in \mathbb{W}^{\ge-1},\\
    0, & \text{otherwise},
  \end{cases}\\
  \phi_{2i-2}(x)&:=
  \begin{cases}
    1, & x\in \mathbb{W}^{=-i} \text{ and }\varpi^i(x,x)\in (O_F^\times)^2,\\
    0, & \text{otherwise},
  \end{cases} \quad i\ge2,
\\
  \phi_{2i-1}(x)&:=
  \begin{cases}
    1, & x\in \mathbb{W}^{=-i} \text{ and }\varpi^i(x,x)\not\in (O_F^\times)^2,\\
    0, & \text{otherwise},
  \end{cases}\quad i\ge2.
\end{align*}
Then $\Dd(\mathbb{W})_k$ has a basis given by the distributions represented by $\phi_0,\phi_1,\ldots,\phi_{2k-1}\in \Lloc(\mathbb{W}^\mathrm{an})$ for $k\ge1$ and $\phi_0$ for $k=0$.
\end{altenumerate}
\end{proposition}

\begin{proof}
  \begin{altenumerate}
  \item We use the Weil representation $\omega$ of $\SL_2(F)$ on $\Ss(\mathbb{W})$ (associated to our fixed unramified additive character $\psi:F\rightarrow \mathbb{C}^\times$).  Recall that $\omega$ is defined by \begin{align*}\label{eqn weil}
\omega\left(\begin{matrix} a& \\
& a^{-1}
\end{matrix}\right)\phi(x)&=|a|\phi(ax),\notag\\
\omega\left(\begin{matrix} 1&b \\
&1
\end{matrix}\right)\phi(x)&=\psi\left(\frac{1}{2}\,b\, (x,x)\right)\phi(x),\\
\omega\left(\begin{matrix} &1\\
-1&
\end{matrix}\right)\phi(x)&=\wh\phi(x) ,\notag
\end{align*}  where $a\in F^\times, b\in F$.
Let $T\in \Dd(\mathbb{W})$. By the definition of $\omega$, we have
\begin{altitemize}
\item $T$  is $O_F^\times$-invariant if and only if $T$ is invariant under $A(O_F)\subseteq\SL_2(F)$,
\item $\supp(T)\subseteq \mathbb{W}^{\ge-k}$ if and only if $T$ is invariant under $N(\varpi^{k}O_F)$.
\item $\supp(\hat T)\subseteq \mathbb{W}^\circ$ if and only if $T$ is invariant under $N_-(O_F)$.
\end{altitemize}
Since $k\ge1$, by Lemma \ref{lem:K0pk} we obtain
\begin{equation}
  \label{eq:Dk}
  \Dd(\mathbb{W})_k=\Dd(\mathbb{W})^{\O(\mathbb{W})(F)\times K_0(\varpi^k)^t},
\end{equation}
the space of $\O(\mathbb{W})(F)\times K_0(\varpi^k)^t$-invariants of $\Dd(\mathbb{W})$, where $(-)^t$ denotes the transpose.

On the other hand, consider the degenerate principal series $I(\mathbf{1}):=\Ind_{B(F)}^{\SL_2(F)}(\mathbf{1})$, which consists of locally constant functions $f: \SL_2(F)\rightarrow \mathbb{C}$ such that $$f\left(\left(\begin{smallmatrix}a & b \\ 0 & a^{-1}\end{smallmatrix}\right)g\right)=|a|f(g), \quad \left(\begin{smallmatrix}a & b \\ 0 & a^{-1}\end{smallmatrix}\right)\in B(F), \ g\in \SL_2(F),$$ and on which $\SL_2(F)$ acts via right translation. We have an $\SL_2(F)$-equivariant map (the Siegel--Weil section) $$\Ss(\mathbb{W})\rightarrow I(\mathbf{1}),\quad \phi\mapsto \omega(g)\phi(0).$$ By \cite[Proposition 1.1]{Kudla1997a}, this maps factors through the space $\Ss(\mathbb{W})_{\O(\mathbb{W})(F)}$ of $\O(\mathbb{W})(F)$-coinvariants, and  induces an isomorphism $\Ss(\mathbb{W})_{\O(\mathbb{W})(F)}\isoarrow I(\mathbf{1})$ (as $\mathbb{W}$ is the unique 2-dimensional quadratic space with $\chi(\mathbb{W})=+1$). Taking contragredients it follows that there is isomorphism as $\SL_2(F)$-representations between the space $\Dd(\mathbb{W})^{\O(\mathbb{W})(F)}$ of $\O(\mathbb{W})(F)$-invariant distributions  and the contragredient representation $I(\mathbf{1})^\vee$. Using $I(\mathbf{1})^\vee\cong I(\mathbf{1}^{-1})=I(\mathbf{1})$, we know that
\begin{equation}
  \label{eq:SWsection}
  \Dd(\mathbb{W})^{\O(\mathbb{W})(F)}\cong I(\mathbf{1})
\end{equation}
 as $\SL_2(F)$-representations. Combining (\ref{eq:Dk}) and (\ref{eq:SWsection}), we obtain $$\Dd(\mathbb{W})_k\cong I(\mathbf{1})^{K_0(\varpi^k)^t}.$$ The result then follows from the fact that $\dim I(\mathbf{1})^{K_0(\varpi^k)^t}=\dim I(\mathbf{1})^{K_0(\varpi^k)}=2k$ by the theory of newforms for $\SL_2(F)$ (\cite[Proposition 3.2.4]{Lansky2007}). 

\item Since $\phi_i$'s are clearly linearly independent, by part (\ref{item:invDD1}) it remains to show that $\phi_i\in \Dd(\mathbb{W})_k$ for $i=0,\ldots, 2k-1$ when $k\ge1$ and $\phi_0\in \Dd(\mathbb{W})_0$. The properties Definition \ref{def:DD} (\ref{item:DD1}--\ref{item:DD3}) are clearly satisfied by definition, so it remains to check Definition \ref{def:DD} (\ref{item:DD4}), i.e., $\supp(\hat\phi_i)\subseteq \mathbb{W}^\circ$. To compute  $\supp(\hat \phi_i)$, we realize $\phi_i$ as an $\O(\mathbb{W})(F)$-orbital integral on $\mathbb{W}$. Recall that for any $\varphi\in\Ss(\mathbb{W})$, its $\O(\mathbb{W})(F)$-orbital integral is defined by $$\Orb(x,\varphi):=\int_{\O(\mathbb{W})(F)}\varphi(h^{-1}x)\rd h,$$ and denote by $\wh\Orb(x,\varphi)$ its Fourier transform (in the first variable $x$). Then by definition we have an identity in $\Dd(\mathbb{W})$,
 \begin{equation}
   \label{eq:OrbFourier}
   \wh\Orb(x,\varphi)=\Orb(x,\hat \varphi).
 \end{equation}
 Since $\mathbb{W}$ is isotropic, we may choose a basis $\{e_1,e_2\}$ of $\mathbb{W}$ such that its fundamental matrix is $\left(\begin{smallmatrix}0 & 1\\ 1&0\end{smallmatrix}\right)$. Define compact open subsets of $\mathbb{W}$,
 \begin{align*}
   \Omega_0&=O_F e_1\times O_F e_2=\langle e_1,e_2\rangle,\\
   \Omega_1&=\varpi^{-1}O_F^\times e_1\times O_F e_2,\\
   \Omega_{2i-2}&=\varpi^{-i}(O_F^\times)^2e_1\times (O_F^\times)^2 e_2,\quad i\ge2,\\
   \Omega_{2i-1}&=\varpi^{-i}\epsilon(O_F^\times)^2e_1\times (O_F^\times)^2 e_2,\quad i\ge2,
 \end{align*}
 where $\epsilon\in O_F^\times\setminus (O_F^\times)^2$. Define $\varphi_i=\mathbf{1}_{\Omega_i}\in\Ss(\mathbb{W})$ for $i\ge0$. Then it is easy to see that we have  $$\phi_i(x)=\Orb(x, \varphi_i),$$ (possibly up to a nonzero scalar, which we ignore).  Hence by (\ref{eq:OrbFourier}) we obtain $$\hat\phi_i(x)=\Orb(x, \hat\varphi_i).$$ It remains to show that $\supp(\Orb(x,\hat\varphi_i))\subseteq \mathbb{W}^\circ$. 

First consider $\varphi_0$ and $\varphi_1$. Notice that $$\varphi_0=\mathbf{1}_{\langle e_1,e_2\rangle},\quad \varphi_1=\mathbf{1}_{\langle e_1,e_2\rangle}-\mathbf{1}_{\langle \varpi^{-1}e_1,e_2\rangle}.$$ Since $\langle e_1, e_2\rangle^\vee=\langle e_1, e_2\rangle$ and $\langle \varpi^{-1}e_1,e_2\rangle^\vee=\langle e_1, \varpi e_2\rangle$, by (\ref{eq:latticefourier}) we have $$\hat\varphi_0=\mathbf{1}_{\langle e_1,e_2\rangle},\quad \hat\varphi_1=\mathbf{1}_{\langle e_1,e_2\rangle}-q\cdot\mathbf{1}_{\langle e_1,\varpi e_2\rangle}.$$ Hence $\supp(\hat\varphi_i)\subseteq \langle e_1,e_2\rangle$ for $i=0,1$. The $\O(\mathbb{W})(F)$-orbits of elements in $\langle e_1,e_2\rangle$ are contained in $\mathbb{W}^\circ$, hence $\supp(\Orb(x,\hat \varphi_i))\subseteq \mathbb{W}^\circ$ for $i=0,1$ as desired. 

 Now consider $i\ge2$. Notice that $\varphi_{2i-2}$ (resp. $\varphi_{2i-1}$) is a linear combination of functions $\varphi=\mathbf{1}_\Omega$ of compact open subsets $\Omega\subseteq \mathbb{W}$ of form $$\Omega=\varpi^{-i}(a+\varpi O_F)e_1 \times (b+\varpi O_F)e_2=\langle \varpi^{-i+1}e_1, \varpi e_2\rangle+(\varpi^{-i}ae_1+be_2),$$ where $a,b\in O_F^\times$. It remains to show that $\supp(\Orb(x,\hat \varphi))\subseteq \mathbb{W}^\circ$. We compute
 \begin{align*}
   \hat\varphi(x)& =\int_{\mathbb{W}}\varphi(y)\psi((x,y))\rd y\\
   &=\int_{\mathbb{W}}\mathbf{1}_{\langle \varpi^{-i+1}e_1, \varpi e_2\rangle}(y)\psi((x, y+\varpi^{-i}a e_1+b e_2))\rd y\\
   &= \psi((x, \varpi^{-i}a e_1+b e_2))\int_{\mathbb{W}}\mathbf{1}_{\langle \varpi^{-i+1}e_1, \varpi e_2\rangle}(y)\psi((x,y))\rd y\\
   &= \psi((x, \varpi^{-i}a e_1+b e_2))\cdot \wh{\mathbf{1}}_{\langle\varpi^{-i+1}e_1, \varpi e_2\rangle}(x).
 \end{align*}
Hence by (\ref{eq:latticefourier}) we have $\supp(\hat \varphi)\subseteq \langle\varpi^{-i+1}e_1, \varpi e_2\rangle^\vee=\langle \varpi^{-1}e_1, \varpi^{i-1}e_2\rangle$. The $\O(\mathbb{W})(F)$-orbits of elements in $\langle \varpi^{-1}e_1, \varpi^{i-1}e_2\rangle$ are contained in $\mathbb{W}^\circ$ (as $i\ge2$), hence $\supp(\Orb(x,\hat \varphi))\subseteq \mathbb{W}^\circ$ as desired.\qedhere
\end{altenumerate} 
\end{proof}

\begin{corollary}\label{cor:distribution}
 Let $k\ge0$ and $T\in\Dd(\mathbb{W})_k$. Then $T$ is represented by a function $\phi\in \Lloc(\mathbb{W}^\mathrm{an}$) such that $\phi(x)-\phi(\varpi^{-1}x)$ is a constant for all $x\in \mathbb{W}^{\circ\circ}\cap \mathbb{W}^\mathrm{an}$.
\end{corollary}

\begin{proof}
This follows from Proposition \ref{prop:distributionbasis}  as each basis distribution $\phi_i$ of $\Dd(\mathbb{W})_k$ clearly satisfies the claimed property.
\end{proof}

\subsection{The partial Fourier transform $\Intp$} Now we come back to the setup of \S\ref{sec:kudla-rapop-cycl}. Fix an $O_F$-lattice $L^\flat\subset\BV$ of rank $n-1$ and denote by $\mathbb{W}=(L^\flat_F)^\perp\subseteq \BV$.

\begin{definition}
  For $x\in \mathbb{W}^\mathrm{an}$,  define the \emph{partial Fourier transform} of $\Int_{L^\flat,\sV}$ by $$\Intp(x):=\int_{L^\flat_F}\Int_{L^\flat, \sV}(y+x)\rd y.$$ By Corollary \ref{cor:LC int}, we have $\Intp\in \Lloc(\mathbb{W}^\mathrm{an})$ (and $\Intp\in\Ss(\mathbb{W})$ if $L^\flat$ is co-anisotropic).
\end{definition}

\begin{lemma}\label{lem:IntpDD}
  $\Intp\in\Dd(\mathbb{W})_k$ for some $k\ge0$.
\end{lemma}

\begin{proof}
  We show that all items (\ref{item:DD1}--\ref{item:DD4}) in Definition \ref{def:DD} are satisfied for $\Intp$.
  \begin{altenumerate}
  \item To show that $\Intp$ is $\O(\mathbb{W})(F)$-invariant, it suffices to show that the function $\Int_{L^\flat,\sV}(y+x)$ on $\mathbb{W}^\mathrm{an}$ is $\O(\mathbb{W})(F)$-invariant for any $y\in L_F^\flat$. If $x'$ lies in the $\O(\mathbb{W})$-orbit of $x$, then the two quadratic $O_F$-lattices $L^\flat+\langle y+x\rangle$ and $L^\flat+\langle y+x'\rangle$ are isometric.  By Remark \ref{rem:Intisometry} we have $\Int_{L^\flat}(y+x)=\Int_{L^\flat}(y+x')$. By Theorem \ref{thm:Int H} and (\ref{pDen H}), we also have $\Int_{L^\flat,\sH}(y+x)=\Int_{L^\flat,\sH}(y+x')$.  Hence $\Int_{L^\flat,\sV}(y+x)=\Int_{L^\flat,\sV}(y+x')$.
  \item For any $\epsilon\in O_F^\times$, by the change of variable $y\mapsto \epsilon y$ and $|\epsilon|=1$ we obtain $$\Intp(\epsilon x)=\int_{L^\flat_F}\Intp(y+\epsilon x)\rd y=\int_{L^\flat_F}\Intp(\epsilon y+\epsilon x)\rd y,$$ which equals $\Intp(x)$ as $L^\flat+\langle \epsilon y+\epsilon x\rangle=L^\flat+\langle y+x\rangle$ for $\epsilon\in O_F^\times$. So $\Intp$ is $O_F^\times$-invariant.
  \item If $\Intp(y+x)\ne0$, then $L^\flat+\langle y+x\rangle$ is integral. Hence $y\in (L^\flat)^\vee$ and the integrality of $y+x$ implies that there exists $k\ge0$ such that $\val(x)\ge -k$ when $y$ runs over $(L^\flat)^\vee/L^\flat$. So $\Intp(y+x)$ is supported on $\mathbb{W}^{-k}$ for all $y\in L^\flat_F$, and hence $\supp(\Intp)\subseteq \mathbb{W}^{\ge-k}$.
  \item By definition, for $x\in \mathbb{W}^\mathrm{an}$, we have $$\wh{\Intp}(x)=\wh{\Int_{L^\flat,\sV}}(x),$$ where $\wh{\ }$ denotes the Fourier transform on $\mathbb{W}$ (resp. $\mathbb{V}$) on the left-hand-side (resp. right-hand-side). Combining with Corollary \ref{cor:FT int} we know that $$\wh{\Intp}(x)=\gamma_\mathbb{V}\Int_{L^\flat,\sV}(x).$$ Since $\Int_{L^\flat,\sV}(x)$ is zero unless $x$ is integral, we know that $\supp(\wh\Intp)\subseteq \mathbb{W}^\circ$ as desired. \qedhere
  \end{altenumerate}
\end{proof}

Now we can prove the following recurrence  relations for the partial Fourier transform $\Intp$ (analogous to those for $\pDenp$ in Proposition \ref{prop:pDenp}).

\begin{proposition}\label{prop:Intp}\quad
  \begin{altenumerate}
  \item\label{item:Intp1}If $L^\flat$  is co-anisotropic, then $\Intp(x)$ is a constant for all $x\in \mathbb{W}^\circ$. 
  \item\label{item:Intp2} If  $L^\flat$  is co-isotropic, then $\Intp(x)-\Intp(\varpi^{-1}x)$ is a constant for all $x\in \mathbb{W}^{\circ\circ}\cap \mathbb{W}^\mathrm{an}$. 
  \end{altenumerate}

  \begin{proof}
    \begin{altenumerate}
    \item By  Lemma \ref{lem:IntpDD}, we know that $\supp(\wh\Intp)\subseteq \mathbb{W}^\circ$. Since $\mathbb{W}$ is anisotropic, we know that $\mathbb{W}^{\circ}\subseteq \mathbb{W}$ is an integral $O_F$-lattice. Hence $\Intp$ is invariant under translation by the dual lattice $(\mathbb{W}^{\circ})^\vee$, in particular, invariant under translation by $\mathbb{W}^\circ\subseteq (\mathbb{W}^\circ)^\vee$. It follows that $\Intp$ is a constant for all $x\in \mathbb{W}^\circ$.
    \item It follows immediately from Lemma \ref{lem:IntpDD} and Corollary \ref{cor:distribution}.\qedhere
    \end{altenumerate}
  \end{proof}
\end{proposition}

\subsection{The proof of Main Theorem \ref{thm: main}}\label{ss:proof}

\begin{lemma}\label{lem:decreaseval}
  Let $L^\flat\subseteq\BV$ be  an $O_F$-lattice of rank $n-1$. Let $x\in \Omega(L^\flat)$.
  Assume that $L^\flat$ is integral. Let $(a_1,\ldots,a_{n-1})$ be the fundamental invariants of $L^\flat$. 
  \begin{altenumerate}
  \item\label{item:decreaseval} If $x\not\in L^\flat\obot \mathbb{W}^{\ge a_{n-1}}$, then there exists $L'^\flat\subseteq \mathbb{V}$ an $O_F$-lattice of rank $n-1$ and $x'\in \mathbb{V}$ such that $L'^\flat+\langle x'\rangle=L^\flat+\langle x\rangle$ and $\val(L'^\flat)<\val(L^\flat)$.
  \item\label{item:changecoanisotrpic} If $x\in L^\flat\obot\mathbb{W}^{=a_{n-1}}$ and $L^\flat\subseteq \mathbb{V}$ is co-isotropic, then there exists a co-anisotropic $O_F$-lattice $L'^\flat\subseteq \mathbb{V}$  of rank $n-1$ and $x'\in \mathbb{V}$ such that $L'^\flat+\langle x'\rangle=L^\flat+\langle x\rangle$ and $\val(L'^\flat)=\val(L^\flat)$.
    \end{altenumerate}
\end{lemma}

\begin{proof}Let $\{e_1,\ldots,e_{n-1}\}$ be a standard orthogonal basis of $L^\flat$.
  \begin{altenumerate}
  \item Since $x\not\in L^\flat\obot \mathbb{W}^{\ge a_{n-1}}$, we have two cases:

    If $x^\flat\in L^\flat$, then $\val(x^\perp)<a_{n-1}$. Taking $L'^\flat=\langle e_1,\ldots, e_{n-2},x^\perp\rangle$ and $x'=e_{n-1}$ works.
    
    If $x^\flat\not\in L^\flat$, write $x^\flat=\lambda_1e_1+\ldots+\lambda_{n-1}e_{n-1}$, then $\val(\lambda_i)<0$ for some $1\le i\le n-1$. The fundamental matrix of $L^\flat+\langle x\rangle$ for the basis $\{e_1,\ldots,e_{n-1},x\}$ has the form $$
T=\begin{pmatrix}(e_1,e_1)&&& (e_1,x)\\ 
&\ddots&&\vdots\\
&&(e_{n-1},e_{n-1})& (e_{n-1},x)\\
    (x,e_1)&\cdots&(x,e_{n-1})& (x,x)
\end{pmatrix}.
$$ Let $(a_1',\ldots, a_n')$ be the fundamental invariants of $L^\flat +\langle x\rangle$. By the theory of Smith normal forms, we know that $a_1'+\cdots+ a_{n-1}'$ is equal to the minimum among the valuations of the determinants of all $(n-1)\times(n-1)$-minors of $T$. The set of such minors is bijective to the set of $(i,j)$-th entry: removing $i$-th row and $j$-th column to get such a minor. The valuation of the determinant of the $(n,i)$-th minor is $$\val((e_i,x)) -a_i+(a_1+\cdots+a_{n-1}).$$  Since $\val(\lambda_i)<0$, it follows that $\val((e_i,x))<a_i,$ and hence $a_1'+\cdots+ a_{n-1}'<a_1+\cdots+ a_{n-1}$. Let $\{e_1',\ldots,e_n'\}$ be a standard orthogonal basis of $L^\flat+\langle x\rangle$. Then taking $L'^\flat=\langle e_1',\ldots,e_{n-1}'\rangle$ and $x'=e_n'$ works.
\item Since $x^\flat\in L^\flat$, replacing $x$ by $x^\perp$ we may assume that $x\in \mathbb{W}^{=a_{n-1}}$. Let $\lambda=\frac{(x,x)}{(e_{n-1},e_{n-1})}$. Then $\val(\lambda)=0$. We have two cases:

  If $\lambda\in (O_F^\times)^2$, then we may find $\mu \in O_F^\times$ such that $1+ \mu^2\lambda\in O_F^\times\setminus (O_F^\times)^2$. Let $e_{n-1}'=e_{n-1}+\mu x$. Then $$\frac{(e_{n-1}',e_{n-1}')}{(e_{n-1},e_{n-1})}=\frac{(e_{n-1},e_{n-1})+\mu^2(x,x)}{(e_{n-1},e_{n-1})}=1+\mu^2\lambda\in O_F^\times\setminus (O_F^\times)^2.$$ Let $L'^\flat=\langle e_1,\ldots,e_{n-2},e_{n-1}'\rangle$. Then $$\chi(L'^\flat)=\chi\left(\frac{(e_{n-1}',e_{n-1}')}{(e_{n-1},e_{n-1})}\right)\chi(L^\flat)=\chi(1+\mu^2\lambda)=-\chi(L^\flat).$$ Then $L^\flat$ is co-isotropic implies that $L'^\flat$ is co-anisotropic (see Definition \ref{def:horizontallattice}). Taking $L'^\flat$ and $x'=x$ then works.

 If $\lambda\in O_F^\times\setminus (O_F^\times)^2$, then similarly taking $L'^\flat=\langle e_1,\ldots,e_{n-2},x\rangle$ and $x'=e_{n-1}$ works, as $$\chi(L'^\flat)=\chi\left(\frac{(x,x)}{(e_{n-1},e_{n-1})}\right)\chi(L^\flat)=\chi(\lambda)\chi(L^\flat)=-\chi(L^\flat).\qedhere$$
  \end{altenumerate}
\end{proof}

\begin{theorem}\label{thm:finalinduction}Let $L^\flat\subseteq\BV$ be an $O_F$-lattice of rank $n-1$. Then as functions on $\Omega(L^\flat)$,
\begin{equation}\label{eqn:int=den}
\Int_{L^\flat}=\pDen_{L^\flat}
\end{equation}
\end{theorem}
\begin{proof}
When $L^\flat$ is not integral both sides are zero, so we may assume that $L^\flat$ is integral. Let $(a_1,\ldots,a_{n-1})$ be the fundamental invariants of $L^\flat$. When $a_1=0$ by the cancellation laws (\ref{eq:cancelInt}) and (\ref{eq:cancel den}) we may reduce to the case for smaller $n$, hence by induction on $n$ we may assume that $a_1\ge1$ and thus $a_{n-1}\ge1$.  Now we induct on $\val(L^\flat)$.  There are two cases:
  \begin{altenumerate}
  \item   Assume that $L^\flat$ is co-anisotropic. By Lemma \ref{lem:decreaseval} (\ref{item:decreaseval}), when $x\not\in L^\flat\obot \mathbb{W}^{\ge a_{n-1}}$, there exists $L'^\flat\subseteq \mathbb{V}$ of rank $n-1$ and $x'\in \mathbb{V}$ such that $\val(L'^\flat)<\val(L^\flat)$ and $$\Int_{L^\flat}(x)=\Int_{L'^\flat}(x'),\quad \pDen_{L^\flat}(x)=\pDen_{L'^\flat}(x').$$ By the induction hypothesis we know that when $x\not\in L^\flat\obot \mathbb{W}^{\ge a_{n-1}}$ the equality $$\Int_{L^\flat}(x)=\pDen_{L^\flat}(x)$$ holds. Hence the difference function has support $$\supp(\Int_{L^\flat}-\pDen_{L^\flat})\subseteq L^\flat\obot \mathbb{W}^{\ge a_{n-1}}.$$ Since both $\Int_{L^\flat}$ and $\pDen_{L^\flat}$ are invariant under translation by $L^\flat$, by the support condition we know that $$\Int_{L^\flat}-\pDen_{L^\flat}=\mathbf{1}_{L^\flat}\otimes \phi^\perp$$ for some function $\phi^\perp$ on $\mathbb{W}^\mathrm{an}$ with $\supp(\phi^\perp)\subseteq \mathbb{W}^{\ge a_{n-1}}$. It remains to show that $\phi^\perp=0$.  Using Theorem \ref{thm:Int H} and performing the partial Fourier transform, we have $$\supp(\Intp-\pDenp)=\supp(\phi^\perp)\subseteq \mathbb{W}^{\ge a_{n-1}}.$$ Since $a_{n-1}\ge1$, we know that $\Intp-\pDenp$ vanishes on $\mathbb{W}^{=-1}\cup \mathbb{W}^{=0}$. Hence by Proposition \ref{prop:Intp} (\ref{item:Intp1}) and Proposition \ref{prop:pDenp} (\ref{item:pDenp1}) we know that  $\Intp-\pDenp$ vanishes identically, and hence $\phi^\perp=0$ as desired.
  \item Assume that $L^\flat$ is co-isotropic.  By Lemma \ref{lem:decreaseval} (\ref{item:decreaseval}) and the induction hypothesis, we similarly know that $$\supp(\Int_{L^\flat}-\pDen_{L^\flat})\subseteq L^\flat\obot \mathbb{W}^{\ge a_{n-1}}.$$ Moreover by Lemma \ref{lem:decreaseval} (\ref{item:changecoanisotrpic}) and the already proved co-anisotropic case, we know that $$\supp(\Int_{L^\flat}-\pDen_{L^\flat})\subseteq L^\flat\obot \mathbb{W}^{\ge a_{n-1}+1}.$$ 
 Hence $\Int_{L^\flat}-\pDen_{L^\flat}=\mathbf{1}_{L^\flat}\otimes \phi^\perp$ for some function $\phi^\perp$ on $\mathbb{W}^\mathrm{an}$ with $\supp(\phi^\perp)\subseteq \mathbb{W}^{\ge a_{n-1}+1}$. Since $a_{n-1}+1\ge2$, we know that $\Intp-\pDenp$ vanishes on $\mathbb{W}^{=-1}\cup \mathbb{W}^{=0}\cup\mathbb{W}^{=1}$. Hence by Proposition \ref{prop:Intp} (\ref{item:Intp2}) and Proposition \ref{prop:pDenp} (\ref{item:pDenp2}) we know that  $\Intp-\pDenp$ vanishes identically, and hence $\phi^\perp=0$ as desired.\qedhere
  \end{altenumerate}
\end{proof}

Our main Theorem \ref{thm: main} now follows immediately from Theorem \ref{thm:finalinduction} by decomposing $L=L^\flat+\langle x\rangle$ for some non-degenerate $O_F$-lattice $L^\flat\subseteq \mathbb{V}$ of rank $n-1$ and $x\in \Omega(L^\flat)$ (for example $L^\flat=\langle e_1,\ldots,e_{n-1}\rangle$ and $x=e_n$ for a standard orthogonal basis $\{e_1,\ldots, e_n\}$ of $L$).

\part{Semi-global arithmetic Siegel--Weil formula}\label{part:semi-global-global}

In this part we apply our main Theorem \ref{thm: main} to prove an identity between the intersection number of special cycles on the integral canonical model of GSpin Shimura varieties  at a good place and the derivative of Fourier coefficients of Siegel--Eisenstein series (also known as the \emph{semi-global arithmetic Siegel--Weil formula}). This is achieved by relating the special cycles on GSpin Shimura varieties to those on GSpin Rapoport--Zink spaces via the $p$-adic uniformization, and by relating the Fourier coefficients to local representation densities. This deduction is more or less standard (see \cite{Kudla1999a},\cite{Terstiege2011} and \cite{Bruinier2018}), and we will formulate and prove the results for  GSpin Shimura varieties with general tame level structures.

\section{Notations on quadratic spaces}\label{sec:notat-quadr-spac}

In this part we switch to global notations. Denote by $F$ a totally real number field and $\mathbb{A}=\mathbb{A}_F$ its ring of ad\`eles (starting from \S\ref{sec:quadr-latt-spec}, we restrict to the case $F=\mathbb{Q}$). Denote by $\mathbb{V}$  a quadratic space over $\mathbb{A}$ of rank $m=n+1\ge3$ with symmetric bilinear form $(\ ,\ )$. For a place $v$ of $F$, write $\mathbb{V}_v=\mathbb{V}\otimes_{\mathbb{A}}F_v$, a quadratic space over the local field $F_v$. We say $\mathbb{V}$ is \emph{totally positive definite} if $\mathbb{V}_v$ has signature $(n,0)$ for all places $v|\infty$ of $F$. We define
\begin{align*}
  \det(\mathbb{V})&:=(\det(\mathbb{V}_v))_v\in \mathbb{A}^\times/(\mathbb{A}^\times)^2,\\
  \disc(\mathbb{V})&:=(\disc(\mathbb{V}_v))_v\in \mathbb{A}^\times/(\mathbb{A}^\times)^2,\\
  \Has(\mathbb{V})&:=\prod_v\varepsilon(\mathbb{V}_v)\in\{\pm1\}. 
\end{align*}
Define a quadratic character $\chi_\mathbb{V}: \mathbb{A}^\times\rightarrow \{\pm1\}$ such that for any $(a_v)_v\in \mathbb{A}^\times$, $$\chi_\mathbb{V}((a_v)_v)=\prod_v(a_v, \disc(\mathbb{V}_v))_{F_v}.$$ If $\disc(\mathbb{V})\in F^\times\!\!\mod (\mathbb{A}^\times)^2$, then by the product formula $\chi_\mathbb{V}$ factors through $\mathbb{A}^\times/F^\times\rightarrow\{\pm1\}$. In this case, we say $\mathbb{V}$ is \emph{coherent} if $\Has(\mathbb{V})=+1$ and \emph{incoherent} if $\Has(\mathbb{V})=-1$. Notice that if $\mathbb{V}$ is coherent then there exists a global quadratic space $V$ over $F$ of rank $m$ such that $\mathbb{V}\simeq V \otimes_F \mathbb{A}$. If $\mathbb{V}$ is incoherent, then such global quadratic space $V$ does not exist, but for any place $v_0$ of $F$, there exists a \emph{nearby} global quadratic space $V$ (associated to $v_0$) such that $\chi_\mathbb{V}=\chi_V$, and $\Has(\mathbb{V}_v)=\Has(V_v)$ for all $v\ne v_0$. 

\section{Incoherent Eisenstein series}\label{sec:incoh-eisenst-seri-3}

\subsection{Siegel Eisenstein series}\label{sec:sieg-eisenst-seri} Let $W$ be the standard symplectic space over $F$ of dimension $2n$. Let $P=MN\subseteq \Sp(W)$ be the standard Siegel parabolic subgroup, so that under the standard basis of $W$,
\begin{align*}
M(F)&=\left\{m(a)=\begin{pmatrix}a & 0\\0 &{}^ta^{-1}\end{pmatrix}: a\in \GL_n(F)\right\},\\
  N(F) &= \left\{n(b)=\begin{pmatrix} 1_n & b \\0 & 1_n\end{pmatrix}: b\in \Sym_n(F)\right\}.  
\end{align*}
Let $\Mp(W_{\mathbb{A}})$ be the metaplectic extension of $\Sp(W)(\mathbb{A})$, $$1\rightarrow \mathbb{C}^1\rightarrow \Mp(W_\mathbb{A})\rightarrow \Sp(W)(\mathbb{A})\rightarrow 1,$$ where $\mathbb{C}^1=\{z\in \mathbb{C}^\times: |z|=1\}$. There is an isomorphism $\Mp(W_{\mathbb{A}})\isoarrow \Sp(W)(\mathbb{A})\times \mathbb{C}^1$ with the multiplication on the latter is given by the global Rao cycle, which allows us to write an element of $\Mp(W_{\mathbb{A}})$ as $(g, t)$ where $g\in \Sp(W)(\mathbb{A})$ and $t\in \mathbb{C}^1$. Recall that the metaplectic extension has a canonical splitting over $\Sp(W)(F)$ and $N(\mathbb{A})$, hence we may view $\Sp(W)(F)$ and $N(\mathbb{A})$ as subgroups of $\Mp(W_\mathbb{A})$ (\cite[p. 549]{Kudla1997a}).

Write $$G_n(\mathbb{A})=
\begin{cases}
  \Sp(W)(\mathbb{A}), & n\text{ odd}, \\
  \Mp(W_{\mathbb{A}}),& n\text{ even},
\end{cases}$$ for short. Let $P_n(\mathbb{A})=M_n(\mathbb{A})N_n(\mathbb{A})$ be the standard Siegel parabolic subgroup of $G_n(\mathbb{A})$, i.e., $P_n(\mathbb{A})$ (resp. $M_n(\mathbb{A})$) is $P(\mathbb{A})$ (resp. $M(\mathbb{A})$) if $n$ is odd and is the pullback of $P(\mathbb{A})$ (resp. $M(\mathbb{A})$) along the metaplectic extension if $n$ is even; and $N_n(\mathbb{A})=N(\mathbb{A})$. In particular,
\begin{align*}
M_n(\mathbb{A})&=
\begin{cases}
  \{m(a): a\in \GL_n(\mathbb{A})\}, & n \text{ odd}, \\
  \{(m(a), t): a\in \GL_n(\mathbb{A}), t\in \mathbb{C}^1\}, & n\text{ even},
\end{cases}\\
N_n(\mathbb{A})&=\{n(b): b\in \Sym_n(\mathbb{A})\}.
\end{align*} When $n$ is even, by abusing notation we write $m(a)\in M_n(\mathbb{A})$ for the element $(m(a),1)\in M_n(\mathbb{A})$.

Fix a quadratic character $\chi: \mathbb{A}^\times/F^\times\rightarrow \mathbb{C}^\times$. Fix the additive character $\psi=\psi_\mathbb{Q}\circ\tr_{F/\mathbb{Q}}: \mathbb{A}/F\rightarrow \mathbb{C}^\times$, where $\psi_\mathbb{Q}: \mathbb{A}_\mathbb{Q}/\mathbb{Q}\rightarrow \mathbb{C}^\times$ is the standard addicitve character such that $\psi_{\mathbb{Q},\infty}(x)=e^{2\pi i x}$. We may view $\chi$ as a character on $M_n(\mathbb{A})$ by $$
\begin{cases}
  \chi(m(a))= \chi(\det a), & n \text{ odd}, \\
 \chi(m(a), t)=\chi(\det a)\cdot \gamma(\det a,\psi)^{-1} \cdot t, & n\text{ even}.
\end{cases}
$$ and extend it to $P_n(\mathbb{A})$ trivially on $N_n(\mathbb{A})$. Here $\gamma(x,\psi)$ is the Weil index, an 8-th root of unity (\cite[p.548]{Kudla1997a}. Define the \emph{degenerate principal series} to be the unnormalized smooth induction $$I_n(s,\chi)\coloneqq \Ind_{P_n(\mathbb{A})}^{G_n(\mathbb{A})}(\chi\cdot |\cdot|_{F}^{s+(n+1)/2}),\quad s\in \mathbb{C}.$$ For a standard section $\Phi(-, s)\in I_n(s,\chi)$ (i.e., its restriction to the standard maximal compact subgroup of $G_n(\mathbb{A})$ is independent of $s$), define the associated \emph{Siegel Eisenstein series} $$E(g,s, \Phi)\coloneqq \sum_{\gamma\in P(F)\backslash \Sp(W)(F)}\Phi(\gamma g, s),\quad g\in G_n(\mathbb{A}),$$ which converges for $\Re(s)\gg 0$ and admits meromorphic continuation to $s\in \mathbb{C}$.  Notice that $E(g,s,\Phi)$ depends on the choice of $\chi$.

\subsection{Fourier coefficients and derivatives}\label{sec:four-coeff-deriv}

We have a Fourier expansion $$E(g,s,\Phi)=\sum_{T\in\Sym_n(F)}E_T(g,s,\Phi),$$ where $$E_T(g,s,\Phi)=\int_{N_n(F)\backslash N_n(\mathbb{A})} E(n(b)g,s,\Phi)\psi(-\tr(Tb))\,\rd n(b),$$ and the Haar measure $\rd n(b)$ is normalized to be self-dual with respect to $\psi$. When $T$ is nonsingular, for factorizable $\Phi=\otimes_v\Phi_v$ we have a factorization of the Fourier coefficient into a product $$E_T(g,s,\Phi)=\prod_v W_{T,v}(g_v, s, \Phi_v),$$ where the \emph{local (generalized) Whittaker function} is defined by $$W_{T,v}(g_v, s, \Phi_v)=\int_{N_n(F_v)}\Phi_v(w_n^{-1}n(b)g,s)\textstyle\psi(-\tr(\frac{1}{2}Tb))\, \rd n(b),\quad w_n=
\begin{pmatrix}
0  & 1_n\\
  -1_n & 0\\
\end{pmatrix}.$$ and has analytic continuation to $s\in \mathbb{C}$. Thus we have a decomposition of the derivative of a nonsingular Fourier coefficient at $s=s_0$,
\begin{equation}
  \label{eq:eissum}
E_T'(g, s_0, \Phi)=\sum_v E'_{T,v}(g, s_0,\Phi),
\end{equation}
 where  
\begin{equation}
  \label{eq:eisenfactor}
  E'_{T,v}(g, s, \Phi)=W_{T,v}'(g_v, s,\Phi_v)\cdot \prod_{v'\ne v}W_{T,v'}(g_{v'},s,\Phi_{v'}).
\end{equation}

\subsection{Incoherent Eisenstein series} \label{sec:incoh-eisenst-seri} Let $\mathbb{V}$ be a quadratic space over $\mathbb{A}$ of rank $m=n+1$ with $\chi_\mathbb{V}=\chi$.  Let $\sS(\mathbb{V}^n)$ be the space of Schwartz functions on $\mathbb{V}^n$. The fixed choice of $\chi$ and $\psi$ gives a \emph{Weil representation} $\omega=\omega_{\chi,\psi}$ of $G_n(\mathbb{A})\times \O(\mathbb{V})$ on $\sS(\mathbb{V}^n)$. Explicitly, for $\varphi\in \sS(\mathbb{V}^n)$ and $\mathbf{x}\in \mathbb{V}^n$,
\begin{align*}
  \omega(m(a))\varphi(\mathbf{x})&=\chi(m(a))|\det a|_F^{m/2}\varphi(\mathbf{x}\cdot a)&m(a)\in M_n(\mathbb{A}),\\
\omega(n(b))\varphi(\mathbf{x})&=\textstyle\psi(\frac{1}{2}\tr b\,T(\mathbf{x}))\varphi(\mathbf{x}),&n(b)\in N_n(\mathbb{A}),\\
\omega(w_n)\varphi(\mathbf{x})&=\gamma_{\mathbb{V}}^n\cdot\widehat \varphi(\mathbf{x}),&w_n=\left(\begin{smallmatrix}
0  & 1_n\\
  -1_n & 0\\
\end{smallmatrix}\right),\\
\omega(h)\varphi(\mathbf{x})&=\varphi(h^{-1}\cdot\mathbf{x}),& h\in \O(\mathbb{V}),\\
\omega(1,t)\varphi(\mathbf{x})&=t\cdot \varphi(\mathbf{x}), &t\in \mathbb{C}^1\text{, if }n \text{ even}.
\end{align*} 
Here $T(\mathbf{x})=((x_i,x_j))_{1\le i,j\le n}$ is the fundamental matrix of $\mathbf{x}$, the constant $$\gamma_{\mathbb{V}}=\gamma(\det \mathbb{V},\psi)^{-1}\gamma(\psi)^{-m}\Has(\mathbb{V})$$ is the Weil constant (cf. \cite[p.642]{Kudla1997a}), and $\widehat\varphi$ is the Fourier transform of $\varphi$ using the self-dual Haar measure on $\mathbb{V}^n$ with respect to $\psi$.

For $\varphi\in \sS(\mathbb{V}^n)$, define a function $$\Phi_\varphi(g)\coloneqq \omega(g)\varphi(0),\quad g\in G_n(\mathbb{A}).$$ Then $\Phi_\varphi\in I_n(0,\chi)$. Let $\Phi_\varphi(-,s)\in I_n(s,\chi)$ be the associated standard section, known as \emph{the standard Siegel--Weil section} associated to $\varphi$. For $\varphi\in\sS(\mathbb{V}^n)$, we write $$E(g,s,\varphi)\coloneqq E(g, s, \Phi_\varphi), \quad E_T(g,s,\varphi)\coloneqq E_T(g,s,\Phi_\varphi),\quad E'_{T,v}(g,s,\varphi)\coloneqq E'_{T,v}(g,s,\Phi_\varphi),$$ and similarly for $W_{T,v}(g_v,s,\varphi_v)$. When $\mathbb{V}$ is incoherent, the central value $E(g,0,\varphi)$ automatically vanishes. In this case, we write the central derivatives as $$\pEis(g,\varphi)\coloneqq E'(g,0,\varphi),\quad \pEis_{T}(g,\varphi)\coloneqq E'_{T}(g,0,\varphi), \quad \pEis_{T,v}(g,\varphi)\coloneqq E'_{T,v}(g,0,\varphi).$$

\begin{remark}\label{rem:DiffTV}
Let $T\in \Sym_n(F)$ be nonsingular. There exists a unique quadratic space $\mathbb{V}$ over $\mathbb{A}$ of rank $m=n+1$ with $\chi_\mathbb{V}=\chi$ representing $T$, which satisfies that for any place $v$
\begin{equation}
  \label{eq:diff}
  \Has(\mathbb{V}_v)=\Has(T_v)(\det T_v, \det \mathbb{V}_v\cdot \det T_v)_{F_v}
\end{equation}
 in view of Lemma \ref{lem:directsum} (\ref{item:l3}) (cf. \cite[Proposition 1.3]{Kudla1997a}).
 More generally, we define $\Diff(T,\mathbb{V})$ to be the set of places $v$ for which (\ref{eq:diff}) does not hold.  By \cite[Proposition 1.4]{Kudla1997a}, we know that $W_{T,v}(g_v,0,\varphi_v)\ne0$ only if $v\not\in \Diff(T, \mathbb{V})$. Hence by (\ref{eq:eisenfactor}) we know that $\pEis_{T,v}(g,\varphi)\ne0$ only if $\Diff(T,\mathbb{V})=\{v\}$.  
\end{remark}

\subsection{Classical incoherent Eisenstein series}\label{sec:incoh-eisenst-seri-1}

Assume that $\mathbb{V}$ is totally positive definite and incoherent. The hermitian symmetric domain for $\Sp(W)$ is the \emph{Siegel upper half space}
\begin{align*}
  \mathbb{H}_n  &=\{\sz=\sx+i\sy:\ \sx\in\Sym_n(F_\infty),\ \sy\in\Sym_n(F_\infty)_{>0}\},  
\end{align*}
where $F_\infty=F \otimes_{\mathbb{Q}} \mathbb{R}$. Define the \emph{classical incoherent Eisenstein series} to be $$E(\sz,s,\varphi)\coloneqq \chi_\infty(m(a))^{-1}|\det(m(a))|_F^{-m/2}\cdot E(g_\sz,s, \varphi), \quad g_\sz\coloneqq n(\sx)m(a)\in G_n(\mathbb{A}),$$ where $a\in\GL_n(F_\infty)$ such that $\sy=a{}^{t} a$. Notice that $E(\sz,s,\Phi)$ does not depend on the choice of $\chi$. We write the central derivatives as $$\pEis(\sz,\varphi)\coloneqq E'(\sz,0,\varphi),\quad \pEis_T(\sz, \varphi)\coloneqq E'_T(\sz,0,\varphi),\quad \pEis_{T,v}(\sz, \varphi)\coloneqq E'_{T,v}(\sz,0,\varphi).$$ Then we have a Fourier expansion
\begin{equation}
  \label{eq:eisz}
  \pEis(\sz,\varphi)=\sum_{T\in\Sym_n(F)}\pEis_T(\sz,\varphi)
\end{equation}
For an open compact subgroup $\KG\subseteq \GSpin(\mathbb{V})(\mathbb{A}_f)$, we will choose $$\varphi= \varphi_{\KG} \otimes \varphi_{\infty}\in \sS(\mathbb{V}^n)$$ such that $\varphi_K\in\sS(\mathbb{V}^n_f)$ is $\KG$-invariant and $\varphi_{\infty}$ is the Gaussian function $$\varphi_{\infty}(\mathbf{x})=e^{-2\pi\tr T(\mathbf{x})}\coloneqq \prod_{v|\infty}e^{-2\pi \tr T(\mathbf{x}_v)}.$$ 
For our fixed choice of Gaussian $\varphi_\infty$, we write $$E(\sz,s,\varphi_K)=E(\sz,s,\varphi_K \otimes \varphi_\infty),\quad \pEis(\sz, \varphi_K)=\pEis(\sz, \varphi_K \otimes \varphi_\infty)$$ and so on for short.

\section{Special cycles on GSpin Shimura varieties}\label{sec:kudla-rapop-cycl-2}

\subsection{GSpin Shimura varieties}\label{sec:gspin-shim-vari}

Let $F$ be a totally real number field. Let $\mathbb{V}$ be a totally positive definite incoherent quadratic space over $\mathbb{A}$ of rank $m=n+1\ge3$. A choice of an infinite place $v_0|\infty$ of $F$ gives rise to a global nearby quadratic space $V$ (associated to $v_0$) such that its signatures at infinite places $v|\infty$ are given by $$\sgn(V_v)=
\begin{cases}
  (m-2,2), & v=v_0,\\
  (m,0), & v\ne v_0.
\end{cases}$$ Define $G=\Res_{F/\mathbb{Q}}\GSpin(V)$.

For an oriented negative 2-plane $Z=\langle e_1,e_2\rangle\subseteq V_{v_0}$ such that the $\mathbb{R}$-basis $\{e_1,e_2\}$ has fundamental matrix $\left(\begin{smallmatrix} -1 & 0 \\ 0 & -1\end{smallmatrix}\right)$, define an $\mathbb{R}$-algebra homomorphism into the even part of the Clifford algebra of $V_{v_0}$. $$\mathbb{C}\rightarrow C^+(V_{v_0}), \quad a+bi\mapsto a+b e_1e_2.$$  Its restriction to $\mathbb{C}^\times$ lands in $\GSpin(V_{v_0})(\mathbb{R})\subseteq C^+(V_{v_0})^\times$ and gives a homomorphism $$h_Z: \mathbb{C}^\times\rightarrow \GSpin(V_{v_0})(\mathbb{R})\simeq\GSpin(m-2,2)(\mathbb{R}).$$  Define a cocharacter $$h_{G} : \mathbb{C}^\times \rightarrow G(\mathbb{R})=\prod_{v|\infty} \GSpin(V_v)(\mathbb{R})\simeq\GSpin(m-2,2)(\mathbb{R})\times\prod_{v\ne v_0}\GSpin(m,0)(\mathbb{R}),$$ whose $v$-component is given by $$(h_{G}(z))_v=
\begin{cases}
  h_Z(z), & v=v_0,\\
  1, & v\ne v_0.
\end{cases}$$ Let $\mathcal{D}$ be its $G(\mathbb{R})$-conjugacy class, which is isomorphic to the  hermitian symmetric domain of oriented negative 2-planes in $V_{v_0}$ and has two connected components. The action of $G(\mathbb{R})$ on $\mathcal{D}$ factors through the natural quotient map $G(\mathbb{R})\rightarrow \SO(V_{v_0})(\mathbb{R})$. 

We obtain a Shimura datum $(G, \mathcal{D})$, which is of Hodge type when $F=\mathbb{Q}$ and of abelian type in general. Let $K\subseteq G(\mathbb{A}_f)$  be an open compact subgroup. Denote by $\Sh_K=\Sh_K(G, \mathcal{D})$ the canonical model of the Shimura variety with complex uniformization $$\Sh_K(\mathbb{C})=G(\mathbb{Q})\backslash [\mathcal{D}\times G(\mathbb{A}_f)/ K].$$ The canonical model $\Sh_K$ is a smooth Deligne--Mumford stack (and a smooth quasi-projective variety when $K$ is neat) over the reflex field $F\subseteq \mathbb{C}$ (via the embedding induced by the chosen infinite place $v_0$), which is proper when $V$ is anisotropic (e.g., when $F\ne \mathbb{Q}$). 

\subsection{Special cycles $Z(T,\varphi_K)$}\label{sec:special-cycles-zt} Let $1\le r\le m-2$. Let $\mathbf{x}=[x_1,\ldots,x_r]\subseteq V^r$ be an $r$-tuple of vectors in $V$. Assume that the fundamental matrix $T(\mathbf{x})=((x_i, x_j))_{i,j}\in \Sym_r(F)_{>0}$ (totally positive definite). So $\langle x_1,\ldots,x_r\rangle$ is a totally positive definite subspace of $V$, and we denote by $V_\mathbf{x}$ its orthogonal complement in $V$. Then $V_\mathbf{x}$ has signatures $$\sgn(V_{\mathbf{x},v})=
\begin{cases}
  (m-r-2,2), & v=v_0,\\
  (m-r,0), & v\ne v_0.
\end{cases}$$ Let $G_\mathbf{x}=\Res_{F/\mathbb{Q}}\GSpin(V_\mathbf{x})$ and $\mathcal{D}_\mathbf{x}$ be the hermitian symmetric domain of oriented negative 2-planes in $V_{\mathbf{x},v_0}$. Then we have a natural embedding of Shimura data $$(G_\mathbf{x}, \mathcal{D}_\mathbf{x})\hookrightarrow (G, \mathcal{D}).$$

Let $g\in G(\mathbb{A}_f)$. Define the \emph{special cycle} $Z(\mathbf{x}, g)_K$ on $\Sh_K(G, \mathcal{D})$ to be the image of the composition $$\Sh_{gKg^{-1}\cap G_\mathbf{x}(\mathbb{A}_f)}(G_\mathbf{x}, \mathcal{D}_\mathbf{x})\hookrightarrow \Sh_{gKg^{-1}}(G,\mathcal{D})\xrightarrow{\cdot g} \Sh_{K}(G, \mathcal{D}).$$ It admits complex uniformization $$Z(\mathbf{x},g)_K(\mathbb{C})\simeq G_\mathbf{x}(\mathbb{Q})\backslash [\mathcal{D}_\mathbf{x}\times G_\mathbf{x}(\mathbb{A}_f) g K/K].$$  Let $\varphi_K\in \sS(\mathbb{V}_f^r)$ be a $K$-invariant Schwartz function. Let $T\in \Sym_r(F)_{>0}$. Define the \emph{(weighted) special cycle} $Z(T, \varphi_K)$ on $\Sh_K(G, \mathcal{D})$ to be $$Z(T,\varphi_K):=\sum_{\mathbf{x}\in G(\mathbb{Q})\backslash V^r \atop T(\mathbf{x})=T}\sum_{g\in G_\mathbf{x}(\mathbb{A}_f)\backslash G(\mathbb{A}_f)/K}\varphi_K(g^{-1}\mathbf{x})\cdot Z(\mathbf{x},g)_K\in \RZ^r(\Sh_K(G, \mathcal{D}))_\mathbb{C}.$$

\subsection{Semi-global integral models $\mathcal{M}=\mathcal{M}_K$ at hyperspecial levels} Let $p$ be a prime. Let $\nu$ be a place of $F$ above $p$. We take $K=\prod_{v|p} K_{v}\times K^p\subseteq G(\mathbb{Q}_p)\times G(\mathbb{A}_f^p)$. Assume that
\begin{enumerate}[label=(H\arabic*)]
\item for each $v|p$, the group $K_{v}$ is a hyperspecial subgroup of $\GSpin(V_v)(F_v)$ (e.g., the stabilizer of a self-dual lattice $\Lambda_v\subseteq V_v$).
\item $K^p\subseteq G(\mathbb{A}_f^p)$ is an open compact subgroup.
\end{enumerate}
Then by Kisin \cite{Kisin2010} ($p>2$) and Kim--Madapusi-Pera \cite{Kim2016} ($p=2$), there exists a smooth integral canonical model $\mathcal{M}_K$ of $\Sh_K$ over the localization $O_{F,(\nu)}$. We remark that by the work of Lovering \cite{Lovering2017}, these semi-global integral models over $O_{F,(\nu)}$ glue to an integral canonical model over the ring of $S$-integers of $F$, for any $S$ containing all finite places $v$ where $K_v$ is not hyperspecial. 

We often fix the level $K$ as above and write $\mathcal{M}=\mathcal{M}_K$ for short.  

\subsection{The quadratic lattice of special endomorphisms $V(A_S)$}\label{sec:quadr-latt-spec}

From now on we assume that $F=\mathbb{Q}$ (thus $\nu=p$, $F_\nu=\mathbb{Q}_p$, $O_{F,(\nu)}= \mathbb{Z}_{(p)}$) and $p>2$. Assume that $K_p$ is the stabilizer of of a self-dual lattice $\Lambda_p\subseteq V_p$. The Shimura datum $(G, \mathcal{D})$ is of Hodge type, and there exists an embedding of Shimura datum $$(G, \mathcal{D})\rightarrow (\wit G=\GSp(C(V)), \mathcal{H})$$ into a symplectic Shimura datum (\cite[\S3.5]{MadapusiPera2016}). Let $\wit K_p\subseteq \wit G(\mathbb{Q}_p)$ be the stabilizer the lattice $C(\Lambda_p)\subseteq C(V_p)$. Let $\wit K^p\subseteq \wit G(\mathbb{A}_f)$ be an open compact subgroup such that $K^p=\wit K^p\cap G(\mathbb{A}_f)$. Then we obtain a finite unramified morphism of Shimura varieties known as the \emph{Kuga--Satake morphism} $$\Sh_K(G, \mathcal{D})\rightarrow \Sh_{\wit K}(\wit G, \mathcal{H}).$$ The Siegel modular variety $\Sh_{\wit K}(\wit G, \mathcal{H})$ has a smooth integral model $\wit{\mathcal{M}}$ over $O_{F,(\nu)}$ via moduli interpretation of principally polarized abelian varieties with $\wit K^p$-level structures (\cite[\S3.9]{MadapusiPera2016}). Kisin's integral canonical model $\mathcal{M}$ of $\Sh_K(G, \mathcal{D})$ over $O_{F,(\nu)}$ is the normalization of the Zariski closure of $\Sh_K$  in $\wit{\mathcal{M}}$ \cite[Theorem 4.4]{MadapusiPera2016} (moreover the normalization is redundant in view of the recent work of Xu \cite{Xu2020}).

Denote by $(A, \lambda, \bar\eta^p)$ the pullback of the universal object on $\wit{\mathcal{M}}$ to $\mathcal{M}$. Here
\begin{altenumerate}
\item $(A, \lambda)$ is a principally polarized abelian variety (the \emph{Kuga--Satake abelian variety}),
\item  $\bar\eta^p$ is a $\wit K^p$-level structure, i.e., an $\wit K^p$-orbit of $\mathbb{A}_f^p$-linear isomorphisms of lisse $\mathbb{A}_f^p$--sheaves $$\eta^p: H^1_\et(A, \mathbb{A}_f^p)\isoarrow C(V) \otimes_F \mathbb{A}_f^p.$$
\end{altenumerate}
For an $\mathcal{M}$-scheme $S$, we denote by $$V(A_S)\subseteq \End(A_S) \otimes\OFp$$ the space of \emph{special endomorphisms} as defined in \cite[Definition 5.11]{MadapusiPera2016}.

By \cite[Lemma 5.12]{MadapusiPera2016}, $V(A_S)$ is equipped with a positive definite $\OFp$-quadratic form such that $f\circ f=(f,f)\cdot\id$. Let $r\ge1$. Given an $r$-tuple $\mathbf{x}=[x_1,\ldots,x_r]\in V(A_S)^r$, define its \emph{fundamental matrix} to be $T(\mathbf{x})\coloneqq ((x_i,x_j))_{i,j}\in \Sym_r(\OFp)$. 

\subsection{Semi-global special cycles $\mathcal{Z}(T,\varphi_K)$}\label{sec:semi-global-kudla}

Let $r\ge1$. We say a Schwartz function $$\varphi_K= \bigotimes_{v\nmid\infty}\varphi_{K,v}\in \sS(\mathbb{V}^r_f)$$ is \emph{$p$-admissible} if it is $K$-invariant and $\varphi_{K,p}=\mathbf{1}_{(\Lambda_{p})^r}$. Given such a $p$-admissible Schwartz function $\varphi_K$ and a nonsingular $T\in \Sym_r(F)$, define the semi-global  \emph{special cycle} $\mathcal{Z}(T,\varphi_K)$ over $\mathcal{M}$ as follows.

First consider a special $p$-admissible Schwartz function of the form
\begin{equation}
  \label{eq:testfunction}
  \varphi_K=(\varphi_i)\in \sS(\mathbb{V}^r_f),\quad \varphi_{i}=\mathbf{1}_{\Omega_i},\quad i=1,\ldots,r,
\end{equation}
where $\Omega_i\subseteq \mathbb{V}_f$ is a $K$-invariant open compact subset such that $\Omega_{i,p}=\Lambda_p$.  Notice that the prime-to-$p$ part $\Omega_i^{(p)}$ can be viewed as a subset $$\Omega_i^{(p)}\subseteq \mathbb{V}_f^{(p)}=V \otimes_F \mathbb{A}_f^p\subseteq \End(C(V) \otimes_F \mathbb{A}_f^p).$$  For an $\mathcal{M}$-scheme $S$, define $\mathcal{Z}(T,\varphi_K)(S)$ to be the set of $r$-tuples $\mathbf{x}=[x_1,\ldots,x_r]\in V(A_S)^r$  such that
\begin{altenumerate}
\item $T(\mathbf{x})=T$,
\item $\eta^p\circ x_{i}^*\circ (\eta^p)^{-1}\in \Omega_i^{(p)}$, for $i=1,\ldots,r$. Here $x_{i}^*\in\End(H^1_\et(A, \mathbb{A}_f^p))$ is induced by $x_i$.
\end{altenumerate}
The functor $S\mapsto \mathcal{Z}(T,\varphi_K)(S)$ is represented by a (possibly empty) Deligne--Mumford stack which is finite and unramified over $\mathcal{M}$ (cf. \cite[Proposition 1.3]{Kudla1999a}, \cite[Proposition 2.5]{Kudla2000a}, \cite[Proposition 6.13]{MadapusiPera2016}), and thus defines a cycle $\mathcal{Z}(T,\varphi_K)\in \RZ^*(\mathcal{M})$. By definition $\mathcal{Z}(T,\varphi_K)$ is nonempty only when $T\in \Sym_r(\OFp)_{>0}$. Notice that the generic fiber $\mathcal{Z}(T, \varphi_K)_{F}$ of the special cycle $\mathcal{Z}(T, \varphi_K)$ is nonempty only when $1\le r\le m-2$ (cf. \cite[Proposition 3.6]{Soylu2017}), in which case it agrees with the special cycle $Z(T,\varphi_K)$ defined in \S\ref{sec:special-cycles-zt}. When $r=m-1=n$ and $T\in \Sym_r(\OFp)_{>0}$, the generic fiber $\mathcal{Z}(T, \varphi_K)_{F}$ is empty, and $\mathcal{Z}(T, \varphi_K)$ is supported in the supersingular locus $\mathcal{M}^\mathrm{ss}_{\kappa_\nu}$ of the special fiber $\mathcal{M}_{\kappa_\nu}$ (cf. \cite[Proposition 3.7]{Soylu2017}). 

For a general $p$-admissible Schwartz function $\varphi_K\in \sS(\mathbb{V}^r_f)$ (which can be written as a $\mathbb{C}$-linear combination of special $p$-admissible functions, after possibly shrinking $K$), we obtain a cycle $\mathcal{Z}(T,\varphi_K)\in\RZ^*(\mathcal{M})_\mathbb{C}$ by  extending $\mathbb{C}$-linearly.

\subsection{$p$-adic uniformization of the supersingular locus of $\mathcal{M}$}\label{sec:p-adic-unif}

Let $\Mss$ be the completion of the base change $\mathcal{M}_{\OFnb}$ along the supersingular locus $\mathcal{M}_{\kappa_\nu}^\mathrm{ss}$. Then by \cite[Theorem 7.2.4]{Howard2015} (see also \cite{Kim2018a}), we have an isomorphism of formal stacks over $\Spf \OFnb$ (the \emph{$p$-adic uniformization})
\begin{equation}
  \label{eq:padicunif}
  \Theta: G'(\mathbb{Q})\backslash [\RRZ_{G_p}\times G(\mathbb{A}_f^p)/\KG^p]\isoarrow \Mss.
\end{equation}
Here $G'=\GSpin(V')$ for the nearby quadratic space $V'$ of $\mathbb{V}$ associated to the place $\nu$, and $\RRZ_{G_p}$ is the GSpin Rapoport--Zink space defined in \S\ref{sec:gspin-rapoport-zink}. Notice that $G'_p$ is the inner form of $G_p$  defined in \S\ref{sec:group-j} and thus $G'(\mathbb{Q})$ naturally acts on $\RRZ_{G_p}$ via the embedding $G'(\mathbb{Q})\hookrightarrow G'_p(F_p)$; and $G'(\mathbb{Q})$ naturally acts on $G(\mathbb{A}_f^p)$ via the embedding $G'(\mathbb{Q})\hookrightarrow G'(\mathbb{A}_f^p)\simeq G(\mathbb{A}_f^p)$.

\subsection{$p$-adic uniformization of $\Zss(T,\varphi_K)$}

Let $r\ge1$, $T\in \Sym_r(\OFp)_{>0}$ and $\varphi_K\in \sS(\mathbb{V}_f^r)$ be $p$-admissible. First consider that $\varphi_k$ is of the special  form (\ref{eq:testfunction}). Denote by $\Zss(T,\varphi_K)$ the completion of the base change $\mathcal{Z}(T,\varphi_K)_{\OFnb}$ along its supersingular locus $\mathcal{Z}^\mathrm{ss}(T,\varphi_K):=\mathcal{Z}(T,\varphi_K)\times_{\mathcal{M}_{\kappa_\nu}} \mathcal{M}_{\kappa_\nu}^\mathrm{ss}$, viewed as an element of $K_0'(\Mss)$. For a general $p$-admissible $\varphi_K$, we obtain an element $\Zss(T,\varphi_K)\in K_0'(\Mss)_\mathbb{C}$ by  extending $\mathbb{C}$-linearly.  For $\mathbf{x}\in V'^r$, we may view $\mathbf{x}$ as a subset of $V'_p$, which gives the local special cycles $\mathcal{Z}(\mathbf{x})$  on $\RRZ_{G_p}$ defined as in \S\ref{sec:kudla-rapop-cycl-1}.

\begin{proposition}
Assume that $F=\mathbb{Q}$ and $p>2$.  Assume that $\varphi_\KG\in \sS(\mathbb{V}^r_f)$ is $p$-admissible (\S\ref{sec:semi-global-kudla}). Then for any $T\in \Sym_r(\OFp)_{>0}$, the $p$-adic uniformization isomorphism $\Theta$ in (\ref{eq:padicunif})  induces the following identity in $K_0'(\Mss)_\mathbb{C}$,
 \begin{equation}  \label{eq:ZTpadicunif}
\Zss(T,\varphi_K)=\sum_{\mathbf{x}\in G'(\mathbb{Q})\backslash V'^r \atop T(\mathbf{x})=T}\sum_{g\in G'_\mathbf{x}(\mathbb{Q})\backslash G'(\mathbb{A}_f^p)/K^p}\varphi_K(g^{-1}\mathbf{x})\cdot\Theta(\mathcal{Z}(\mathbf{x}), g).   
 \end{equation}
\end{proposition}

\begin{proof}
  By the $\mathbb{C}$-linearality, it suffices to treat the case of special Schwartz functions $\varphi_K=\mathbf{1}_\Omega$ as in (\ref{eq:testfunction}), in which case the assertion amounts to an isomorphism of formal stacks induced by the morphism $\Theta$,
  \begin{equation}
    \label{eq:specialZTunif}
    \Zss(T,\varphi_K)\simeq G'(\mathbb{Q})\backslash\Big[\bigsqcup_{\mathbf{x}\in V'^r \atop T(\mathbf{x})=T}\bigsqcup_{g\in G'(\mathbb{A}_f^p)/K^p\atop g^{-1}\mathbf{x}\in \Omega}(\mathcal{Z}(\mathbf{x}),g)\Big].
  \end{equation}
Choose a point $y_0\in \mathcal{M}(\kb)$ giving rise to the local unramified Shimura--Hodge data of \S\ref{sec:gspin-rapoport-zink} by \cite[Proposition 7.2.3]{Howard2015}. Then $y_0$ corresponds to  a supersingular abelian variety $A_0$ over $\kb$ together with the level structure $\eta_0^p: H^1_\et(A_0, \mathbb{A}_f^p)\isoarrow C(V) \otimes_F \mathbb{A}_f^p$. Its $p$-divisible group $\mathbb{X}=A_0[p^\infty]$ gives rise to the framing object of the Rapoport--Zink space $\RRZ_{G_p}$. By \cite[Remark 7.2.5]{Howard2015}, the space of special quasi-endomorphisms $V(A_0) \otimes \mathbb{Q}$ is isomorphic to $V'$. 

The isomorphism $\Theta$ is explicitly described as follows (see the proof of \cite[Theorem 3.2.1]{Howard2015} and \cite[\S6.14]{RZ96}). We fix a lift $\tilde{\mathbb{X}}$ of $\mathbb{X}$ over $\OFnb$ and let $\tilde A_0$ be the corresponding lift of $A_0$. Let $R\in \mathrm{ANilp}_\OFnb$.  Let $(z, g)\in \RRZ_{G_p}(R)\times G(\mathbb{A}_f^p)/K^p$. Then $z$ corresponds to a $p$-divisible group $X$ over $R$ with a quasi-isogeny $\rho_X: \tilde{\mathbb{X}} \otimes_{\OFnb} R\rightarrow X$. Let $y=\Theta(z, g)\in \Mss(R)$. Then $y$ corresponds to an abelian scheme $A$ over $R$ such that
  \begin{altenumerate}
  \item\label{item:unif1} there is a quasi-isogeny $\rho_A: \tilde A_0 \otimes_{\OFnb} R\rightarrow A$  inducing $\rho_X$ on $p$-divisible groups, which identifies $V'\simeq V(A_0) \otimes \mathbb{Q}$ with $V(A)\otimes \mathbb{Q}$ by $\mathbf{x}\rightarrow \rho_A\circ \mathbf{x}\circ\rho_A^{-1}$.
  \item\label{item:unif2} the level structure on $A$ is given by $\eta^p=g^{-1}(\eta^p_0\circ \rho_A^*): H^1_\et(A, \mathbb{A}_f^p)\isoarrow C(V) \otimes_F \mathbb{A}_f^p$.
  \end{altenumerate}  Let $\mathbf{x}\in V'^r$ such that $T(\mathbf{x})=T$. It then follows from (\ref{item:unif1}) and (~\ref{item:unif2}) that
  \begin{altenumerate}
  \item $z\in \mathcal{Z}(\mathbf{x})(R)$ if and only if $\mathbf{x}\in V(A)$;
  \item $g^{-1}\mathbf{x}\in \Omega$ if and only if $\eta^p\circ \mathbf{x}^*\circ(\eta^p)^{-1}\in \Omega^{(p)}$.
  \end{altenumerate}
Then by the definition of $\Zss(T,\varphi_K)$, we know that $\Theta$ induces the isomorphism in (\ref{eq:specialZTunif}) after dividing the action of $G'(\mathbb{Q})$.

\end{proof}

\subsection{Arithmetic intersection number $\Int_{T,p}(\varphi_K)$} \label{sec:local-arithm-inters} Assume $T\in \Sym_n(F)_{>0}$. Let $t_1,\ldots,t_n$ be the diagonal entries of $T$. 
Let $\varphi_K\in \sS(\mathbb{V}^n_f)$ be a special Schwartz function as in (\ref{eq:testfunction}). Define
\begin{equation}\label{eq:localint}
\Int_{T,p}(\varphi_K)\coloneqq \chi(\mathcal{Z}(T,\varphi_K), \mathcal{O}_{\mathcal{Z}(t_1,\varphi_1)} \otimes^\mathbb{L}\cdots \otimes^\mathbb{L}\mathcal{O}_{\mathcal{Z}(t_n,\varphi_n)})\cdot\log p,  
\end{equation}
where $\mathcal{O}_{\mathcal{Z}(t_i,\varphi_i)}$ denotes the structure sheaf of the semi-global special divisor $\mathcal{Z}(t_i,\varphi_i)$, $\otimes^\mathbb{L}$ denotes the derived tensor product of coherent sheaves on $\mathcal{M}$, and $\chi$ denotes the Euler--Poincar\'e characteristic (an alternating sum of lengths of $\mathcal{O}_{F,(\nu)}$-modules). We extend the definition of $\Int_{T,p}(\varphi_K)$ to a general $p$-admissible $\varphi_K\in \sS(\mathbb{V}^n_f)$ by  extending $\mathbb{C}$-linearly.

\subsection{Arithmetic Siegel--Weil formula: the semi-global identity} Now we are ready to prove our main semi-global application to the arithmetic Siegel--Weil formula. Recall that $m=n+1\ge3$.

\begin{theorem}\label{thm:semi-global-identity}
 Assume that $F=\mathbb{Q}$ and $p>2$. Assume that $\varphi_\KG\in \sS(\mathbb{V}^n_f)$ is $p$-admissible (\S\ref{sec:semi-global-kudla}). Then for any $T\in \Sym_n(F)_{>0}$, $$\Int_{T,p}(\varphi_{\KG})q^T=c_K\cdot \pEis_{T,p}(\sz,\varphi_{\KG}),$$ where $c_K=\frac{(-1)^n}{\vol(K)}$ is a nonzero constant independent of $T$ and $\varphi_K$, and $\vol(K)$ is the volume of $K$ under a suitable Haar measure on $G(\mathbb{A}_f)$.
\end{theorem}

\begin{proof}
If $T\not\in \Sym_n(\OFp)$, then $\Int_{T,p}(\varphi_K)$ vanishes (as $\mathcal{Z}(T,\varphi_K)$ is empty by definition) and so does $\pEis_{T,p}(\sz, \varphi_K)$ by (\ref{eq:eisenfactor}) and (\ref{eq:localWhittaker1}). So we may assume that $T\in \Sym_n(\OFp)_{>0}$. Since $\mathcal{Z}(T,\varphi_K)$ is supported on the supersingular locus, by (\ref{eq:ZTpadicunif}) we know that $$\Int_{T,p}(\varphi_K)=\sum_{\mathbf{x}\in G'(\mathbb{Q})\backslash V'^n \atop T(\mathbf{x})=T}\sum_{g\in G'_\mathbf{x}(\mathbb{Q})\backslash G'(\mathbb{A}_f^p)/K^p}\varphi_K(g^{-1}\mathbf{x})\cdot\chi(\RRZ_{G_p}, \mathcal{Z}(x_1) \otimes^\mathbb{L}\cdots \otimes^\mathbb{L}\mathcal{Z}(x_n))\cdot \log p.$$ Let $\mathcal{N}=\RRZ_{G_p}^{(0)}$ be the connected GSpin Rapoport--Zink space in \S\ref{sec:conn-gspin-rapop}. Then we have an isomorphism (\cite[\S7.2.6]{Howard2015}) $$G'(\mathbb{Q})\backslash [\RRZ_{G_p}\times G(\mathbb{A}_f^p)/K^p]\simeq G'(\mathbb{Q})_0\backslash[\mathcal{N}\times G(\mathbb{A}_f^p)/K^p],$$ where $G'(\mathbb{Q})_0\subseteq G'(\mathbb{Q})$ consists of elements whose spinor norm is a $p$-unit. Hence $$\Int_{T,p}(\varphi_K)=\sum_{\mathbf{x}\in G'(\mathbb{Q})_0\backslash V'^n \atop T(\mathbf{x})=T}\sum_{g\in G'_\mathbf{x}(\mathbb{Q})_0\backslash G'(\mathbb{A}_f^p)/K^p}\varphi_K(g^{-1}\mathbf{x})\cdot\chi(\mathcal{N}, \mathcal{Z}(x_1) \otimes^\mathbb{L}\cdots \otimes^\mathbb{L}\mathcal{Z}(x_n))\cdot \log p.$$ By Theorem \ref{thm: main} and Remark \ref{sec:relation-with-local}, we have $$\chi(\mathcal{N}, \mathcal{Z}(x_1) \otimes^\mathbb{L}\cdots \otimes^\mathbb{L}\mathcal{Z}(x_n))\cdot \log p=c_p\cdot W_{T,p}'(1, 0,\varphi_{K,p})$$ for an explicit constant $c_p$ independent of $T$ and $\varphi_K$.

  Since $G'(\mathbb{Q})$ acts on $V'$ via its quotient $G'(\mathbb{Q})\rightarrow \SO(V')(\mathbb{Q})$, under which $G'(\mathbb{Q})_0$ surjects onto $\SO(V')(\mathbb{Q})$, we know that $G'(\mathbb{Q})\backslash V'^n=\SO(V')(\mathbb{Q})\backslash V'^n$. Thus $$\Int_{T,p}(\varphi_K)=c_p\cdot W_{T,p}'(1, 0,\varphi_{K,p})\cdot\sum_{\mathbf{x}\in \SO(V')(\mathbb{Q})\backslash V'^n \atop T(\mathbf{x})=T}\sum_{g\in G'_\mathbf{x}(\mathbb{Q})_0\backslash G'(\mathbb{A}_f^p)/K^p}\varphi_K(g^{-1}\mathbf{x}).$$ Since $T$ is nonsingular, we know that there exists a unique orbit in $\SO(V')(\mathbb{Q})\backslash V'^n$ with $T(\mathbf{x})=T$. Since $V'_\mathbf{x}$ is 1-dimensional, we know that $G'_\mathbf{x}$ is the center $Z'\simeq \mathbb{G}_m\subseteq G'$. Since $G'/Z'\simeq \SO(V')$ and the function $\varphi_K$ is invariant under the action of $G'_\mathbf{x}(\mathbb{Q})$ and  $Z'(\mathbb{A}_f^p)$, it follows that there exists a Haar measure $\rd g$ on $G'(\mathbb{A}_f^p)$ independent of $T$ and $\varphi_K$ such that $$\Int_{T,p}(\varphi_K)= c_p\cdot W'_{T,p}(1,0,\varphi_{K,p})\cdot\frac{1}{\vol(K^p)}\cdot \int_{\SO(V')(\mathbb{A}_f^p)}\varphi_K(g^{-1}\mathbf{x})\ \rd g.$$ By \cite[Proposition 1.2]{Kudla1997a}, we have $$\int_{\SO(V')(\mathbb{A}_f^p)}\varphi_K(g^{-1}\mathbf{x})\ \rd g=c^{p,\infty}\cdot\prod_{v\ne p,\infty}W_{T,v}(1,0,\varphi_{K,v}),$$  for a nonzero constant $c^{p,\infty}$ independent of $T$ and $\varphi_K$. Thus we have $$\Int_{T,p}(\varphi_K)=\frac{c_pc^{p,\infty}}{\vol(K^p)}\cdot W'_{T,p}(1,0,\varphi_{K,p})\cdot \prod_{v\ne p,\infty}W_{T,v}(1,0,\varphi_{K,v}).$$ Since $T>0$, by \cite[Proposition 4.3 (ii)]{Bruinier2018} we have $q^{T}=c_\infty\cdot W_{T,\infty}(\sz,0,\varphi_\infty)$, where $c_\infty$ is a nonzero constant independent of $T$. The result then follows from the factorization (\ref{eq:eisenfactor}) after scaling the Haar measure by $c_pc_\infty c^{p,\infty}$. Notice that the normalization factor $(-1)^n$ comes from the Weil constant $\gamma_\mathbb{V}^n=(-1)^n$ as $\mathbb{V}$ is incoherent (so that $\vol(K)>0$).
\end{proof}

\bibliographystyle{alpha}
\bibliography{KR}

\end{document}